\documentclass[11pt,a4paper]{article}

\usepackage[T1]{fontenc}
\usepackage[utf8]{inputenc}
\usepackage{lmodern}
\usepackage{microtype}
\usepackage[margin=28mm]{geometry}
\usepackage{amsmath,amssymb,amsthm,mathtools}
\usepackage{mathrsfs}
\usepackage{booktabs,longtable,array,float}
\usepackage{tikz}
\usepackage[hidelinks]{hyperref}
\hypersetup{pdftitle={Dimensional transition and spectral instability in lattice Hardy inequalities},
 pdfauthor={Carlos Lizama},pdfsubject={Defect sign law; exterior optimization at the eight--nine transition; finite resolvent parts and eigenvalue asymptotics},
 pdfkeywords={Hardy inequality, total critical defect, dimensional transition, discrete Laplacian, critical coupling, finite resolvent part}}

\numberwithin{equation}{section}

\newtheorem{theorem}{Theorem}[section]
\newtheorem{lemma}[theorem]{Lemma}
\newtheorem{proposition}[theorem]{Proposition}
\newtheorem{corollary}[theorem]{Corollary}
\theoremstyle{remark}
\newtheorem{remark}[theorem]{Remark}

\newcommand{\Z}{\mathbb Z}
\newcommand{\R}{\mathbb R}
\newcommand{\C}{\mathbb C}
\newcommand{\E}{\mathcal E}
\newcommand{\D}{\mathcal D}
\newcommand{\W}{\mathcal W}
\newcommand{\BesselK}{\mathrm K_0}

\newcommand{\supp}{\operatorname{supp}}
\newcommand{\ess}{\mathrm{ess}}
\newcommand{\spec}{\sigma}
\newcommand{\ip}[2]{\langle #1,#2\rangle}
\newcommand{\norm}[1]{\lVert #1\rVert}
\newcommand{\ind}{\mathbf 1}
\newcommand{\cP}{\mathcal P}
\newcommand{\cL}{\mathcal L}
\newcommand{\B}{\mathsf B}
\newcommand{\T}{\mathsf T}
\newcommand{\Rcal}{\mathcal R}
\newcommand{\Ac}{\mathcal A}
\newcommand{\Bc}{\mathcal B}

\title{Dimensional transition and spectral instability\\ in lattice Hardy inequalities}
\author{Carlos Lizama\\[4pt]
\small Departamento de Matem\'atica y Ciencia de la Computaci\'on\\
\small Universidad de Santiago de Chile, Santiago, Chile\\[4pt]
\small Corresponding author:
\href{mailto:carlos.lizama@usach.cl}{\nolinkurl{carlos.lizama@usach.cl}}\\[2pt]
\small ORCID: \href{https://orcid.org/0000-0002-9807-1100}{0000-0002-9807-1100}}
\date{}

\begin{document}
\maketitle

\begin{abstract}
Discretization can change a sharp stability threshold even when the
continuum threshold is recovered at infinity. We study this phenomenon
for the Hardy inequality on the integer lattice in dimensions at least
three: its best coefficient
agrees with the continuum value through dimension eight and is smaller
from dimension nine onwards. We characterize this transition by
the total critical defect, the summable error obtained by testing the
lattice operator against the critical continuum power. A theta-function
representation proves that the defect is negative and strictly
decreasing from dimension nine; explicit positive Hardy remainders
establish its positivity in dimensions three through eight. An exact energy
identity relates this defect to the energy saved by optimizing the
profile outside the neighbours of the origin. This exterior remains
stable in dimensions eight and nine, locating the first loss of
positivity in its coupling to those neighbours. We give self-contained
proofs of both dimensional ranges. We then describe how local
attraction creates a negative eigenvalue in a general class of lattice
operators with critical inverse-square tails. Its magnitude is
exponentially small near the threshold, with both the rate and the
multiplicative prefactor determined by zero-energy solutions. The
result allows connected finite modifications of the lattice and
nonradial corrections to the tail. The same analysis controls the
spread of the emerging state and gives a two-dimensional return scale
for the associated Markov chain in every ambient dimension.
\end{abstract}



\section*{Acknowledgments}
The author acknowledges support from the Agencia Nacional de
Investigaci\'on y Desarrollo (ANID), Chile, through the Fondo Nacional
de Desarrollo Cient\'ifico y Tecnol\'ogico (FONDECYT), grant No.~1260057.

\clearpage
\setcounter{tocdepth}{2}
\begingroup\small\tableofcontents\endgroup
\setcounter{tocdepth}{2}
\clearpage

\section{Introduction}\label{sec:intro}

Replacing derivatives by nearest-neighbour differences preserves many
large-scale features of a differential operator. At a sharp stability
threshold, however, the geometry near a singularity can change the
answer. The lattice Hardy inequality exhibits a precise dimensional
transition: the critical continuum coefficient remains valid through
dimension eight and fails from dimension nine onwards, although the
coefficient seen at infinity is unchanged in every dimension. We study
how this discrepancy is measured by an energy defect, how it is
modified by optimization, and how it leads to spectral instability.

Hardy inequalities compare the energy of spatial variation with the
attraction generated by an inverse-square potential.
For $N\ge3$, the classical Hardy inequality is
\[
 \int_{\R^N}|\nabla u|^2\,dx
 \ge A_N\int_{\R^N}\frac{|u|^2}{|x|^2}\,dx,
 \qquad A_N=\frac{(N-2)^2}{4},
 \qquad u\in C_c^\infty(\R^N\setminus\{0\}).
\]
Its sharp coefficient is the largest inverse-square attraction
compatible with nonnegative energy. The gradient term and the
potential have the same scaling, so neither disappears on changing
length scale. This balance underlies the role of Hardy inequalities
in spectral theory and equations with singular
potentials~\cite{Hardy,FLS,KapitanskiLaptev}.

Let $C(N)$ be the best coefficient when the gradient energy is replaced
by the sum of squared nearest-neighbour differences on $\Z^N$, with
the value at the origin fixed to zero. The classification
$C(N)=A_N$ for $3\le N\le8$ and $C(N)<A_N$ for $N\ge9$
has been established in recent preprints. Lizama~\cite{LizamaSharp}
proved the equalities in dimensions three and four and the strict
inequalities from dimension nine. Alpay~\cite{AlpayThree,AlpayLow}
proved the three-dimensional equality and the equalities in dimensions
five through eight. Our purpose is to relate this classification to
quantities that describe both the dimensional comparison and the
onset of instability, and to develop the resulting local threshold
theory.

The distinction between the coefficient at infinity and the global
coefficient has a direct consequence for approximation. On a lattice
of mesh size $h$, the critical Hardy operator is unitarily equivalent
to $h^{-2}$ times the unit-mesh operator. In dimensions three through
eight it remains nonnegative for every $h>0$. From dimension nine,
the bottom of its spectrum tends to $-\infty$ as $h\downarrow0$.
Thus local consistency with the continuum operator away from the
origin coexists with a loss of uniform spectral stability
(Corollary~\ref{cor:meshstability}). The issue is already present in
the interaction of the critical inverse-square potential with the
lattice near the singularity.

The first quantity we introduce tests this interaction against the
critical continuum power. Let $H_N$ be the Dirichlet lattice
Laplacian minus $A_N|x|^{-2}$. The power profile and its
\emph{total critical defect} are
\begin{equation}\label{eq:introprofiledefect}
 \begin{gathered}
 h_N(x)=|x|^{-(N-2)/2}\quad(x\ne0),\qquad h_N(0)=0,\\
 \mathfrak d_N=\sum_{x\in\Z^N\setminus\{0\}}
                h_N(x)(H_Nh_N)(x).
 \end{gathered}
\end{equation}
The two terms in the energy of the uncut profile diverge separately,
but their local difference is absolutely summable. This gives a
finite scalar quantity that can be compared across dimensions.
We obtain an integral representation through the Jacobi theta
function and prove the sign law
\begin{equation}\label{eq:introcharacterization}
 \begin{aligned}
 \mathfrak d_N>0&\quad\Longleftrightarrow\quad C(N)=A_N
                   \quad\Longleftrightarrow\quad3\le N\le8,\\
 \mathfrak d_N<0&\quad\Longleftrightarrow\quad C(N)<A_N
                   \quad\Longleftrightarrow\quad N\ge9.
 \end{aligned}
\end{equation}
together with strict decrease of $\mathfrak d_N$ along integer
dimensions from nine onwards. The negative defect provides explicit
violations of the Hardy inequality by logarithmic cutoffs. In the
sharp dimensions, positive remainders for arbitrary competitors
establish both the inequality and the positive defect sign.
The proofs of both ranges are included. The constructions in
dimensions three through eight share the reciprocal-field method
with~\cite{AlpayLow}; Section~\ref{sec:hardyproofs} specifies the
different fields, profiles and estimates used here.

The next step accounts for the adjustment of the profile.
The defect tests a prescribed function, while stability depends on
all competitors. We minimize the energy under a fixed mass on the
nearest neighbours of the origin. The resulting infimum is nonnegative
if and only if the lattice operator is nonnegative. Whenever the
exterior has the required positive energy structure, an exact identity
computes how much energy is saved by optimizing the profile there
(Theorem~\ref{thm:main}).

At the dimensional transition, the exterior obtained by removing the
first sphere remains subcritical in both dimensions eight and nine:
it has a finite zero-energy Green kernel and a positive Hardy
remainder. The change of sign is therefore located in the coupling
between this exterior and the neighbours of the origin. Dimension
eight permits additional attraction on these vertices, whereas
dimension nine requires repulsion to restore nonnegativity.
This reduction connects the dimensional transition to a local
perturbation problem with a continuously varying coupling strength.

The spectral results determine what happens when that strength crosses
its stability threshold. More generally, let $H$ be a nonnegative
subcritical lattice operator whose potential has the critical tail
$-A_N|x|^{-2}+O(|x|^{-2-\eta})$, $\eta>0$, and let $W\ge0$ be a
nonzero potential supported on finitely many vertices. For the family
$H-\tau W$, let $\tau_c$ be the largest coupling preserving
nonnegativity. Just above $\tau_c$, precisely one negative eigenvalue
appears. Its magnitude satisfies
\[
 \varepsilon(\tau)=\mathfrak P(H,W)
 \exp\!\left[-\frac{\mathfrak A(H,W)}{\tau-\tau_c}\right](1+o(1)),
 \qquad \mathfrak A(H,W),\mathfrak P(H,W)>0.
\]
We determine both constants from the positive threshold state and a
Poisson solution with a prescribed normalization at infinity
(Theorem~\ref{thm:robustintro}). The proof identifies the finite part
of the resolvent after its logarithmic singularity is subtracted.
This precision determines the prefactor, as well as the exponential
rate. The hypotheses allow connected finite modifications of the
lattice, nonradial tail corrections of either sign, and arbitrary
finitely supported nonnegative $W$.

The threshold also has a spatial and a probabilistic description.
The emerging state spreads over a length comparable to
$\varepsilon^{-1/2}$. Reweighting the jump rates by the positive
threshold state gives a Markov chain whose return probabilities have
the scale of dimension two in every ambient dimension. The basic
reason is that the square of the critical power changes radial volume
from $r^{N-1}\,dr$ to $r\,dr$; Section~\ref{sec:results} states the
precise estimates. These conclusions extend the stability analysis
from the sign of the energy to the size and spatial behaviour of
the state created at the threshold.

Hardy inequalities on graphs have also been studied by optimizing
the spatial weight itself; see~\cite{KPPhardy,KPPMonthly,KrejcirikStampach}
and the overview~\cite{Oberwolfach2025}.
Fischer~\cite[Theorems~2.3 and~2.5]{FischerHardy} develops
supersolution constructions for nonlinear graph energies and necessary
decay conditions for optimal weights. Keller--Lemm~\cite{KellerLemm}
identify the continuum inverse-square term as the leading term of an
optimal lattice weight and establish robustness under elliptic
coefficients. Power weights and improved inequalities on the
half-line are treated in~\cite{Dyda,DasFuente,Gupta1D,HuangYe}.
Gerhat--Krej\v{c}i\v{r}\'ik--\v{S}tampach~\cite{GKSRellich}
use factorization to improve the classical discrete Rellich weight.
Fractional lattice Hardy inequalities and supersolution constructions
are studied in~\cite{CiaurriRoncal,HKPFractional,HKPSupersolution}.
For the fixed inverse-square weight considered here, positive
remainders in the sharp dimensions are established
in~\cite{LizamaSharp,AlpayLow}. They express subcriticality in graph
criticality theory~\cite{PinchoverTintarev,KPP}. The explicit
remainders constructed below also quantify the first-sphere
stability margin.

The large-dimension asymptotics clarify why eventual failure is
expected. Gupta~\cite{GuptaLattice} proved order bounds;
Gupta~\cite{GuptaAsymptotic} and Huang--Ye~\cite{HuangYeAsymptotic}
establish $C(N)/N\to2$, and Lizama~\cite{LizamaAsymptotic} gives a
three-term expansion. The continuum coefficient grows quadratically.
Our defect comparison starts at the first strict dimension and gives
a strictly decreasing scalar quantity throughout that range.

Exponential eigenvalue laws have classical continuum antecedents.
Simon~\cite{Simon} developed the two-dimensional weak-coupling theory;
Kondej--Lotoreichik~\cite[Theorem~3.2]{KondejLotoreichik} obtain an
explicit positive prefactor for compactly supported potential measures
in the plane. Jia--Zhao~\cite[Theorem~3.8]{JiaZhao} study continuum
inverse-square backgrounds with strictly positive angular indices.
Threshold expansions and inversion methods are developed
in~\cite{JensenNenciu,JensenNenciuAddendum,WangConical,WangCorrigendum},
and Ito--Jensen~\cite{ItoJensen} study thresholds of the free lattice
Laplacian. For a class of convolution operators on the two-dimensional
lattice, Kholmatov--Lakaev--Almuratov
\cite[Theorem~1.5]{KholmatovLakaevAlmuratov} prove exponential
weak-coupling eigenvalue asymptotics with positive multiplicative
prefactors. Discrete ground-state and Green-kernel estimates appear
in~\cite{JexStampach,ChenHuaXu}. Our expansion determines the constant
term for critical inverse-square lattice tails in dimensions at least
three from normalized zero-energy data. For singular parabolic equations,
Baras--Goldstein~\cite[Theorem~2.2]{BarasGoldstein} show that
solutions for bounded potential truncations diverge at every positive
time above the continuum Hardy threshold, for every nonzero
nonnegative integrable initial datum. The mesh-refinement consequence
above concerns the stability of the discrete critical operators
themselves.

Section~\ref{sec:results} states the defect and stability results,
followed by the local threshold law and its consequences.
Section~\ref{sec:structure} develops the theta representation and the
dimensional sign comparison. Section~\ref{sec:hardyproofs} supplies
the positive Hardy remainders, and Section~\ref{sec:interface}
derives the exterior optimization identity. With this geometric and
variational description in place, Section~\ref{sec:criticallaws}
proves the general spectral theory. Section~\ref{sec:N8application}
returns to the eight--nine transition, describing eigenvalue creation
by attraction and extinction by repulsion. The appendices collect
the scalar estimates; Table~\ref{tab:N8proofguide} is a guide to the
eight-dimensional construction.

\section{Main results}\label{sec:results}

We first compare the critical continuum profile with the optimized
lattice energy. The defect gives a scalar dimensional comparison;
the optimized energy identifies the coupling at which stability is
lost. We then describe the eigenvalue and the spatial behaviour
created when a local attraction crosses such a threshold.

\subsection{Hardy estimates and the defect transition}

Let $X_N=\Z^N\setminus\{0\}$ and extend functions by zero at the
origin. Write $e_j$ for the coordinate unit vectors. The lattice
Dirichlet energy sums squared differences between neighbours; the
Hardy mass weights squared values by the inverse square of the
distance to the origin. We denote them by $\E_N$ and $\D_N$:
\begin{equation}\label{eq:energies}
 \E_N[u]=\sum_{x\in\Z^N}\sum_{j=1}^N|u(x+e_j)-u(x)|^2,
 \qquad \D_N[u]=\sum_{x\in X_N}\frac{|u(x)|^2}{|x|^2}.
\end{equation}
The best Hardy coefficient is the infimum of their quotient.
With $C_c$ denoting finite support, it is
\begin{equation}\label{eq:CNdef}
 C(N)=\inf_{0\ne u\in C_c(X_N;\C)}\frac{\E_N[u]}{\D_N[u]}.
\end{equation}
Each unoriented nearest-neighbour edge is counted once, including those
incident with the origin. Write $y\sim x$ when $x$ and $y$ are
nearest neighbours, and $M_f$ for multiplication by a function $f$.
The Dirichlet Laplacian $L_N^D$ retains the full diagonal at the
neighbours of the removed origin. Subtracting the critical potential
gives the Hardy operator $H_N$:
\[
 H_N=L_N^D-A_N M_{|x|^{-2}},\qquad
 (L_N^Du)(x)=2Nu(x)-\sum_{y\in X_N,\,y\sim x}u(y),
\]
on $\ell^2(X_N)$. To isolate the behaviour far from the origin,
restrict the Hardy quotient to functions supported outside a ball.
The coefficient at infinity, denoted by $C_\infty(N)$, is the limit
of these infima as the radius tends to infinity.

The critical power balances the gradient energy and the inverse-square
term in the continuum. We denote its decay exponent by $s_N$ and half
that exponent by $\nu_N$. Restricting this power to the lattice gives
the profile $h_N$:
\begin{equation}\label{eq:defectprofile}
 s_N=\frac{N-2}{2},\qquad \nu_N=\frac{N-2}{4},\qquad
 h_N(x)=|x|^{-s_N}\quad(x\ne0),\qquad h_N(0)=0.
\end{equation}
The local energy defect is the pointwise error left by this profile,
multiplied by the profile itself. We denote it by $\rho_N$ and its
sum, the total critical defect, by $\mathfrak d_N$:
\begin{equation}\label{eq:defectdefinition}
 \rho_N(x)=h_N(x)(H_Nh_N)(x),\qquad
 \mathfrak d_N=\sum_{x\in X_N}\rho_N(x).
\end{equation}
The local defects are summed before energy and mass are separated.
Their absolute convergence is proved in Section~\ref{sec:defect}; the
two quantities for the uncut profile separately diverge.

The sharp constants and the existence of positive spatial remainders
are established in
Lizama's preprint on the dimensional classification~\cite{LizamaSharp}
(Theorem~1.1 and Corollary~3.7) and in
Alpay's low-dimensional preprint~\cite{AlpayLow}
(Theorem~2.1 and Corollary~8.1).
Below we prove the strict inequalities in Section~\ref{sec:defect}
and the sharp equalities, including positive remainders, in
Section~\ref{sec:hardyproofs}. The proposition also records the
quantitative bounds we construct on the neighbours of the origin.
These $2N$ vertices form the \emph{first sphere}, denoted by
$\Sigma_N=\{\pm e_j:1\le j\le N\}$. A positive spatial remainder is
a strictly positive vertex weight whose weighted mass is bounded by
the critical Hardy energy; we denote that weight by $w_N$.
The constants $\beta_N$ below measure the resulting first-sphere
bound, and $c_*$ is the corresponding vertex bound in dimension eight.

\begin{proposition}[Sharp Hardy inequalities and positive remainders]
\label{prop:hardyremainder}
One has $C(N)=A_N$ exactly for $3\le N\le8$.
For each of these dimensions there is a signed-permutation invariant
weight $w_N(x)>0$ on $X_N$ such that
\begin{equation}\label{eq:hardyremainder}
 \langle u,H_Nu\rangle\ge\sum_{x\in X_N}w_N(x)|u(x)|^2
 \qquad(u\in C_c(X_N;\C)).
\end{equation}
The weights may be chosen to satisfy
\begin{equation}\label{eq:owntraceweights}
 w_N|_{\Sigma_N}=\frac{\beta_N}{2N},\qquad
 (\beta_3,\beta_4,\beta_5,\beta_6,\beta_7,\beta_8)
       =\left(9,15,21,21,10,\frac{16}{119}\right).
\end{equation}
In particular,
\begin{equation}\label{eq:ownmargin8}
 w_8|_{\Sigma_8}=c_*,\qquad c_*:=\frac1{119}>0.
\end{equation}
\end{proposition}
\begin{proof}
Theorem~\ref{thm:lowtrace} supplies the positive weights and the
first-sphere values for $3\le N\le7$. The completed reciprocal field
of Theorem~\ref{thm:N8} has remainder $W_8>0$ everywhere and
$W_8|_{\Sigma_8}>1/119$. Take $w_8=W_8$ off $\Sigma_8$ and
$w_8=1/119$ on $\Sigma_8$.
These constructions give $C(N)\ge A_N$ through dimension eight.
Proposition~\ref{prop:scaling} gives the reverse inequality, and
Theorem~\ref{thm:defectsign} supplies finitely supported negative
tests for $H_N$ when $N\ge9$.
\end{proof}

We next separate the first sphere from the rest of the lattice.
The exterior vertex set is $Y_N=X_N\setminus\Sigma_N$. Its operator
$K_N$ is the compression of $H_N$ to $Y_N$, with the original
diagonal, so the energy cost of edges crossing from the first sphere
to the exterior is retained. Boundary values $v$ act on the exterior
through the adjacency source $B_Nv$. We denote the unit constant
boundary pattern by $e_N$ and the source it produces by $b_N$:
\[
 (B_Nv)(x)=\sum_{a\in\Sigma_N,\,a\sim x}v(a),\qquad
 e_N=(2N)^{-1/2}\mathbf1_{\Sigma_N},\qquad b_N=B_Ne_N.
\]
The \emph{optimized defect} is the least energy when the first sphere
carries the same total mass as the critical profile. It measures the
energy remaining after the function is allowed to adjust. We denote
this infimum by $\mathfrak D_N$:
\[
 \mathfrak D_N=\inf\{\langle u,H_Nu\rangle:
 u\in C_c(X_N;\C),\ \|u|_{\Sigma_N}\|_2^2=2N\}.
\]
It may take the value $-\infty$.

When $K_N\ge0$, the solution of $(K_N+\varepsilon)v=b_N$
describes the exterior profile for the positive parameter $\varepsilon$.
Its scalar
response, denoted by $m_N(\varepsilon)$, and its zero-energy limit are
\[
 m_N(\varepsilon)=\langle b_N,(K_N+\varepsilon)^{-1}b_N\rangle,
 \qquad m_N(0)=\lim_{\varepsilon\downarrow0}m_N(\varepsilon).
\]
Subtracting this limit from the first-sphere diagonal gives the local
stability margin $\tau_N=2N-A_N-m_N(0)$ whenever the limit is finite.
The zero-energy Green kernel is the pointwise limit of the resolvent
kernels as the energy tends to zero. For the nonnegative connected
lattice operators considered here, finiteness of this kernel is
equivalent to subcriticality~\cite[Theorem~5.12]{KPP}.

The energy dual norm measures the strength of a source against tests
of unit energy. For a nonnegative injective operator $K$ on
$\ell^2(Y)$, it is
\[
 \|f\|_{K,*}^2=
 \sup_{0\ne u\in C_c(Y)}
 \frac{|\langle f,u\rangle|^2}{\langle u,Ku\rangle}.
\]
The restriction of the critical profile to the exterior is denoted by
$\phi_N$. Its residual, denoted by $r_N$, is the mismatch left when
that profile is inserted into the exterior equation. The squared dual
norm of this residual will be the energy saved by correcting the
profile. Write $P_E$ for multiplication by $\mathbf1_E$. 

\begin{theorem}[Defect and exterior transition]\label{thm:main}
For every integer $N\ge3$, $C_\infty(N)=A_N$ and
$\mathfrak d_N$ is finite. Its signs and dimensional comparison are
\begin{equation}\label{eq:introdefectsign}
 \mathfrak d_N>0\quad(3\le N\le8),\qquad
 \mathfrak d_N<0\quad(N\ge9),\qquad
 \mathfrak d_{N+1}<\mathfrak d_N\quad(N\ge9).
\end{equation}
For all $N\ge3$, one has $\mathfrak D_N\le\mathfrak d_N$ and
$\mathfrak D_N\ge0$ if and only if $H_N\ge0$.
For $3\le N\le9$, $K_N$ is nonnegative, injective, and has a finite
zero-energy Green kernel. The zero-energy response $m_N(0)$ is
finite. With $\phi_N=h_N|_{Y_N}$ and
$r_N=K_N\phi_N-\sqrt{2N}\,b_N$, interpreted in the energy dual,
\begin{equation}\label{eq:introoptimized}
 \mathfrak D_N=\mathfrak d_N-\|r_N\|_{K_N,*}^2
             =2N\tau_N.
\end{equation}
The sharp dimensions have $\tau_N>0$. At the transition,
\begin{equation}\label{eq:introinterfacechange}
 \tau_8\ge c_*>0,
 \qquad \tau_9<-\frac1{54}.
\end{equation}
In particular, for $t\ge0$,
\[
 H_8-tP_{\Sigma_8}\ge0\ \Longleftrightarrow\ t\le\tau_8,
 \qquad
 H_9+tP_{\Sigma_9}\ge0\ \Longleftrightarrow\ t\ge-\tau_9.
\]
\end{theorem}

Section~\ref{subsec:energytools} constructs the energy space and
proves the variational identity behind \eqref{eq:introoptimized}.
The theorem gives the sign relation
\begin{equation}\label{eq:transitionintro}
 \begin{aligned}
 \mathfrak d_N>0&\quad\Longleftrightarrow\quad C(N)=A_N
                  \quad\Longleftrightarrow\quad3\le N\le8,\\
 \mathfrak d_N<0&\quad\Longleftrightarrow\quad C(N)<A_N
                  \quad\Longleftrightarrow\quad N\ge9.
 \end{aligned}
\end{equation}
Its positive side uses Proposition~\ref{prop:hardyremainder}; its negative
side follows directly from the theta formula and logarithmic cutoffs.
The energy identity \eqref{eq:introoptimized} quantifies the difference
between testing the critical power and minimizing over the exterior.
The estimate for $K_9$ shows that the first loss of positivity occurs
in the coupling while the exterior stays subcritical.

The positive remainders also imply finite Green diagonals for
$3\le N\le8$. Writing $\delta_x$ for the unit point source at $x$,
the diagonal response is
\begin{equation}\label{eq:introsubcritical}
 G_N(x,x):=\lim_{\varepsilon\downarrow0}
 \langle\delta_x,(H_N+\varepsilon)^{-1}\delta_x\rangle<\infty.
\end{equation}
Thus a sharp global Hardy coefficient coexists with a positive threshold
for finite-support attraction. Proposition~\ref{prop:spectraltransition}
identifies the negative discrete spectrum for $N\ge9$;
Corollary~\ref{cor:meshstability} gives the corresponding loss of a
uniform lower spectral bound under mesh refinement.

\subsection{Loss of stability under local attraction}

The positive margins above leave room for an additional attraction
on a finite set. The next theorem determines the last admissible
strength and the negative eigenvalue that appears beyond it. We state
it for the full class of critical-tail operators to which the same
mechanism applies.

Fix
$N\ge3$, set $\alpha=(N-2)/2$, $A_N=\alpha^2$ and $\omega=|\mathbb S^{N-1}|$, and let
$X=\Z^N\setminus D$ with $D$ finite. Let $L_c$ be the Laplacian of a
connected graph whose conductances agree with the integer lattice outside
a finite core. Suppose
\[
 H=L_c+q\ge0,\qquad
 q(x)=-A_N|x|^{-2}+O(|x|^{-2-\eta}),\quad\eta>0,
\]
and the zero-energy Green kernel $G_0$ is finite at one vertex, hence
at every pair by Lemma~\ref{rob:lem:green}. The full assumptions are
stated in Section~\ref{sec:lowenergy}. Write $f\asymp g$ for bounds
by positive constant multiples in the indicated limit.

Let $W$ be a nonzero nonnegative potential of finite support.
Increasing $\tau$ in $H-\tau W$ strengthens the attraction on that
support. The critical coupling $\tau_c$ is the last value for which
the operator remains nonnegative. Just above it, $\varepsilon(\tau)$
measures how far the new eigenvalue lies below zero. The constants
$\mathfrak A(H,W)$ and $\mathfrak P(H,W)$ give its exponential rate
and multiplicative prefactor. Two zero-energy functions determine
them: a positive state at the threshold and its response to the
perturbation, obtained from a Poisson equation with a prescribed
logarithmic correction to the critical power.

\begin{theorem}[Critical coupling and eigenvalue asymptotics]\label{thm:robustintro}
For every nonzero nonnegative multiplication operator $W$ of finite
support, there is $\tau_c>0$ such that, among $\tau\ge0$,
$H-\tau W\ge0$ precisely for $0\le\tau\le\tau_c$. For
$0<\delta=\tau-\tau_c$ sufficiently small, the negative spectrum
consists of one simple eigenvalue $-\varepsilon(\tau)$ and
\[
 \varepsilon(\tau)=\mathfrak P(H,W)
 \exp\!\left[-\frac{\mathfrak A(H,W)}{\delta}\right](1+o(1)).
\]
At $\tau_c$ there is a unique positive pointwise zero-energy solution
$h_c$ up to a scalar, with
$h_c(x)=a_c|x|^{-\alpha}(1+O(|x|^{-\sigma}))$ for some
$a_c,\sigma>0$; zero is not an $\ell^2$ eigenvalue.
The $W$-weighted mass of the critical state is
$M_c=\sum_xW(x)h_c(x)^2$. Its logarithmic Poisson response $g_c$
is the unique solution fixed by the equation and normalization
\[
 (H-\tau_cW)g_c=Wh_c,\qquad
 g_c(x)=-\frac{M_c}{\omega a_c}|x|^{-\alpha}\log|x|
          +O(|x|^{-\alpha-\sigma}),
\]
after decreasing $\sigma$ if necessary. Then
\begin{align}
 \mathfrak A(H,W)&=\frac{2\omega a_c^2}{M_c},
       \label{eq:introexactrate}\\
 \mathfrak P(H,W)&=4e^{-2\gamma_{\mathrm E}}
   \exp\!\left(\frac{2\omega a_c^2}{M_c^2}\ip{Wh_c}{g_c}\right),
       \label{eq:introprefactor}
\end{align}
where $\gamma_{\mathrm E}$ is Euler's constant. Both constants are
positive and independent of the scalar normalization of $h_c$.

Let $\Psi_\tau>0$ denote the normalized eigenvector for
$-\varepsilon(\tau)$. The least radius
$\mathcal L_\theta$ containing a fixed fraction $\theta\in(0,1)$ of
its mass satisfies
\[
 \mathcal L_\theta(\tau)\asymp_\theta\varepsilon^{-1/2}.
\]
For every fixed nonempty finite set $E$,
\[
 \sum_{x\in E}|\Psi_\tau(x)|^2
 \asymp_E\varepsilon\log^2(1/\varepsilon).
\]
The constants may depend on the fixed operator and on $W$.
\end{theorem}

The normalization of $g_c$ is what makes the prefactor unambiguous:
its expansion has no constant multiple of $|x|^{-\alpha}$ after the
logarithmic term. Adding such a multiple of $h_c$ would change the
quantity in \eqref{eq:introprefactor}. Theorems~\ref{rob:thm:birth}
and~\ref{pref:thm:prefactor} give the proof. Finite changes of the
core can change $\tau_c$, $\mathfrak A$ and $\mathfrak P$; the formulas
hold for each resulting operator under the stated assumptions.
Proposition~\ref{rob:prop:nonempty} constructs admissible starting
operators for every such tail and connected finite modification by
adding a nonnegative potential of finite support.

The resolvent expansion underlying the theorem concerns the critical
operator $H_c=H-\tau_cW$. Normalize its state by $h=h_c/a_c$ and write
$L_\varepsilon=\log(1/\varepsilon)$ for the logarithmic energy scale.
For a finite source $f$, the canonical Poisson operator $\mathcal G$
selects the solution with the normalization
\[
 H_c\mathcal Gf=f,\qquad
 \mathcal Gf=-(\ip h f/\omega)|x|^{-\alpha}\log|x|
             +O_f(|x|^{-\alpha-\sigma}).
\]
For a fixed finite set $S$, also use $P_S$ for restriction and
extension by zero as appropriate. The restricted state and Poisson
matrix are $h_S=P_Sh$ and
$\mathcal G_S(x,y)=(\mathcal G\delta_y)(x)$.
Proposition~\ref{pref:prop:finitepart} proves the expansion
\begin{equation}\label{eq:introfinitepart}
 P_S(H_c+\varepsilon)^{-1}P_S
 =\frac{L_\varepsilon}{2\omega}h_Sh_S^*
   +\mathcal G_S+\frac{\log2-\gamma_{\mathrm E}}{\omega}h_Sh_S^*
   +o_S(1)
\end{equation}
in matrix norm. The constant term, including the Bessel constant,
yields \eqref{eq:introprefactor} by finite-dimensional
Birman--Schwinger perturbation. The vanishing remainder in
\eqref{eq:introfinitepart} fixes this constant term uniquely.

For the subcritical operator, the approach to zero energy is measured
by the \emph{Green defect}, the difference between $G_0$ and the
resolvent at spectral parameter $-\varepsilon<0$. The source amplitude $a(y)$ is the
coefficient of the common power tail in the response to a point source
at $y$. Theorem~\ref{rob:thm:exactresponse} proves that
$a(y)=\lim_{|x|\to\infty}|x|^\alpha G_0(x,y)>0$ exists and that
\begin{equation}\label{eq:introexactresponse}
 P_S[G_0-(H+\varepsilon)^{-1}]P_S
 =\frac{2\omega}{L_\varepsilon}a_Sa_S^*
       +O_S(L_\varepsilon^{-2}),\qquad a_S=(a(y))_{y\in S}.
\end{equation}
The matrix error controls energy-dependent sources on $S$.
Theorem~\ref{rob:thm:response} supplies the uniform spatial estimates
needed both for this expansion and for the mass bounds above. The
two expansions concern different operators: $G_0$ is finite for the
subcritical $H$, whereas the resolvent of $H_c$ diverges logarithmically.

\subsection{The effective two-dimensional return scale}

The threshold profile explains why the same scale appears for every
$N\ge3$ in the critical-tail class. Its square changes radial volume to
\begin{equation}\label{eq:introeffective}
 r^{N-1}\bigl(r^{-(N-2)/2}\bigr)^2\,dr=r\,dr.
\end{equation}
For radial continuum functions, the ground-state transform gives
\[
 r^{(N-2)/2}(-\Delta-A_Nr^{-2})
 \bigl(r^{-(N-2)/2}v(r)\bigr)=-v''(r)-r^{-1}v'(r).
\]
This is the radial Laplacian in dimension two. The modified Bessel
function $\BesselK$ and its logarithmic expansion~\cite{DLMF} therefore
provide the exterior matching profile. Related continuum state and
heat-kernel estimates appear in~\cite{IKO}; weak-coupling results for
more general kinetic energies, including discrete operators, appear
in~\cite{HoangHundertmarkRichterVugalter}. On the lattice, amplitude,
capacity and matching estimates are needed to turn this radial
calculation into the expansions above.

For the fractional Laplacian with a Hardy potential in Euclidean space,
Bogdan--Grzywny--Jakubowski--Pilarczyk~\cite[Theorem~1.1]{BogdanHardy}
prove two-sided heat-kernel bounds up to the sharp Hardy coefficient.
At that coefficient, as time $t$ tends to infinity, their estimate
yields diagonal decay of order $t^{-1}$ at each fixed nonzero point.

At the local threshold, the positive state turns the Schr\"odinger
operator into a continuous-time Markov chain by rescaling each jump
with the ratio of the state at its endpoints. This Doob
transform~\cite{Doob} has jump rates $c_{xy}h_c(y)/h_c(x)$ and
reversible measure $h_c(x)^2$. Denote the chain by $X_t$ and its law
started at $x$ by $\mathbb P_x^c$. Let $T_x$ be the first return
after departure, including the initial holding time. Null recurrence
means that return occurs almost surely but the mean return time is
infinite. Corollary~\ref{rob:cor:returns} proves null recurrence and,
at each fixed vertex,
\[
 \mathbb P_x^c(X_t=x)\asymp_x t^{-1},\qquad
 \mathbb P_x^c(T_x>t)\asymp_x(\log t)^{-1}.
\]
Resolvent and renewal estimates give these two-sided return bounds.
The transformed chain has a two-dimensional return scale, while the
free lattice walk is transient for $N\ge3$~\cite{Polya}.
These return estimates hold in every ambient dimension under the
critical-tail hypotheses. The dimensional Hardy transition determines
which unperturbed operators $H_N$ satisfy the required nonnegativity.

\subsection{Eigenvalue emergence in dimension eight}

Theorem~\ref{thm:main} allows us to apply the threshold asymptotics to
every $H_N$ with $3\le N\le8$. Attraction on the first sphere in
dimension eight relates the constants in these asymptotics to the
transition margin. We denote this perturbed operator by
$H_8^{\mathrm{loc}}(\tau)$, where $\tau$ is the attraction strength:
\begin{equation}\label{eq:introlocalfamily}
 \Sigma=\{\pm e_j:1\le j\le8\},\qquad
 H_8^{\mathrm{loc}}(\tau)=H_8-\tau M_{\mathbf1_\Sigma},\quad\tau\ge0.
\end{equation}

\begin{corollary}\label{thm:birthintro}
There is $c_*\le\tau_*<7$ such that $H_8^{\mathrm{loc}}(\tau)$ has no
negative spectrum for $0\le\tau\le\tau_*$. At $\tau_*$ it has a
positive pointwise zero-energy state comparable to $|x|^{-3}$, and
zero is not an $\ell^2$ eigenvalue. For $\tau>\tau_*$ sufficiently close to $\tau_*$, there is
exactly one negative eigenvalue, which is simple. In that limit,
\begin{equation}\label{eq:introbirthlaw}
 \varepsilon(\tau)=\mathfrak P_8
 \exp\!\left[-\frac{\mathfrak A_8}{\tau-\tau_*}\right](1+o(1)),
\end{equation}
where $\mathfrak A_8=\pi^4a_8^2/24$ and $a_8$ is the amplitude
at infinity of the threshold state normalized to equal one on $\Sigma$.
The positive prefactor $\mathfrak P_8$ is given in
\eqref{pref:eq:N8prefactor}.
The normalized bound state satisfies
\begin{equation}\label{eq:introdelocalization}
 \log\mathcal L_\theta(\tau)=
 \frac{\mathfrak A_8}{2(\tau-\tau_*)}+O_\theta(1),\qquad
 \|\mathbf1_\Sigma\Psi_\tau\|_2^2
       \asymp\frac{\varepsilon(\tau)}{(\tau-\tau_*)^2}.
\end{equation}
\end{corollary}

The exterior response links the explicit stability bound to this
eigenvalue law. If $K$ is the exterior compression and $b$ is the adjacency
source of the normalized constant interface vector, put
$m(\varepsilon)=\langle b,(K+\varepsilon)^{-1}b\rangle$.
Then
\[
 \tau_*=7-m(0),\qquad
 \tau-\tau_*=\varepsilon+m(0)-m(\varepsilon).
\]
The response inequality $m(0)<7$ gives unperturbed stability; the
low-energy response gives the emerging eigenvalue.
Proposition~\ref{prop:birthschur} identifies the threshold and the
interval determined by the remaining interface spectrum. The constants in
\eqref{eq:introbirthlaw} are determined by the amplitude at infinity
and the normalized Poisson solution. Proposition~\ref{prop:defectschur}
also relates the local stability margin to the total critical defect
and the energy of the exterior adjustment.

The corresponding phenomenon in dimension nine is the disappearance
of a bound state under first-sphere repulsion. With
$t_9=-\tau_9>1/54$, the family $H_9+tP_{\Sigma_9}$ is critical at
$t=t_9$. As $t\uparrow t_9$, its remaining negative eigenvalue obeys
an exponential law in $(t_9-t)^{-1}$, with both constants determined by
the same theorem; see Corollary~\ref{cor:extinction9}.

\section{The total critical defect}\label{sec:structure}

The coefficient at infinity is fixed by continuum scaling. The total
defect then measures the finite limiting energy of the critical power.
We derive its theta representation, prove its negative sign and strict
decrease from dimension nine, and use the positive remainders proved below
for the complementary dimensions.

In Lizama's preprint~\cite{LizamaSharp}
(Propositions~4.8 and~5.4), the strict inequalities follow from a
Gaussian Rayleigh--Ritz construction in dimension nine and
finite-dimensional trial spaces spanned by indicators of symmetry
classes in higher dimensions. Here the total critical defect gives a single proof for
all $N\ge9$: its theta representation supplies the estimate in
dimension nine and the comparison between consecutive dimensions,
and logarithmic cutoffs turn its negative sign into test functions.

\subsection{The coefficient at infinity}

The classical continuum Hardy inequality~\cite{Hardy,FLS} is
\begin{equation}\label{eq:continuumHardy}
 \int_{\R^N}|\nabla u|^2\,dx
 \ge A_N\int_{\R^N}\frac{|u|^2}{|x|^2}\,dx,
 \qquad A_N=\frac{(N-2)^2}{4},\qquad N\ge3,
\end{equation}
for $u\in C_c^\infty(\R^N\setminus\{0\})$, with best coefficient $A_N$.
Two complementary lattice estimates identify the same coefficient
at infinity: scaled continuum tests give
the upper bound, and a reciprocal field gives the exterior lower bound.
The same edge identity will be used for the exterior in dimension nine.

\subsubsection{The reciprocal-edge identity}

The Dirichlet reciprocal-edge decomposition is recorded in
Alpay's low-dimensional preprint~\cite{AlpayLow}
(Proposition~3.3). It retains the full lattice degree
at neighbours of the removed origin. A \emph{reciprocal edge field}
assigns a positive number to each oriented edge, with the reciprocal
number assigned to the reverse orientation.

Let $N\ge3$. On every oriented edge $x\sim y$ in $X_N$, assume
\begin{equation}\label{eq:reciprocity}
 a_{xy}>0,\qquad a_{yx}=a_{xy}^{-1},
\end{equation}
The diagonal left after these edge contributions is the vertex
weight, denoted by $\W_a^{(N)}$:
\begin{equation}\label{eq:Wa}
 \W_a^{(N)}(x)=2N-\sum_{y\in X_N,\,y\sim x}a_{xy}.
\end{equation}
The full diagonal in this definition is essential at neighbours of the
removed origin.

\begin{lemma}\label{lem:edge}
For every field satisfying \eqref{eq:reciprocity} and every
$u\in C_c(X_N;\C)$,
\begin{equation}\label{eq:edge}
 \E_N[u]=
 \sum_{\substack{\{x,y\}\subset X_N\\x\sim y}}
 |\sqrt{a_{xy}}u(x)-a_{xy}^{-1/2}u(y)|^2
 +\sum_{x\in X_N}\W_a^{(N)}(x)|u(x)|^2.
\end{equation}
One orientation is chosen for each unoriented edge in the first sum.
\end{lemma}

\begin{proof}
For $a>0$ and $z,w\in\C$,
\[
 |\sqrt a\,z-a^{-1/2}w|^2
 =a|z|^2+a^{-1}|w|^2-2\operatorname{Re}(z\overline w).
\]
By reciprocity, the edge squares contribute
$\sum_{y\in X_N,\,y\sim x}a_{xy}$ to the coefficient of $|u(x)|^2$,
regardless of their orientations. Adding \eqref{eq:Wa} restores $2N$.
The mixed terms are exactly those of the nearest-neighbour energy.
If $x$ is adjacent to the origin, the omitted edge contributes
$|u(x)|^2$ in \eqref{eq:energies}, since $u(0)=0$.
Thus the same diagonal $2N$ is correct there as well. This proves
\eqref{eq:edge}.
\end{proof}

In particular, if $\rho(x):=\W_a^{(N)}(x)-A_N/|x|^2>0$ on $X_N$, then
\[
 \E_N[u]-A_N\D_N[u]\ge\sum_{x\in X_N}\rho(x)|u(x)|^2>0
 \qquad(0\ne u\in C_c(X_N;\C)).
\]

\subsubsection{The continuum upper bound}

The continuum upper bound is proved in
Lizama's preprint on the dimensional classification~\cite{LizamaSharp}
(Proposition~2.2) and in
Alpay's low-dimensional preprint~\cite{AlpayLow}
(Proposition~2.5).
We recall the scaling argument, which also gives competitors
supported outside any prescribed ball.

\begin{proposition}\label{prop:scaling}
For every integer $N\ge3$, $C(N)\le A_N$. For every $R_0>0$, the same upper
bound holds for the infimum restricted to competitors supported in
$\{|x|>R_0\}$.
\end{proposition}

\begin{proof}
Let $\alpha=(N-2)/2$ and choose a real-valued
$0\ne\eta\in C_c^\infty((0,1))$. For $T>0$, set
\[
 f_T(x)=|x|^{-\alpha}\eta(\log|x|/T).
\]
Its extension by zero belongs to
$C_c^\infty(\R^N\setminus\{0\})$. With $t=\log r/T$,
\[
 \frac{d}{dr}\big[r^{-\alpha}\eta(t)\big]
 =r^{-\alpha-1}\big[-\alpha\eta(t)+T^{-1}\eta'(t)\big].
\]
Since $2\alpha=N-2$, radial integration uses the measure $dr/r=T\,dt$.
The cross term vanishes because $\int_0^1\eta\eta'=0$. Hence
\begin{equation}\label{eq:continuousratio}
 \frac{\int_{\R^N}|\nabla f_T|^2}
      {\int_{\R^N}|f_T|^2/|x|^2}
 =A_N+\frac1{T^2}\frac{\int_0^1|\eta'|^2}{\int_0^1|\eta|^2}.
\end{equation}

Fix $T$ and put $u_R(n)=f_T(n/R)$ for positive integers $R$.
These functions have finite support and vanish at zero. The normalized
differences
\[
 R\big[f_T(x+e_j/R)-f_T(x)\big]
 =\int_0^1\partial_jf_T(x+\theta e_j/R)\,d\theta
\]
converge uniformly to $\partial_jf_T(x)$ on one fixed compact set containing
all relevant supports. Their squares are uniformly bounded, and the number
of lattice points in that compact set is $O(R^N)$. Consequently ordinary
Riemann sums give
\begin{align}
 R^{2-N}\E_N[u_R]&\longrightarrow\int_{\R^N}|\nabla f_T|^2,
 \label{eq:scalingenergy}\\
 R^{2-N}\D_N[u_R]&\longrightarrow
              \int_{\R^N}|f_T(x)|^2/|x|^2\,dx>0.
 \label{eq:scalingmass}
\end{align}
For the second limit, $|f_T|^2/|x|^2$ is continuous and compactly supported
away from zero. Dividing the two limits, then using
\eqref{eq:continuousratio} and letting $T\to\infty$, proves $C(N)\le A_N$.
For fixed $T$, the inner radius of $\operatorname{supp}u_R$ tends to infinity.
Thus the same competitors prove the assertion with any fixed excluded ball.
\end{proof}

\subsubsection{The continuum coefficient at infinity}

For $R>0$, define
\[
 C_R^\infty(N)=
 \inf_{\substack{0\ne u\in C_c(X_N;\C)\\
                  \operatorname{supp}u\subset\{|x|>R\}}}
       \frac{\E_N[u]}{\D_N[u]}.
\]
These quantities increase with $R$ and are bounded above by $A_N$,
by Proposition~\ref{prop:scaling}.

The coefficient $A_N$ also occurs in the large-distance asymptotics
of lattice Hardy weights; see \cite{KPPhardy} and
\cite[Section~3.1]{KellerLemm}.
The following proposition gives its exterior characterization for
the prescribed inverse-square weight.

\begin{proposition}\label{prop:infinity}
For every integer $N\ge3$,
\begin{equation}\label{eq:infinitythreshold}
 C_\infty(N):=\lim_{R\to\infty}C_R^\infty(N)=A_N.
\end{equation}
\end{proposition}

\begin{proof}
Only the lower bound remains. Put $\alpha=(N-2)/2$ and
$h(x)=|x|^{-\alpha}$ on $\R^N\setminus\{0\}$.
For $|x|\ge2$, all line segments $x+t e_j$, $|t|\le1$, stay away from zero.
The fourth derivatives of $h$ on these segments are bounded by
$c_N|x|^{-\alpha-4}$, because of homogeneity and the bound
$|x+t e_j|\ge|x|/2$. Taylor's formula with integral remainder, applied to
the two opposite moves, therefore gives
\[
 h(x+e_j)+h(x-e_j)-2h(x)
 =\partial_j^2h(x)+O_N(|x|^{-\alpha-4}),
\]
uniformly in the direction of $x$. Also
$-\Delta_{\R^N}h=A_N|x|^{-\alpha-2}$. Summing over $j$ yields
\begin{equation}\label{eq:infinityfield}
 \frac{L_Nh(x)}{h(x)}
 =\frac{A_N}{|x|^2}+O_N(|x|^{-4}).
\end{equation}
Assign the reciprocal field $a_{xy}=h(y)/h(x)$ on $X_N$.
At $|x|>R\ge2$ its vertex weight is the left-hand side of
\eqref{eq:infinityfield}. Lemma~\ref{lem:edge} consequently gives, for
competitors supported there,
\[
 \E_N[u]\ge A_N\D_N[u]
       -c_N\sum_x\frac{|u(x)|^2}{|x|^4}
 \ge\left(A_N-\frac{c_N}{R^2}\right)\D_N[u].
\]
The limit proves \eqref{eq:infinitythreshold}.
\end{proof}

For every fixed $\delta>0$, competitors with Hardy quotient at most
$A_N-\delta$ cannot have supports escaping every fixed ball.
A fixed improvement therefore uses finite spatial scales, while its
optimizing profile may still extend throughout the exterior. The estimate
\eqref{eq:infinityfield} identifies the threshold at infinity; the global
classification is proved in Proposition~\ref{prop:hardyremainder}.

\subsection{The theta formula and the strict dimensions}\label{sec:defect}

The total defect measures the discrepancy between the discrete operator
and the continuum-critical power. Recall $h_N,s_N,\nu_N$ from
\eqref{eq:defectprofile} and $\rho_N,\mathfrak d_N$ from
\eqref{eq:defectdefinition}. The local cancellation makes the sum finite
and lets us compare dimensions without first optimizing an exterior
profile.
By \eqref{eq:infinityfield},
$\rho_N(x)=O_N(|x|^{-N-2})$, so the series in
\eqref{eq:defectdefinition} is absolutely convergent.

A Gaussian-weighted count of one-dimensional lattice points encodes
the lattice sum. Its $N$th power accounts for the $N$ coordinates.
Denote this theta function by
\[
 \Theta(t)=\sum_{k\in\Z}e^{-tk^2},\qquad t>0.
\]
Let $V_N$ have the beta distribution
$\operatorname{Beta}(\nu_N,\nu_N)$. The beta average represents the
interaction between neighbouring critical powers. Its edge kernel is
denoted by $\mathcal B_N$, and $\mathcal J_N$ collects these
contributions over the lattice:
\begin{align}
 \mathcal B_N(t)
 &=\mathbb E\big[(1-e^{-tV_N})(1-e^{-t(1-V_N)})\big],
 \label{eq:defectB}\\
 \mathcal J_N(t)
 &=2N\bigg[e^{-t}+(\Theta(t)^{N-1}-1)\mathcal B_N(t)\notag\\
 &\hspace{24mm}
 +\Theta(t)^{N-1}\sum_{k\ge1}e^{-tk^2}
       \mathcal B_N((2k+1)t)\bigg].\label{eq:defectJ}
\end{align}
The Mellin--theta calculation in Appendix~\ref{app:defecttheta} gives the
following exact representation.

\begin{proposition}\label{prop:defecttheta}
For every integer $N\ge3$,
\begin{equation}\label{eq:defecttheta}
 \mathfrak d_N=
 \frac{(N+2)\pi^{N/2}}{N\Gamma(s_N)}
 +\frac1{\Gamma(s_N)}\int_0^\infty t^{s_N-1}
 \big[\mathcal J_N(t)-s_Nt(\Theta(t)^N-1)\big]dt.
\end{equation}
The integral converges as a combined expression at both endpoints.
\end{proposition}

The representation makes the strict side of the dimensional transition
accessible by comparing the edge kernel with an elementary upper
bound. Denote the difference by $\Xi_N$ and the coefficient in
that bound by $\eta_N$:
\begin{equation}\label{eq:defectetaXi}
 \eta_N=\frac{N-2}{4N},\qquad
 \Xi_N(t)=\eta_Nt(1-e^{-t})-\mathcal B_N(t).
\end{equation}

The positive gap also gives a comparison between dimensions. This is
stronger than proving the sign separately in each dimension.

\begin{theorem}\label{mono:thm:main}
The total critical defect is strictly decreasing in integer dimensions
from nine onwards:
\begin{equation}\label{mono:eq:main}
 \mathfrak d_{N+1}<\mathfrak d_N\qquad(N\ge9).
\end{equation}
\end{theorem}

The comparison of the beta laws is particularly simple after a
normalization. Define
\begin{equation}\label{mono:eq:Xiinf}
 \Xi_\infty(t)=\frac t4(1-e^{-t})-(1-e^{-t/2})^2.
\end{equation}

\begin{lemma}\label{mono:lem:beta}
For $t>0$, the function $\Xi_N(t)/(1-2/N)$ is strictly increasing
in the real parameter $N>2$ and is bounded above by $\Xi_\infty(t)$.
In particular $\Xi_N(t)>0$ and
\begin{equation}\label{mono:eq:betaorder}
 0<\Xi_N(t)\le\frac{N-2}{N}\Xi_\infty(t),\qquad
 \Xi_{N+1}(t)\ge
 \frac{N(N-1)}{(N+1)(N-2)}\Xi_N(t).
\end{equation}
\end{lemma}
\begin{proof}
We use the rising factorial, the product of consecutive factors,
written $(a)_m=a(a+1)\cdots(a+m-1)$ for $m\ge1$, with $(a)_0=1$.
Set $Z_N=2V_N-1$. Its even moments are
\begin{equation}\label{eq:defectmoments}
 \mathbb E Z_N^{2m}=\frac{(1/2)_m}{(N/4)_m},\qquad m\ge0.
\end{equation}
Expanding the hyperbolic functions in
$\mathcal B_N(t)=2e^{-t/2}(\cosh(t/2)-\mathbb E\cosh(tZ_N/2))$ gives
\begin{equation}\label{eq:defectXiseries}
 \Xi_N(t)=2e^{-t/2}\sum_{m\ge2}
 \frac{\kappa_{m,N}}{(2m)!}(t/2)^{2m},
\end{equation}
where
\begin{equation}\label{eq:defectkappa}
 \kappa_{m,N}=m(1-2/N)-1+\frac{(1/2)_m}{(N/4)_m}.
\end{equation}
The distribution of $Z_N^2$ is
$\operatorname{Beta}(1/2,(N-2)/4)$. Applying
$1-y^m=(1-y)\sum_{j=0}^{m-1}y^j$ and tilting this beta law by
$1-y$ yields
\begin{equation}\label{mono:eq:kappa}
 \frac{\kappa_{m,N}}{1-2/N}
 =\sum_{j=1}^{m-1}
 \left[1-\frac{(1/2)_j}{(N/4+1)_j}\right].
\end{equation}
Every summand is positive and strictly increases with $N$; their sum
is less than $m-1$. Substitution in the positive series proves the
claims, since replacing its coefficients by $m-1$ gives
\eqref{mono:eq:Xiinf}.
\end{proof}

The elementary identity
\begin{equation}\label{eq:defecttelescoping}
 \sum_{k\ge0}(2k+1)
 \big(e^{-k^2t}-e^{-(k+1)^2t}\big)=\Theta(t)
\end{equation}
reorganizes the critical cancellation. Substituting
$\mathcal B_N(u)=\eta_Nu(1-e^{-u})-\Xi_N(u)$ in
\eqref{eq:defectJ} and using $2N\eta_N=s_N$ gives
\begin{align}
 \frac{\mathcal J_N(t)-s_Nt(\Theta(t)^N-1)}{2N}
 ={}&e^{-t}(1+\eta_Nt)
 -(\Theta(t)^{N-1}-1)\Xi_N(t)\notag\\
 &-\Theta(t)^{N-1}\sum_{k\ge1}e^{-k^2t}
       \Xi_N((2k+1)t).\label{eq:defectmaster}
\end{align}
Write $s=s_N$ and $n=N-1$ throughout this subsection. Integration yields
\begin{align}
 \mathfrak d_N&=\mathfrak c_N+\frac{(N+2)^2}{4}-S_N,
 \label{mono:eq:defect}\\
 \mathfrak c_N&=\frac{(N+2)\pi^{N/2}}{N\Gamma(s)},
 \label{eq:defectcN}\\
 S_N&=\frac{2N}{\Gamma(s)}\int_0^\infty t^{s-1}
 \bigg[(\Theta(t)^n-1)\Xi_N(t)\notag\\
 &\hspace{34mm}+\Theta(t)^n\sum_{k\ge1}e^{-k^2t}
               \Xi_N((2k+1)t)\bigg]dt.\label{mono:eq:S}
\end{align}
The convergence follows from Proposition~\ref{prop:defecttheta} and
\eqref{eq:defectmaster}; all terms in $S_N$ are nonnegative.

The useful split extracts the two contributions with the smallest
squared radius. Their Mellin integrals are explicit, while the remaining
positive integral admits a comparison between consecutive dimensions.
This separates the elementary dimensional decrease from the possible
loss in the remainder.

For $\lambda,b>0$, beta integration gives the exact elementary formula
\begin{align}
 I_N(\lambda,b)
 &:=\frac{1}{\Gamma(s)}\int_0^\infty t^{s-1}e^{-\lambda t}\Xi_N(bt)\,dt
 \notag\\
 &=\eta_Nbs\big[\lambda^{-s-1}-(\lambda+b)^{-s-1}\big]
   -\lambda^{-s}-(\lambda+b)^{-s}
   +2[\lambda(\lambda+b)]^{-s/2}.\label{mono:eq:I}
\end{align}
Here the mixed term uses
$\mathbb E(\lambda+bV_N)^{-s}=[\lambda(\lambda+b)]^{-s/2}$.
Extract from \eqref{mono:eq:S} the terms
$2n e^{-t}\Xi_N(t)$ and $e^{-t}\Xi_N(3t)$. This gives
\begin{equation}\label{mono:eq:QR}
 \mathfrak d_N=Q_N-R_N,\qquad
 Q_N=\mathfrak c_N+\frac{(N+2)^2}{4}
     -2N\big[2nI_N(1,1)+I_N(1,3)\big],
\end{equation}
where the positive remainder is
\begin{equation}\begin{aligned}
 R_N=\frac{2N}{\Gamma(s)}\int_0^\infty t^{s-1}
 \bigg[&\big(\Theta^n-1-2n e^{-t}\big)\Xi_N(t)
    +(\Theta^n-1)e^{-t}\Xi_N(3t)\\
   &+\Theta^n\sum_{k\ge2}e^{-k^2t}\Xi_N((2k+1)t)\bigg]dt.
\end{aligned}\label{mono:eq:R}\end{equation}
All three terms are nonnegative, because
$\Theta\ge1+2e^{-t}$.

The explicit part can now be written without integrals. Set
\begin{align}
 z_N&=2^{-(N-2)/4},\qquad
 P(N)=-\frac{N^3}{2}+6N^2-2N,\notag\\
 B(N)&=\frac{N^3+11N^2-24N-4}{4},\qquad
 D(N)=\frac{3N^2+20N+12}{16},\notag\\
 E(N)&=-8N(N-1)z_N+B(N)z_N^2+D(N)z_N^4.
 \label{mono:eq:E}
\end{align}
Direct substitution of \eqref{mono:eq:I} into \eqref{mono:eq:QR} gives
\begin{equation}\label{mono:eq:Qexplicit}
 Q_N=\mathfrak c_N+P(N)+E(N).
\end{equation}

The following lower comparison controls how much the remainder can
decrease between consecutive dimensions.

\begin{lemma}\label{mono:lem:comparison}
For integer $N\ge9$,
\begin{equation}\label{mono:eq:alpha}
 R_{N+1}\ge a_NR_N,\qquad
 a_N=\frac{N-1}{N-2}\sqrt\pi\,
              \frac{\Gamma((N-2)/2)}{\Gamma((N-1)/2)}.
\end{equation}
Moreover $a_9=5\pi/14>1$, $a_{10}=36/35>1$, and, for $N\ge11$,
\begin{equation}\label{mono:eq:alphalower}
 a_N>\underline a_N:=\frac{5(N-1)}{2(N-2)^{3/2}},\qquad
 0<\underline a_N<1.
\end{equation}
\end{lemma}
\begin{proof}
Write $q=e^{-t}$ and $A_n=\Theta^n-1-2nq$. Then
\[
 A_{n+1}-\Theta A_n=\Theta-1-2q+2nq(\Theta-1)\ge0.
\]
Also $(\Theta^{n+1}-1)\ge\Theta(\Theta^n-1)$.
Thus each theta factor in \eqref{mono:eq:R} gains at least $\Theta$ when
$N$ increases by one. Combine \eqref{mono:eq:betaorder} with the extra
Mellin factor $\sqrt t$, and use Jacobi's identity
$\sqrt t\,\Theta(t)=\sqrt\pi\,\Theta(\pi^2/t)\ge\sqrt\pi$.
This proves \eqref{mono:eq:alpha}. Cauchy--Schwarz in the gamma integral gives
$\Gamma(s+1/2)/\Gamma(s)\le\sqrt s$. Since $\pi>25/8$, this implies
\eqref{mono:eq:alphalower}; its upper bound follows from the decreasing
function $(x+1)x^{-3/2}$ and $\underline a_{11}=25/27$.
\end{proof}

\begin{lemma}\label{mono:lem:550}
For every integer $N\ge9$,
\begin{equation}\label{mono:eq:550}
 0<R_N<550.
\end{equation}
\end{lemma}
\begin{proof}
The uniform estimate is proved in Appendix~\ref{app:defectremainder}.
It uses the upper kernel in Lemma~\ref{mono:lem:beta}, the Jacobi
identity for $0<t\le3$, and geometric theta bounds for $t\ge3$.
\end{proof}

\begin{proof}[Proof of Theorem~\ref{mono:thm:main}]
First, $\mathfrak c_{N+1}<\mathfrak c_N$ for all integers $N\ge9$. At the first two steps,
\[
 \frac{\mathfrak c_{10}}{\mathfrak c_9}=\frac{27\pi}{88}<1,
 \qquad \frac{\mathfrak c_{11}}{\mathfrak c_{10}}=\frac{208}{231}<1.
\]
For real $N\ge11$, logarithmic differentiation gives
\[
 \frac{d}{dN}\log \mathfrak c_N=-\frac{2}{N(N+2)}
                 +\frac12[\log\pi-\psi((N-2)/2)]<0.
\]
For completeness, if $T$ has a gamma distribution of shape $s>1$,
Jensen's inequality gives
$\psi(s)=\mathbb E\log T>-\log\mathbb E(T^{-1})=\log(s-1)$.
Here $s\ge9/2$ and $s-1\ge7/2>\pi$.

We next bound the variation of $E(N)$ in \eqref{mono:eq:E}.
The two positive terms $B(N)z_N^2$ and $D(N)z_N^4$ decrease for
real $N\ge9$. Indeed
\[
 \frac{B'(N)}{B(N)}<\frac{3}{N}\le\frac13<\frac{\log2}{2},
 \qquad \frac{D'(N)}{D(N)}<\frac{2}{N}<\log2.
\]
The first inequality follows from
$11N^2-48N-12>0$; the second follows from the positive coefficients
of $D$. Put $\alpha=2^{-1/4}>5/6$. It follows that
\begin{equation}\label{mono:eq:Edifference}
 E(N+1)-E(N)
 \le8Nz_N[(N-1)-\alpha(N+1)].
\end{equation}
This is negative for $N=9,10$. For $N\ge11$, use
$\log2>2/3$ and write $x=N-11\ge0$. Then
\[
 \frac{E(N+1)-E(N)}{N}
 \le \frac{4}{3}(N-11)e^{-(N-2)/6}
 =\frac{4}{3}e^{-3/2}x e^{-x/6}
 \le8e^{-5/2}<1.
\]
Since $P(N+1)-P(N)=-(3N^2-21N-7)/2$, this proves
\begin{equation}\label{mono:eq:Qdifference}
 Q_{N+1}-Q_N<-p(N),\qquad
 p(N)=\frac{3N^2-23N-7}{2},\qquad N\ge11.
\end{equation}

For $N=9,10$, \eqref{mono:eq:Edifference}, the decrease of $\mathfrak c_N$, and
the exact difference of $P$ give
$Q_{10}-Q_9<-47/2$ and $Q_{11}-Q_{10}<-83/2$.
By Lemma~\ref{mono:lem:comparison}, $R_{10}>R_9$ and $R_{11}>R_{10}$.
Equation~\eqref{mono:eq:QR} proves \eqref{mono:eq:main} for these two steps.

For $N\ge11$, combine Lemmas~\ref{mono:lem:comparison} and~\ref{mono:lem:550}
with \eqref{mono:eq:Qdifference}. Since $\underline a_N<1$,
\begin{align*}
 \mathfrak d_{N+1}-\mathfrak d_N
 &< -p(N)+(1-\underline a_N)R_N\\
 &< -p(N)+550(1-\underline a_N)=-g(N),
\end{align*}
where
\[
 g(x)=\frac{3x^2-23x-7}{2}-550
              +1375\,\frac{x-1}{(x-2)^{3/2}}.
\]
There is a simple uniform lower bound for this elementary function:
\[
 g''(x)=3+\frac{4125}{4}\frac{x+3}{(x-2)^{7/2}}>3\qquad(x>2).
\]
Moreover
\[
 g(12)=-\frac{951}{2}+\frac{3025}{2\sqrt{10}}>\frac{27}{10},
 \qquad -\frac{19}{5}<g'(12)
       =\frac{49}{2}-\frac{715}{8\sqrt{10}}<0.
\]
The rational bounds $79/25<\sqrt{10}<253/80$ suffice for both
inequalities. Strong convexity therefore gives, for every $x>2$,
\[
 g(x)\ge g(12)+g'(12)(x-12)+\frac32(x-12)^2
       \ge g(12)-\frac{g'(12)^2}{6}
       >\frac{27}{10}-\frac{(19/5)^2}{6}
       =\frac{22}{75}.
\]
Together with the first two steps, this proves the stronger estimate
\begin{equation}\label{mono:eq:gap}
 \mathfrak d_{N+1}<\mathfrak d_N-\frac{22}{75}\qquad(N\ge9),
\end{equation}
and hence the asserted strict monotonicity.
\end{proof}

The strict Hardy inequality now requires only one initial sign
estimate. We isolate it from the dimensional comparison.

\begin{lemma}\label{mono:lem:N9base}
The ninth-dimensional defect satisfies $\mathfrak d_9<-1/3$.
\end{lemma}
\begin{proof}
Appendix~\ref{app:defectsmallchecks} proves the stronger bound
$\mathfrak d_9<-7/20$ by a uniform lower estimate on the integrated
beta kernel and the theta terms through $e^{-4t}$.
\end{proof}

\begin{theorem}\label{thm:defectsign}
For every integer $N\ge9$,
\begin{equation}\label{eq:defectnegative}
 \mathfrak d_N<-\frac13-\frac{22}{75}(N-9)<0.
\end{equation}
Consequently,
\begin{equation}\label{eq:defecthardystrict}
 C(N)<A_N=\frac{(N-2)^2}{4}\qquad(N\ge9).
\end{equation}
\end{theorem}

\begin{proof}
The defect bound follows by iterating \eqref{mono:eq:gap}, starting
from Lemma~\ref{mono:lem:N9base}.

It remains to convert the defect sign into a Hardy competitor. For any
finitely supported $\chi$ the ground-state identity gives
\begin{align}\label{eq:defectgroundstate}
 &\E_N[h_N\chi]-A_N\D_N[h_N\chi]\notag\\
 &\qquad=\sum_{\{x,y\}\subset X_N,\,x\sim y}
 h_N(x)h_N(y)|\chi(x)-\chi(y)|^2
 +\sum_{x\in X_N}\rho_N(x)|\chi(x)|^2.
\end{align}
Choose a logarithmic cutoff equal to one on $|x|\le R$, decreasing
linearly in $\log|x|$ on $R<|x|<R^2$, and zero on $|x|\ge R^2$.
The first term on the right of \eqref{eq:defectgroundstate} is
$O_N((\log R)^{-1})$, while the second tends to $\mathfrak d_N$ by
absolute convergence. Hence for $\mathfrak d_N<0$ the left side is
negative for all sufficiently large $R$. This produces a finitely
supported competitor with Hardy quotient strictly below $A_N$, proving
\eqref{eq:defecthardystrict}.
\end{proof}

\begin{remark}\label{rem:defectcontinuous}
Formula \eqref{eq:defecttheta} defines a scalar interpolation
$\mathfrak d(\xi)$ for real $\xi>2$ by replacing $N$ with $\xi$ in the
beta law and the powers of $\Theta$. Corollary~\ref{cor:defectrealzero}
proves its continuity on $[8,9]$ and the existence of at least one zero
in $(8,9)$, using the positive endpoint estimate below.
The interpolation varies a real parameter in the scalar theta formula,
while Theorem~\ref{mono:thm:main} compares consecutive integer dimensions.
\end{remark}

\subsection{The positive defect and the dimensional sign law}
\label{subsec:apriorisign}

The spatial remainders in Proposition~\ref{prop:hardyremainder} give both
the positive defect and Green finiteness by direct form arguments.
These consequences provide the low-dimensional input to the exterior
and spectral analysis.

\begin{proposition}\label{prop:defectpositiveapriori}
For $3\le N\le8$, put
$\beta_N=2N\min_{x\in\Sigma_N}w_N(x)>0$. Then
\begin{equation}\label{eq:defectpositivelower}
 \mathfrak d_N\ge\beta_N>0,
 \qquad \mathfrak d_8\ge16c_*=\frac{16}{119}.
\end{equation}
Each $H_N$ has finite Green kernel, with
$G_N(x,x)\le w_N(x)^{-1}$.
\end{proposition}
\begin{proof}
The cutoffs $h_N\chi_R$ in \eqref{eq:defectgroundstate} equal one
on $\Sigma_N$ for large $R$. Apply \eqref{eq:hardyremainder}, retain
the first-sphere terms, and take the limit to obtain
$\mathfrak d_N\ge\beta_N$.
For $\varepsilon>0$, inverse form order gives
\[
 \langle\delta_x,(H_N+\varepsilon)^{-1}\delta_x\rangle
 \le(w_N(x)+\varepsilon)^{-1}\le w_N(x)^{-1}.
\]
Monotone convergence and resolvent Cauchy--Schwarz prove Green
finiteness at every pair. Formula~\eqref{eq:ownmargin8} gives the
stated bound in dimension eight.
\end{proof}

\begin{lemma}\label{lem:N8subcritical}
For $3\le N\le8$, both $H_N$ and its first-sphere exterior $K_N$
satisfy the assumptions of the general threshold theory in
Section~\ref{sec:lowenergy}, with exact critical tails.
\end{lemma}
\begin{proof}
Both domains are connected integer lattices with finitely many vertices
deleted. Retaining the diagonal of a deleted edge adds a finite
potential to the graph Laplacian. Compression of
\eqref{eq:hardyremainder} gives $K_N\ge M_{w_N|_{Y_N}}$.
The preceding inverse form estimate proves Green finiteness for both
operators. Their potentials equal $-A_N|x|^{-2}$ outside a finite set.
\end{proof}

\begin{corollary}\label{cor:defectsigncomplete}
For every integer $N\ge3$,
\begin{equation}\label{eq:defectsigncomplete}
 \mathfrak d_N>0\quad(3\le N\le8),\qquad
 \mathfrak d_N<0\quad(N\ge9).
\end{equation}
\end{corollary}
\begin{proof}
Use Proposition~\ref{prop:defectpositiveapriori} for $3\le N\le8$ and
Theorem~\ref{thm:defectsign} for $N\ge9$.
\end{proof}

For the same logarithmically cut off continuum-critical profile in each
dimension, the limiting energy changes from positive to negative between
the integers eight and nine. The negative defect produces a negative
test for every $N\ge9$. In the complementary range, the positive
remainders in Proposition~\ref{prop:hardyremainder} control all
competitors and also establish the positive defect sign.

The scalar interpolation also locates a zero strictly between eight
and nine.

\begin{corollary}\label{cor:defectrealzero}
Let $\mathfrak d(\xi)$ be the right side of \eqref{eq:defecttheta} with
$N$ replaced by the real parameter $\xi>2$,
$s_\xi=(\xi-2)/2$, and the beta law
$\operatorname{Beta}((\xi-2)/4,(\xi-2)/4)$. Its restriction to $[8,9]$
is continuous and has at least one zero in $(8,9)$. In particular,
$\mathfrak d(8)\ge16/119>0$ and $\mathfrak d(9)<-1/3$.
\end{corollary}
\begin{proof}
We check continuity with the cancellation at $t=0$ left intact. Put
$\Theta_v(t)=\sum_{k\in\Z}e^{-t(k+v)^2}$ and let $V_\xi$ have the
indicated beta law. The one-dimensional square completion used in
Appendix~\ref{app:defecttheta} is valid also for real $\xi$ and gives
\begin{align*}
 \mathcal J_\xi(t)=\xi\big[2\Theta(t)^\xi
 &-2\Theta(t)^{\xi-1}\mathbb E_\xi\{
       e^{-tV_\xi(1-V_\xi)}\Theta_{V_\xi}(t)\}\\
 &-2+4\mathbb E_\xi e^{-tV_\xi}\big].
\end{align*}
The shifted Jacobi formula~\eqref{eq:appshiftedJacobi} gives
$\Theta_v(t)=(\pi/t)^{1/2}(1+O(e^{-\pi^2/t}))$ uniformly in
$v\in[0,1]$. Also
$\mathbb E_\xi[V_\xi(1-V_\xi)]=(\xi-2)/(4\xi)$.
Expansion to first order in $t$ therefore gives, uniformly for
$8\le\xi\le9$,
\[
 \mathcal J_\xi(t)=s_\xi\pi^{\xi/2}t^{1-\xi/2}
                    +O(t^{2-\xi/2}+1),\qquad 0<t\le1.
\]
The leading term cancels the one in
$s_\xi t(\Theta(t)^\xi-1)$. Including the factor $t^{s_\xi-1}$,
the combined integrand is bounded by a fixed constant on $(0,1]$.
On $[1,\infty)$, $0\le\mathcal B_\xi\le1$ and
$\Theta(t)-1=O(e^{-t})$ bound it by $Ct^{7/2}e^{-t}$, uniformly in
$\xi\in[8,9]$. The integrand is continuous in $\xi$ for each $t>0$:
the beta densities have a common integrable majorant, and the theta
series converge locally uniformly. Dominated convergence proves
continuity of the combined integral and hence of $\mathfrak d$.
The endpoint signs follow from
Proposition~\ref{prop:defectpositiveapriori} and
Lemma~\ref{mono:lem:N9base}. The intermediate value theorem
then gives a zero in $(8,9)$.
\end{proof}

\subsection{From the defect sign to spectral instability}

Let $L_N^D$ be the bounded self-adjoint operator on $\ell^2(X_N)$ given by
\[
 (L_N^Du)(x)=2Nu(x)-\sum_{y\in X_N,\,y\sim x}u(y).
\]
Set
\begin{equation}\label{eq:structuralH}
 H_N=L_N^D-A_N M_{|x|^{-2}}.
\end{equation}
For finitely supported functions,
$\langle u,H_Nu\rangle=\E_N[u]-A_N\D_N[u]$.

The essential spectrum and the high-dimensional spectral consequence
are given in
Lizama's preprint on the dimensional classification~\cite{LizamaSharp}
(Corollary~6.1).
We record them in a variational form that applies in every dimension.

\begin{proposition}\label{prop:spectraltransition}
For every integer $N\ge3$,
\[
 \sigma_{\mathrm{ess}}(H_N)=[0,4N].
\]
Moreover,
\[
 C(N)=A_N\quad\Longleftrightarrow\quad H_N\ge0,
\]
and $C(N)<A_N$ if and only if $H_N$ has a negative discrete eigenvalue.
\end{proposition}

\begin{proof}
The Fourier symbol of the free lattice Laplacian is
\[
 \sigma_N(\theta)=4\sum_{j=1}^N\sin^2(\theta_j/2),
 \qquad \theta\in[-\pi,\pi]^N.
\]
Its range and essential range are $[0,4N]$, so that is its essential spectrum.
On adjoining a one-dimensional space at the origin, $L_N^D\oplus0$ differs
from the free Laplacian only in entries incident with the origin. This is
a finite-rank perturbation. Also $M_{|x|^{-2}}$ is compact: its restriction
to any finite ball has finite rank, and the norm of the omitted diagonal
tends to zero. Invariance of essential spectrum under compact perturbations
gives the stated interval for $H_N$.

If $C(N)=A_N$, its quadratic form is nonnegative on finitely supported
functions, and hence on all of $\ell^2(X_N)$ by boundedness and density.
Conversely, $H_N\ge0$ gives $C(N)\ge A_N$; Proposition~\ref{prop:scaling}
gives equality. If $C(N)<A_N$, a finite competitor has negative quadratic
energy for $H_N$. Thus $\inf\sigma(H_N)<0$. Since the essential spectrum
starts at zero, there is a negative isolated eigenvalue of finite
multiplicity. Conversely, truncating a negative eigenvector gives a finite
competitor with negative quadratic energy, and therefore $C(N)<A_N$.
\end{proof}

The spectral equivalence also yields an exact statement for mesh
refinement. Its dimensional alternatives use the lower bounds proved
in Proposition~\ref{prop:hardyremainder},
and the total-defect sign in Theorem~\ref{thm:defectsign}.

\begin{corollary}
\label{cor:meshstability}
For $h>0$, give $hX_N$ the measure assigning mass $h^N$ to each
vertex, extend functions by zero at the origin, and define
\begin{equation}\label{eq:meshoperator}
 (H_{N,h}u)(x)=h^{-2}\sum_{j=1}^N
 \bigl(2u(x)-u(x+he_j)-u(x-he_j)\bigr)
       -\frac{A_N}{|x|^2}u(x).
\end{equation}
The map $U_h:\ell^2(hX_N,h^N)\to\ell^2(X_N)$ given by
$(U_hu)(n)=h^{N/2}u(hn)$ is unitary and satisfies
\begin{equation}\label{eq:meshunitary}
 U_hH_{N,h}U_h^{-1}=h^{-2}H_N.
\end{equation}
In particular, for every $t>0$,
\begin{equation}\label{eq:meshsemigroup}
 \|e^{-tH_{N,h}}\|_{\ell^2\to\ell^2}=
 \begin{cases}
 1,&3\le N\le8,\\
 \exp(\kappa_Nt/h^2),&N\ge9,
 \end{cases}
 \qquad \kappa_N=-\inf\sigma(H_N)>0\quad(N\ge9).
\end{equation}
Thus these discrete critical operators have a common lower spectral
bound as $h\downarrow0$ if and only if $3\le N\le8$.
\end{corollary}

\begin{proof}
The equality of norms follows from
$\sum_n|h^{N/2}u(hn)|^2=h^N\sum_{x\in hX_N}|u(x)|^2$.
The lattice difference and the potential in \eqref{eq:meshoperator}
both acquire the factor $h^{-2}$ under this map, proving
\eqref{eq:meshunitary}. Proposition~\ref{prop:spectraltransition}, combined with those
dimension-specific bounds, gives $\inf\sigma(H_N)=0$ for $3\le N\le8$ and
$\inf\sigma(H_N)=-\kappa_N<0$ for $N\ge9$.
The spectral theorem now gives
$\|e^{-tH_{N,h}}\|=\exp(-t\inf\sigma(H_{N,h}))$, which proves
\eqref{eq:meshsemigroup} and the last assertion.
\end{proof}

The critical continuum form in \eqref{eq:continuumHardy} has a
nonnegative Friedrichs realization in every dimension. On smooth
functions supported away from the origin, the sampled expression
\eqref{eq:meshoperator} converges to $-\Delta-A_N|x|^{-2}$.
Corollary~\ref{cor:meshstability} concerns the uniform stability of
these approximations on their full Hilbert spaces: local consistency
away from a singularity does not by itself preserve a sharp
nonnegativity threshold.

For a class of regularized continuum parabolic equations, including
higher-order equations, Galaktionov--Kamotski
\cite[Theorem~2.1 and Lemma~2.1]{GalaktionovKamotski} relate exponential
growth to positive eigenvalues of a rescaled generator. Their result
assumes exponential decay of the corresponding eigenfunction and a
polynomial lower bound on the magnitude of its projection against the
rescaled initial data.

Conforming continuous piecewise linear finite elements provide a
variational approximation of the continuum Hardy quotient on bounded
smooth convex domains containing the origin. For shape-regular,
quasi-uniform meshes, Ignat--Zuazua~\cite[Theorem~1.1]{IgnatZuazua}
prove that the corresponding constants approach $A_N$ from above with
an excess of order $|\log h|^{-2}$ as the mesh size $h$ tends to zero.
The inclusion of these finite element spaces in the continuum energy
space preserves nonnegativity of the critical form.

The spectral Rayleigh quotient uses the denominator $\|u\|_2^2$,
whereas the Hardy quotient uses $\D_N[u]$. These normalizations give
different minima and the same criterion for the existence of a
negative direction.

\section{Sharp Hardy inequalities and positive remainders}\label{sec:hardyproofs}\label{sec:common}

This section constructs positive vertex remainders with quantitative
control on the first sphere for all six sharp dimensions. Reciprocal
fields reduce the form inequality to local diagonal estimates.
Local quadratic estimates control a finite interior while the field on
exterior and crossing edges is retained.
In dimensions three through seven, that region consists of at most
three layers of vertices with coordinates in $\{0,\pm1\}$, grouped
by the number of nonzero coordinates. Dimension eight uses two
families of classes under sign changes and coordinate permutations,
together with uniform scalar estimates.

The proofs share the reciprocal-field framework and local completion
with Alpay's low-dimensional preprint~\cite{AlpayLow}
(Sections~3--6), with alternative explicit constructions as follows.
For dimensions three and four,
Lizama's preprint~\cite{LizamaSharp}
(Lemmas~3.3--3.4) uses shifted radial supersolutions;
Alpay's three-dimensional preprint~\cite{AlpayThree}
(Theorem~2.1) uses a shifted linear reciprocal field.
The quadratic fields below give different constructions in those
dimensions. In dimensions four through seven, they belong to the
quadratic family of Alpay's low-dimensional preprint, with different
parameters and positivity estimates. In dimensions six and seven, we
retain and complete the interior quadratic form.

The eight-dimensional construction in Alpay's low-dimensional
preprint assigns an integer profile to the $26$ classes of
$\{1\le |x|^2\le12\}$ under sign changes and coordinate permutations.
Our construction uses a different exterior field and the two families
selected by $\sum_jx_j^4-|x|^2<24$, with profiles whose successive ratios
are rational functions. The formulas, their positivity estimates,
and the resulting explicit remainder are developed below and in
Appendix~\ref{app:fielddata}.

\subsection{Reciprocal fields and local completion}\label{subsec:generalcompletion}

The edge identity in Lemma~\ref{lem:edge} extends directly to weighted
graphs. Let $X$ be a locally finite graph with symmetric conductances
$c_{xy}=c_{yx}\ge0$, $c_{xx}=0$, and a real diagonal $a(x)$. Write
\[
 (Lu)(x)=a(x)u(x)-\sum_yc_{xy}u(y).
\]
All form identities below are for finitely supported complex functions,
with counting measure. A deleted neighbour is omitted from the sum
without changing the prescribed diagonal.

For a reciprocal field $\rho_{xy}>0$, $\rho_{yx}=\rho_{xy}^{-1}$, put
$V_\rho(x)=a(x)-\sum_yc_{xy}\rho_{xy}$. Expansion on each edge gives
\begin{equation}\label{eq:generalweightedidentity}
 \langle u,Lu\rangle
 =\sum_{\{x,y\}}c_{xy}
   \left|\sqrt{\rho_{xy}}u(x)-\sqrt{\rho_{yx}}u(y)\right|^2
   +\sum_xV_\rho(x)|u(x)|^2.
\end{equation}
Each unoriented edge with positive conductance is counted once.
The mixed term is $-2c_{xy}\operatorname{Re}(u(x)\overline{u(y)})$;
the vertex remainder restores the diagonal $a(x)$.

If $\rho_{xy}=h(y)/h(x)$ for a positive function $h$, then
$V_\rho=Lh/h$ and \eqref{eq:generalweightedidentity} is the usual
ground-state representation~\cite{KPPhardy}. Its half-line
factorization and explicit remainder are discussed in
\cite{KrejcirikStampach}. Here the field need only be reciprocal. This permits us to join a
local profile to an independently chosen exterior field without first
finding a global positive supersolution. The essential estimate is
the resulting vertex remainder for the prescribed potential.

The same identity gives a direct form estimate on a finite patch. It
retains the exterior remainder, so a positive interior form supplies a
positive weight at every vertex without constructing an interior field.

\begin{lemma}[Localization with retained exterior costs]
\label{lem:directlocalization}
Let $\Omega\subset X_N$ be finite. Suppose a positive reciprocal field
$b_{xy}$ is given on every edge with at least one endpoint outside
$\Omega$, and put
\[
 V_b(x)=2N-\frac{A_N}{|x|^2}
            -\sum_{y\in X_N,\,y\sim x}b_{xy}
 \qquad(x\notin\Omega).
\]
Define
\begin{align}
 Q_\Omega[u]={}&\sum_{x\in\Omega}
 \left(2N-\frac{A_N}{|x|^2}
       -\sum_{\substack{y\in X_N\setminus\Omega\\y\sim x}}b_{xy}
 \right)|u(x)|^2\notag\\
 &-2\operatorname{Re}\sum_{\substack{\{x,y\}\subset\Omega\\x\sim y}}
                   u(x)\overline{u(y)}.
 \label{eq:directlocalizationcore}
\end{align}
Then every $u\in C_c(X_N;\mathbb C)$ satisfies
\begin{equation}\label{eq:directlocalization}
 \langle u,H_Nu\rangle\ge Q_\Omega[u]
                   +\sum_{x\notin\Omega}V_b(x)|u(x)|^2.
\end{equation}
In particular, if $Q_\Omega[u]\ge\sum_{x\in\Omega}w(x)|u(x)|^2$
with $w>0$ on $\Omega$, and $V_b>0$ outside it, then
$H_N\ge M_{\widetilde w}$ for the strictly positive weight equal to
$w$ on $\Omega$ and to $V_b$ elsewhere.
\end{lemma}
\begin{proof}
Extend $b$ arbitrarily as a positive reciprocal field on the edges
inside $\Omega$. Apply Lemma~\ref{lem:edge} and subtract the Hardy
potential. Discard only those edge squares which are not internal to
$\Omega$. Expanding the retained squares cancels the internal field
ratios and leaves exactly \eqref{eq:directlocalizationcore}; the
remaining vertex terms are $V_b$ outside $\Omega$. This proves
\eqref{eq:directlocalization}, including the original full diagonal at
neighbours of the removed origin. The last assertion follows by
substitution.
\end{proof}

For fields depending on squared radii, reciprocity can be built into the
formula. With $s,t>0$, set
\begin{equation}\label{eq:generaloddfield}
 b(s,t)=\frac{1-\Psi(v,w)}{1+\Psi(v,w)},\qquad
 v=\frac{t-s}{t+s},\qquad w=\frac1{t+s}.
\end{equation}
If $\Psi(-v,w)=-\Psi(v,w)$ and $|\Psi(v,w)|<1$ on the required
edges, then $b$ is positive and reciprocal there. Indeed, exchanging
$s,t$ changes only the sign of $v$. This construction permits rational
corrections while leaving reciprocity exact.

The local patching principle is used in
Alpay's low-dimensional preprint~\cite{AlpayLow}
(Lemma~3.5): a positive profile supplies ratios on
interior edges while the exterior and crossing edges retain their
assigned values. The following version allows weighted edges, a
prescribed diagonal, and a nonnegative interior remainder.
We call this prescribed remainder the \emph{reserve} and denote it by
$\eta$. The outgoing-cost operator below subtracts both this reserve
and the contributions assigned to edges leaving the patch.

\begin{proposition}\label{prop:analyticcompletion}
Let $\Omega\subset X$ be finite. Suppose that a positive reciprocal field
$b$ is given on every edge with at least one endpoint outside $\Omega$,
and that
\[
 V_b(x)=a(x)-\sum_yc_{xy}b_{xy}\ge0\qquad(x\notin\Omega).
\]
Let $\eta:\Omega\to[0,\infty)$ and define the outgoing-cost operator
\begin{equation}\label{eq:generalpatchoperator}
 (M_\eta h)(x)=
 \left(a(x)-\eta(x)-\sum_{y\notin\Omega}c_{xy}b_{xy}\right)h(x)
       -\sum_{y\in\Omega}c_{xy}h(y),\qquad x\in\Omega.
\end{equation}
If $h>0$ on $\Omega$ and $M_\eta h\ge0$, the completed field
\begin{equation}\label{eq:generalcompletedfield}
 (\mathscr P_{\Omega,h}b)_{xy}
 =\begin{cases}
 h(y)/h(x),&x,y\in\Omega,\\
 b_{xy},&\text{otherwise}
 \end{cases}
\end{equation}
is positive and reciprocal, and its remainder is
\begin{equation}\label{eq:generalcompletedremainder}
 V_{\mathscr P_{\Omega,h}b}(x)=
 \begin{cases}
 \eta(x)+(M_\eta h)(x)/h(x),&x\in\Omega,\\
 V_b(x),&x\notin\Omega.
 \end{cases}
\end{equation}
Consequently
$\langle u,Lu\rangle\ge\sum_{x\in\Omega}\eta(x)|u(x)|^2$.
If $M_\eta h>0$ on $\Omega$ and $V_b>0$ outside it, this inequality
is strict for every nonzero finitely supported $u$.
\end{proposition}
\begin{proof}
The internal ratios are reciprocal because $h>0$. Substitution of
\eqref{eq:generalcompletedfield} at a vertex of $\Omega$ gives
\eqref{eq:generalcompletedremainder}. Every edge incident with a vertex
outside $\Omega$ is retained, in both orientations, so its remainder
is unchanged. Apply \eqref{eq:generalweightedidentity}. Under the strict
hypotheses the vertex terms alone give the strict form inequality.
\end{proof}

The reserve $\eta$ is prescribed before choosing the profile.
Reciprocity permits exterior field estimates and an interior profile
to be combined whenever their crossing-edge costs agree.

\subsection{Positive remainders through dimension seven}
\label{sec:lowdimensions}

The first-sphere margin and the spatial remainder can be established
together. In each dimension, a rational reciprocal field controls the
exterior of at most three Boolean layers. One or two completed squares
then provide explicit positive weights on those layers. The estimates apply to all finitely supported complex competitors.

\subsubsection{The fixed field and its exterior remainder}

Put $s(x)=|x|^2$, $\mu=(N-2)/4$, and choose
\begin{equation}\label{eq:lowfixedparameters}
 \kappa_N=\begin{cases}N/4,&3\le N\le6,\\9/4,&N=7.\end{cases}
\end{equation}
These parameters satisfy $4\kappa_N>\mu^2$. Hence the reciprocal field
in \eqref{eq:uniformbase} is positive. Its rational expression is
\begin{equation}\label{eq:rationalfield}
 a_N(s,t)=
 \begin{cases}
 \displaystyle\frac{(N+1)s^2+(4-N)st+3t^2}
 {3s^2+(4-N)st+(N+1)t^2},&3\le N\le6,\\[2mm]
 \displaystyle\frac{9s^2-5st+4t^2}{4s^2-5st+9t^2},&N=7.
 \end{cases}
\end{equation}
Define its unmodified vertex remainder by
\begin{equation}\label{eq:lowbaseremainder}
 V_N(x)=2N-\frac{A_N}{s(x)}-
       \sum_{\substack{y\in X_N\\y\sim x}}a_N(s(x),s(y)).
\end{equation}

At fixed squared radius, each coordinate contributes a pair of edge
terms. We compare that pair with an affine function of the coordinate
square whose sum gives the Hardy potential. The scalar remainder
$R_N$ measures the gap after positive factors are removed.

\begin{lemma}\label{lem:lowexterior}
For the parameters \eqref{eq:lowfixedparameters}, the scalar remainder
in \eqref{eq:uniformfactor} satisfies $R_N(s,\tau)>0$ on the ranges
\[
 \begin{array}{ll}
 s\ge2,\ \tau\ge0,&N=3,4,5,\\
 s\ge3,\ 0\le\tau\le1,&N=6,\\
 s\ge5,\ 0\le\tau\le1,&N=7.
 \end{array}
\]
Consequently $V_N(x)>0$ for $s(x)\ge2$ when $N=3,4,5$, for
$s(x)\ge3$ when $N=6$, and for $s(x)\ge4$ when $N=7$.
\end{lemma}
\begin{proof}
The coefficient estimates are proved in
Appendix~\ref{app:lowalgebra} using the three
explicit coefficients \eqref{eq:uniformR0}--\eqref{eq:uniformR2}.
For every nonzero coordinate, \eqref{eq:uniformfactor} gives a strict
inequality $g_s(x_j^2)>\ell_s(x_j^2)$; for a zero coordinate there is
equality. Summing and using \eqref{eq:uniformsumg} and
\eqref{eq:uniformcalibration} proves the exterior assertion at the
indicated radii. For $N=7$ and $s=4$, the only signed-permutation
types are $(2)$ and $(1^4)$. Their two direct row estimates also appear
in Appendix~\ref{app:affine7}.
\end{proof}

\subsubsection{Explicit spatial weights and first-sphere margins}

Write $\mathcal O_j^{(N)}$ for the signed-permutation orbit of
$(1^j)$, so $\mathcal O_1^{(N)}=\Sigma_N$, and put
\[
 \Omega_N=\begin{cases}
 \mathcal O_1^{(N)},&3\le N\le5,\\
 \mathcal O_1^{(6)}\cup\mathcal O_2^{(6)},&N=6,\\
 \mathcal O_1^{(7)}\cup\mathcal O_2^{(7)}\cup\mathcal O_3^{(7)},&N=7.
 \end{cases}
\]
The following constants specify the weights at every vertex of these
patches:
\begin{equation}\label{eq:lowbetavalues}
 (\beta_3,\beta_4,\beta_5,\beta_6,\beta_7)=(9,15,21,21,10).
\end{equation}

\begin{theorem}\label{thm:lowtrace}
For $3\le N\le7$, define the signed-permutation invariant weight
\begin{equation}\label{eq:lowweight}
 w_N(x)=\begin{cases}
 \beta_N/(2N),&x\in\Sigma_N,\\
 1/46,&N=6,\ x\in\mathcal O_2^{(6)},\\
 1/264,&N=7,\ x\in\mathcal O_2^{(7)},\\
 1/556,&N=7,\ x\in\mathcal O_3^{(7)},\\
 V_N(x),&x\notin\Omega_N.
 \end{cases}
\end{equation}
Then $w_N(x)>0$ for every $x\in X_N$, and
\begin{equation}\label{eq:lowpositiveweight}
 \langle u,H_Nu\rangle\ge\sum_{x\in X_N}w_N(x)|u(x)|^2
 \qquad(u\in C_c(X_N;\mathbb C)).
\end{equation}
In particular,
\begin{equation}\label{eq:lowtrace}
 \langle u,H_Nu\rangle\ge
             \frac{\beta_N}{2N}\|u|_{\Sigma_N}\|_2^2,
 \qquad \mathfrak D_N\ge\beta_N>0.
\end{equation}
The Hardy inequality is strict for every nonzero finite-support
competitor, and $C(N)=A_N$.
\end{theorem}
\begin{proof}
Set $U_j=\|u|_{\mathcal O_j^{(N)}}\|_2$. The degrees between
$\mathcal O_j^{(N)}$ and $\mathcal O_{j+1}^{(N)}$ are $2(N-j)$ and
$j+1$. Cauchy--Schwarz over their edges therefore gives
\begin{equation}\label{eq:lowincidence}
 \left|\sum_{\substack{x\in\mathcal O_j^{(N)},\ y\in\mathcal O_{j+1}^{(N)}\\x\sim y}}
                     u(x)\overline{u(y)}\right|
 \le\sqrt{2(N-j)(j+1)}\,U_jU_{j+1}.
\end{equation}
There are no edges inside an individual orbit, since a lattice edge
changes the parity of the squared radius. We now estimate
$Q_{\Omega_N}$ from Lemma~\ref{lem:directlocalization}.

\paragraph{Dimensions three through five.}
The patch is $\Sigma_N$. Each of its vertices has $2(N-1)$ outgoing
neighbours of squared radius two and one of squared radius four. Thus
$Q_{\Omega_N}[u]=c_NU_1^2$, where the rational field gives
\begin{align*}
 c_3&=\frac{23}{4}-4\frac67-\frac{56}{71}
       =\frac{3047}{1988}>\frac32,\\
 c_4&=7-6\frac{17}{23}-\frac{53}{83}
       =\frac{3678}{1909}>\frac{15}{8},\\
 c_5&=\frac{31}{4}-8\frac{16}{25}-\frac{10}{19}
       =\frac{3997}{1900}>\frac{21}{10}.
\end{align*}
These are the first three weights in \eqref{eq:lowweight}.

\paragraph{Dimension six.}
The diagonal costs on the first two Boolean layers are
\[
 \gamma_1=8-a_6(1,4),\qquad
 \gamma_2=10-8a_6(2,3)-2a_6(2,5).
\]
Since
\[
 a_6(1,4)=\frac{47}{107}<\frac12,\quad
 a_6(2,3)=\frac{43}{63}<\frac{11}{16},\quad
 a_6(2,5)=\frac{83}{167}<\frac12,
\]
one has $\gamma_1>15/2$ and $\gamma_2>7/2$. Hence
$Q_{\Omega_6}[u]\ge F_6(U_1,U_2)$, where
\[
 F_6=\frac{15}{2}U_1^2+\frac72U_2^2-2\sqrt{20}\,U_1U_2.
\]
The following exact completed square retains a positive weight on
both layers:
\begin{equation}\label{eq:low6square}
 F_6=\left(\sqrt{\frac{23}{4}}U_1-\sqrt{\frac{80}{23}}U_2\right)^2
       +\frac74U_1^2+\frac1{46}U_2^2.
\end{equation}

\paragraph{Dimension seven.}
On the first three Boolean layers the diagonal costs are
\[
 \gamma_1=\frac{31}{4}-a_7(1,4),\qquad
 \gamma_2=\frac{87}{8}-2a_7(2,5),\qquad
 \gamma_3=\frac{143}{12}-8a_7(3,4)-3a_7(3,6).
\]
Using
\[
 a_7(1,4)=\frac{53}{128}<\frac5{12},\quad
 a_7(2,5)=\frac{86}{191}<\frac{11}{24},\quad
 a_7(3,4)=\frac{17}{24},\quad a_7(3,6)=\frac12,
\]
we obtain $\gamma_1>22/3$, $\gamma_2>239/24$ and
$\gamma_3=19/4$. Thus $Q_{\Omega_7}[u]\ge F_7$, with
\[
 F_7=\frac{22}{3}U_1^2+\frac{239}{24}U_2^2+\frac{19}{4}U_3^2
            -2\sqrt{24}\,U_1U_2-2\sqrt{30}\,U_2U_3.
\]
The two exact completed squares are
\begin{align}
 F_7={}&\left(\sqrt{\frac{33}{5}}U_1-\sqrt{\frac{40}{11}}U_2\right)^2
   +\left(\sqrt{\frac{139}{22}}U_2-\sqrt{\frac{660}{139}}U_3\right)^2
   \notag\\
 &+\frac{11}{15}U_1^2+\frac1{264}U_2^2+\frac1{556}U_3^2.
 \label{eq:low7squares}
\end{align}
Since $11/15>5/7=\beta_7/14$, this proves the prescribed weights
on all three layers.

In each case, Lemma~\ref{lem:directlocalization} and the strictly
positive exterior remainder in Lemma~\ref{lem:lowexterior} now give
\eqref{eq:lowpositiveweight} with exactly \eqref{eq:lowweight}.
Retaining the first-sphere terms proves \eqref{eq:lowtrace}; the
definition of $\mathfrak D_N$ gives its lower bound. Positivity of
$w_N$ proves strictness for nonzero finite-support competitors, and
Proposition~\ref{prop:scaling} supplies the matching upper bound for
$C(N)$.
\end{proof}
\subsection{An analytic reciprocal field in dimension eight}\label{subsec:N8gauge}\label{sec:N8proof}

The interior patch in $X_8$ contains vertices whose nonzero
coordinates all have magnitude one, together with those having one
coordinate of magnitude two and all others of magnitude at most one.
Its classes under sign changes and coordinate permutations form two
chains: $B_j$ has $j$ unit coordinates, and $E_k$ has one coordinate
of magnitude two and $k$ unit coordinates. Their union is denoted by
$\mathcal P=\mathcal P_8$, as in \eqref{eq:generalgeometricpatch}.
The profile has separate values on the first few levels and common
formulas along the remaining levels. The junction equations match
these pieces while leaving positive operator remainders.
We apply Proposition~\ref{prop:analyticcompletion} with $N=8$
and target Hardy coefficient $\lambda=9$.
Throughout this subsection, write
\[
 s(x)=|x|^2,\qquad \Sigma=\{\pm e_i:1\le i\le8\},\qquad H=H_8.
\]
The chain indices are $1\le j\le8$ and $0\le k\le7$;
Lemma~\ref{lem:generalpatchgeometry} gives their geometric description.

The proof has three parts. Proposition~\ref{n8a:prop:exterior} controls
the retained field, using the angular identity \eqref{eq:generalangularsum}
and Appendices~\ref{n8a:app:large}--\ref{n8a:app:low}.
Lemma~\ref{n8a:lem:join} constructs a positive profile with prescribed
junction residuals. Lemma~\ref{n8a:lem:tails}, proved in
Appendices~\ref{n8a:app:interior} and~\ref{n8a:app:denominators},
controls the remaining chain levels on two real intervals.
Proposition~\ref{prop:analyticcompletion} then joins the interior and
exterior estimates while preserving every crossing edge.
Together these estimates give a positive remainder at every vertex
and an explicit lower bound on the first sphere.

\begin{theorem}\label{thm:N8}
There is an explicit positive reciprocal field $\widehat\beta$ on
$X_8$ whose vertex remainder satisfies
\begin{equation}\label{eq:N8globalW}
 W_8(x):=16-\frac9{s(x)}-
       \sum_{y\in X_8,\,y\sim x}\widehat\beta_{xy}>0
       \quad(x\in X_8),\qquad
 W_8|_\Sigma>c_*=\frac1{119}.
\end{equation}
Consequently, for every $u\in C_c(X_8;\C)$,
\begin{equation}\label{eq:N8globalcoercivity}
 \E_8[u]-9\D_8[u]\ge c_*\|P_\Sigma u\|_2^2,
\end{equation}
and $\E_8[u]>9\D_8[u]$ whenever $u\ne0$. In particular $C(8)=9$.
\end{theorem}

\subsubsection{A geometric patch and two lattice chains}\label{subsec:generalgeometry}

To select the interior, we measure the departure of coordinate
magnitudes from zero and one. Write $s=|x|^2$ and
$P_4(x)=\sum_{i=1}^Nx_i^4$. The difference
$P_4(x)-s=\sum_i x_i^2(x_i^2-1)$ vanishes on vertices with
coordinates in $\{0,\pm1\}$ and increases when larger coordinates
appear. The interior patch is
\begin{equation}\label{eq:generalgeometricpatch}
 \mathcal P_N=\{x\in X_N:P_4(x)-s<24\}.
\end{equation}
Let $B_j$ denote the signed-permutation orbit of $(1^j)$ and $E_k$ that
of $(2,1^k)$. The indices below depend on $N$.

\begin{lemma}\label{lem:generalpatchgeometry}
For every integer $N\ge1$, $\mathcal P_N$ is the disjoint union of
$B_j$, $1\le j\le N$, and $E_k$, $0\le k\le N-1$.
On every edge with at least one endpoint outside $\mathcal P_N$,
the variables in \eqref{eq:generaloddfield} satisfy
\begin{equation}\label{eq:generalretainedbounds}
 0<w\le\frac1{13},\qquad |v|\le\frac5{13}.
\end{equation}
If $x\notin\mathcal P_N$ and $s\le11$, its coordinate squares consist
of either two $4$'s and otherwise $0,1$, or one $9$ and otherwise $0,1$.
\end{lemma}
\begin{proof}
Each coordinate of absolute value $2$ contributes $12$ to $P_4-s$;
one of absolute value at least $3$ contributes at least $72$.
This proves the patch description. Outside it, $s\le11$ permits exactly
the two stated patterns: three $4$'s or a $9$ together with a $4$
already exceed $11$.

For a retained edge, change signs and orientation so the varying
coordinate goes from $k$ to $k+1$, $k\ge0$, and let $R$ be the sum
of the other coordinate squares. If $k\ge2$, then
\[
 s+t=2k^2+2k+1+2R\ge13,\qquad
 |v|\le\frac{2k+1}{2k^2+2k+1}\le\frac5{13};
\]
the last quotient decreases for $k\ge2$. If $k=1$, another coordinate
must have absolute value at least $2$, so $R\ge4$ and $s+t\ge13$.
If $k=0$, the other coordinates already have defect at least $24$;
thus $R\ge8$ and $s+t\ge17$. These bounds prove
\eqref{eq:generalretainedbounds} in every case.
\end{proof}

The neighbour counts yield one recurrence for each of the two chains. For the operator with diagonal $2N-\lambda/s$, let $b(s,t)$
be the field on retained edges and put $\mathcal H_N=M_0$ in
\eqref{eq:generalpatchoperator}. For an invariant profile $h$, put
$U_j=h|_{B_j}$ and $Z_k=h|_{E_k}$. Then
\begin{align}
 (\mathcal H_Nh)_{B_j}
 &=\left(2N-\frac\lambda j\right)U_j-jU_{j-1}
       -2(N-j)U_{j+1}-jZ_{j-1},\notag\\
 (\mathcal H_Nh)_{E_k}
 &=d_{N,k}Z_k-U_{k+1}-kZ_{k-1}-2(N-1-k)Z_{k+1},\notag\\
 d_{N,k}
 &=2N-\frac\lambda{k+4}-b(k+4,k+9)-k b(k+4,k+7).
 \label{eq:generaltwochains}
\end{align}
Here $U_0=0$, and terms with zero coefficients are omitted. The first
row follows from the Boolean neighbour counts above. From $E_k$, lowering
the coordinate $2$ gives $B_{k+1}$; lowering a unit coordinate gives
$k$ neighbours in $E_{k-1}$; changing a zero to $\pm1$ gives
$2(N-1-k)$ in $E_{k+1}$. The remaining moves create a $3$ or a second
$2$, giving the two outgoing terms in $d_{N,k}$. The full diagonal
retains the edge to the removed origin.

\subsubsection{Profiles with rational shift ratios}\label{subsec:generalprofiles}

To estimate these recurrences on real intervals, it is useful to replace
a power by a positive function with the same decay and rational shift
ratios. For $\nu>0$, define
\begin{equation}\label{eq:generalgammaprofile}
 G_\nu(s)=\frac{\Gamma(1+\nu)\Gamma(s)}{\Gamma(s+\nu)},
 \qquad s>0.
\end{equation}
The classical gamma recurrence~\cite[Formula~8.331.1]{GradshteynRyzhik}
gives
\begin{equation}\label{eq:generalgammarecurrence}
 G_\nu(1)=1,\qquad
 G_\nu(s+1)=\frac{s}{s+\nu}G_\nu(s),\qquad
 G_\nu(j)=\prod_{i=1}^{j-1}\frac{i}{i+\nu}.
\end{equation}
In particular, its integer values are rational when $\nu$ is rational.
The standard gamma-ratio asymptotic~\cite[Formula~8.328.2]{GradshteynRyzhik}
gives $G_\nu(s)\sim\Gamma(1+\nu)s^{-\nu}$ as $s\to\infty$.
It also follows directly from Euler's beta
integral~\cite[Formulas~8.380.1 and~8.384.1]{GradshteynRyzhik},
after the substitution $y=s(1-t)$:
\[
 s^\nu G_\nu(s)
 =\nu\int_0^s(1-y/s)^{s-1}y^{\nu-1}\,dy
 \longrightarrow\nu\Gamma(\nu).
\]
For $s\ge2$, the integrand, extended by zero for $y\ge s$, is bounded
by $e^{-y/2}y^{\nu-1}$, so dominated convergence applies.
For the Hardy profile take $\nu=\nu_N=(N-2)/4$, as in
\eqref{eq:defectprofile}. This is half the radial exponent
$\alpha=(N-2)/2$ used in the spectral analysis: with $s=r^2$,
$r^{-\alpha}=s^{-\nu_N}$.

For integers $d\ge0$, set
\begin{equation}\label{eq:generalgammashifts}
 r_d(s)=\frac{G_\nu(s+d)}{G_\nu(s)}
       =\prod_{i=0}^{d-1}\frac{s+i}{s+i+\nu},\qquad
 r_{-d}(s)=1/r_d(s-d).
\end{equation}
The first ratio is defined for $s>0$ and the second for $s>d$.
On levels where the profile and its relevant neighbours have the forms
$U_j=G_\nu(j)p(j)$ and $Z_k=G_\nu(k+4)z(k+4)$,
division of the operator remainder by the positive gamma factor
gives one scalar residual along each chain. We denote these functions
by $\mathcal B_{N,\lambda}$ and $\mathcal E_{N,\lambda}$:
\begin{align}
 \mathcal B_{N,\lambda}(s)
 ={}&\left(2N-\frac\lambda s\right)p(s)-sr_{-1}(s)p(s-1)
       -2(N-s)r_1(s)p(s+1)-sr_3(s)z(s+3),\notag\\
 \mathcal E_{N,\lambda}(s)
 ={}&\left[2N-\frac\lambda s-b(s,s+5)-(s-4)b(s,s+3)\right]z(s)
       \notag\\
 &-(s-4)r_{-1}(s)z(s-1)-2(N+3-s)r_1(s)z(s+1)
       -r_{-3}(s)p(s-3).
 \label{eq:generalchainresiduals}
\end{align}
At the corresponding integer levels these equal
$(\mathcal H_Nh)_{B_j}/G_\nu(j)$ and
$(\mathcal H_Nh)_{E_k}/G_\nu(k+4)$. All arguments of $G_\nu$
in these formulas are taken positive. For rational $b,p,z$, the two
residual functions are rational. Positivity on real intervals therefore
controls every integer level contained in them.

\subsubsection{Affine calibration on the squared-coordinate simplex}\label{subsec:simplexgeneral}

An affine minorant converts one scalar inequality into a bound for all
coordinate patterns at the same radius. The following summation rule
reduces the exterior remainder to one scalar inequality.

\paragraph{Affine minorants and lattice coordinate squares.}

The repeated scalar step is a minorant problem with one linear constraint.
For a fixed squared radius $s>0$, suppose a paired edge contribution
is $g_s(q)$ for $0\le q\le s$. To target a total weight $\Lambda(s)$,
choose an intercept $a_s$ and set
\[
 \ell_s(q)=a_s+\frac{\Lambda(s)-Na_s}{s}\,q.
\]
If $g_s(q)\ge\ell_s(q)$ on $[0,s]$, then, whenever
$q_j\ge0$ and $\sum_{j=1}^Nq_j=s$,
\begin{equation}\label{eq:simplexgeneral}
 \sum_{j=1}^Ng_s(q_j)\ge\sum_{j=1}^N\ell_s(q_j)=\Lambda(s).
\end{equation}
This follows by summing the affine function. If the minorant is strict
for $q>0$, the summed inequality is strict since at least one $q_j$ is
positive. The basic Hardy calibration takes
$\Lambda(s)=A_N/s$ and $a_s=g_s(0)$; a shifted intercept will also
be useful below.
We apply the scalar estimate on $0\le q\le s$ and use local
completion at the remaining finitely many radii.

For lattice points, a minorant need only hold on the allowed coordinate
squares. The sets $\{1\}$, $\{4\}$ and $[9,s]$ cover all positive
integer squares, so estimates on these sets may replace a stronger
continuous assertion on $[0,s]$.

A small change of intercept also keeps a positive contribution from zero
coordinates. For a target $\lambda/s$, take
$a_s=g_s(0)-\theta(s)$, where $\theta(s)\ge0$, and define
\begin{equation}\label{eq:generalangularresidual}
 \mathcal R_{N,\lambda}(s,q)
 =\frac{s^3}{q}\bigl(g_s(q)-\ell_s(q)\bigr),\qquad q>0.
\end{equation}
If $q_i=x_i^2$, $\sum_iq_i=s$ and $n_0=\#\{i:q_i=0\}$, summing
the affine function gives the exact identity
\begin{equation}\label{eq:generalangularsum}
 \sum_{i=1}^Ng_s(q_i)-\frac\lambda s
 =s^{-3}\sum_{q_i>0}q_i\mathcal R_{N,\lambda}(s,q_i)
       +n_0\theta(s).
\end{equation}
Thus a lower bound $\mathcal R_{N,\lambda}(s,q)\ge r(s)>0$ on the
positive coordinate squares yields a remainder at least $r(s)/s^2$.

At low radii, Lemma~\ref{lem:generalpatchgeometry} gives a second
reduction. If
$g_s(q)=2-b(s,s+1+2\sqrt q)-b(s,s+1-2\sqrt q)$, the only remainders
outside $\mathcal P_N$ with $s\le11$ are, when the corresponding
patterns occur,
\begin{align}
 L_{2;N,\lambda}(s)
 &=(N+6-s)g_s(0)+(s-8)g_s(1)+2g_s(4)-\lambda/s,\notag\\
 L_{3;N,\lambda}(s)
 &=(N+8-s)g_s(0)+(s-9)g_s(1)+g_s(9)-\lambda/s.
 \label{eq:generallowradii}
\end{align}
The coefficients count the zero and unit coordinates in the two
patterns. Their signs on the required intervals remain an application-
dependent estimate.

\subsubsection{The eighth-dimensional exterior field}
The critical ratio $(t/s)^{-3/2}$ is reciprocal and begins with
$1-3v$. In the representation \eqref{eq:generaloddfield}, its linear
term is therefore fixed by $\Psi=3v/2+\cdots$.
The remaining odd terms adjust the radial and angular errors in the
paired vertex sum. The rational denominator groups higher corrections
while remaining positive on retained edges. The following choice is
designed for the shifted affine calibration below, and its exterior
bounds follow from the interval estimates in Appendix~\ref{app:fielddata}.
Choose
\begin{equation}\label{eq:gaugefield8}
 \Psi(v,w)=v\left[\frac32+\frac{111}{80}w-\frac{125}{32}v^2+
 \frac{\frac{25}{12}w^2-\frac{169}{8}v^2w+\frac{175}{8}v^4}
 {1+8v^2-\frac{20}{3}w}\right],
\end{equation}
and write $\beta_8(s,t)=(1-\Psi(v,w))/(1+\Psi(v,w))$ for the
resulting field. The function $\Psi$ is odd in $v$.
Lemma~\ref{lem:generalpatchgeometry} gives $0<w\le1/13$ and
$|v|\le5/13$ on every retained edge.

For a uniform sign estimate, put $x=v^2$.
The inner denominator is at least $19/39+8x>0$.
The bracket $T=\Psi/v$, extended at $v=0$, satisfies
\[
 T\ge\frac32-\frac{125}{32}\frac{25}{169}
                         -\frac{169}{64\cdot13}>0,
\]
\[
 T\le\frac32+\frac{111}{80\cdot13}
       +\frac{25}{12\cdot13^2}\frac{39}{19}
       +\left(\frac{175}{64}-\frac{125}{32}\right)x<\frac53.
\]
Here we used $x^2/(19/39+8x)\le x/8$ and
$xw/(19/39+8x)\le w/8$.
Thus $|\Psi|\le25/39<1$, and the field is positive and reciprocal.

For the angular calibration, take $\theta(s)=1/(4s^4)$ in
\eqref{eq:generalangularresidual}. The paired contribution and the
resulting residual $\Rcal=\mathcal R_{8,9}$ are
\begin{align*}
 g_s(q)&=2-\beta_8(s,s+1+2\sqrt q)-\beta_8(s,s+1-2\sqrt q),\\
 \Rcal(s,q)&=\frac{s^3}{q}\bigl(g_s(q)-g_s(0)\bigr)
                 -9s+8s^2g_s(0)+\frac{s/q-8}{4s^2},\quad q>0.
\end{align*}

\begin{proposition}\label{n8a:prop:exterior}\label{lem:gaugeouter8}
The exterior field satisfies
\[
 V_{\beta_8}(x):=\sum_{i=1}^8g_s(x_i^2)-9/s>0
                    \qquad(x\in X_8\setminus\mathcal P).
\]
More precisely, $V_{\beta_8}(x)>1/(200s^5)$ if $s\ge12$.
\end{proposition}

\begin{proof}
For $s\ge19$, Appendix~\ref{n8a:app:large} proves
$\Rcal(s,q)>5/(7s^3)$ for every real $0<q\le s$.
For $12\le s\le19$, Appendix~\ref{n8a:app:compact} proves
$\Rcal(s,q)>1/(200s^3)$ for $q=1$, $q=4$ and $9\le q\le s$.
Every positive integer square belongs to one of these regions.
Equation~\eqref{eq:generalangularsum} with $\theta(s)=1/(4s^4)$
therefore gives the bound for $s\ge12$.

For $s\le11$, Lemma~\ref{lem:generalpatchgeometry} and
\eqref{eq:generallowradii} reduce the exterior to
$L_2=L_{2;8,9}$ and $L_3=L_{3;8,9}$.
Appendix~\ref{n8a:app:low} proves $L_2>0$ on $[8,11]$ and
$L_3>0$ on $[9,11]$. This completes the exterior estimate.
\end{proof}

\subsubsection{A profile defined by positive junction identities}\label{n8a:sec:profile}
Let $h$ denote the local profile, distinct from the prescribed critical
power $h_8(x)=|x|^{-3}$. Write $U_j=h|_{B_j}$, $Z_k=h|_{E_k}$, and
specialize \eqref{eq:generaltwochains} to $N=8$ and $\lambda=9$,
using exterior field $\beta_8(s,t)$. Denote the resulting operator by $\mathcal H$ and its
second-chain diagonal by $d_k=d_{8,k}$. Fix
\[
 \delta=\frac1{119},\qquad \varepsilon=\frac1{3514},\qquad
 A=\mathcal H-\delta P_{B_1},\qquad \alpha_0=7-\delta.
\]
Here $P_{B_1}$ is multiplication by the indicator of $B_1=\Sigma$;
thus $A=M_\eta$ with reserve $\eta=\delta\mathbf1_\Sigma$.
Use the critical gamma factor $G=G_{3/2}$ from
\eqref{eq:generalgammaprofile}, and put $g_j=G(j)$.
Its recurrence \eqref{eq:generalgammarecurrence} gives $g_1=1$ and
\[
 g_2=\frac25,\quad g_3=\frac8{35},\quad g_4=\frac{16}{105},\quad
 g_5=\frac{128}{1155},\quad g_6=\frac{256}{3003},\quad
 g_7=\frac{1024}{15015}.
\]
Write the two profiles as a critical gamma factor times corrections
in $1/s$. Fix the higher coefficients first, leaving two leading
coefficients on the Boolean chain and one on the second chain free.
Together with the four exceptional values $U_1,U_2,U_3,Z_0$, these
give the seven parameters needed to impose the junction equations.
The fixed corrections are
\begin{align*}
 r(s)&=-\frac{1.4824723}{s^2}+\frac{9.9151639}{s^3}
                                      -\frac{17.4876379}{s^4},\\
 q(s)&=\frac{2.0555724}{s}-\frac{15.9024312}{s^2}
                  +\frac{119.8117632}{s^3}-\frac{211.5171688}{s^4}.
\end{align*}
All terminating decimals in this proof denote exact rational numbers.
Seek $p(s)=a+b/s+r(s)$, $z(s)=\gamma+q(s)$ and set
\[
 U_j=G(j)p(j)\ (4\le j\le8),\qquad
 Z_k=G(k+4)z(k+4)\ (1\le k\le7).
\]
Determine the seven parameters $a,b,\gamma,U_1,U_2,U_3,Z_0$ by prescribing
\begin{equation}\label{n8a:eq:impose}
 (Ah)_{B_j}=\varepsilon g_j\quad(1\le j\le4),\qquad
 (Ah)_{E_k}=\varepsilon g_{k+4}\quad(0\le k\le2).
\end{equation}
The value of $\varepsilon$ prescribes a positive common residual
and leaves room for the uniform estimates below. We solve these
equations explicitly and then prove positivity of the resulting profile.

\paragraph{Scalar compatibility and interpolation.}
Eliminating the junction equations gives the following intermediate
coefficients:
\begin{align*}
 K&=\alpha_0 d_0+13,& L&=23/2+2d_0,\\
 C&=(1+\alpha_0 g_4-14g_5)/K,&
 D&=-23g_5/2+2g_4+12g_6-g_2,\\
 P&=23d_1/2+50-14L(\alpha_0+d_1)/K,&
 Q&=-138-12d_2+168L/K,\\
 W&=Pg_5+Qg_6+120g_7.
\end{align*}
The scalar compatibility condition is
\begin{equation}\label{n8a:eq:gamma}
 \gamma=-\frac{Pg_5q(5)+Qg_6q(6)+120g_7q(7)
                                  +\varepsilon(D-LC)}{W}.
\end{equation}
This determines $Z_1,Z_2,Z_3$. Define successively
\begin{align}
 Z_0&=\frac{14(\alpha_0+d_1)Z_1-168Z_2}{K}+\varepsilon C,
                                                      \label{n8a:eq:z0}\\
 U_1&=d_0Z_0-14Z_1-\varepsilon g_4,\notag\\
 U_2&=d_1Z_1-Z_0-12Z_2-\varepsilon g_5,\notag\\
 U_3&=d_2Z_2-2Z_1-10Z_3-\varepsilon g_6,\notag\\
 U_4&=(13U_3-3U_2-3Z_2-\varepsilon g_3)/10,\notag\\
 U_5&=(55U_4/4-4U_3-4Z_3-\varepsilon g_4)/8.\notag
\end{align}
Finally, with $v_4=U_4/g_4-r(4)$ and $v_5=U_5/g_5-r(5)$, take
\begin{equation}\label{n8a:eq:ab}
 a=5v_5-4v_4,\qquad b=20(v_4-v_5).
\end{equation}
Then $a+b/4=v_4$ and $a+b/5=v_5$, so the definitions of
$U_4,U_5$ agree exactly with $G(j)p(j)$.

\begin{lemma}\label{n8a:lem:join}
Formulas \eqref{n8a:eq:gamma}--\eqref{n8a:eq:ab} are well defined,
give a positive profile and satisfy \eqref{n8a:eq:impose}.
\end{lemma}
\begin{proof}
The three equations on the $E$-chain in \eqref{n8a:eq:impose} are
the definitions of $U_1,U_2,U_3$. The first Boolean equation is
\[
 \alpha_0 U_1-14U_2-Z_0=\varepsilon.
\]
Substitution of $U_1,U_2$ gives \eqref{n8a:eq:z0}, since the
coefficient of $Z_0$ is $K=\alpha_0 d_0+13$.
The second Boolean equation becomes
\[
 -LZ_0+(23d_1/2+50)Z_1-(138+12d_2)Z_2+120Z_3
                                      +\varepsilon D=0.
\]
Using \eqref{n8a:eq:z0}, this reduces to
$PZ_1+QZ_2+120Z_3+\varepsilon(D-LC)=0$, which is
\eqref{n8a:eq:gamma}. The third and fourth Boolean equations define
$U_4,U_5$. These are algebraic identities for arbitrary data with
$K,W\ne0$; positivity of their residuals is imposed by construction.

For the chosen field, the three coefficients are
\[
 d_0=\frac{4629944721}{346191500},\quad
 d_1=\frac{5660679880886737}{426558449626885},\quad
 d_2=\frac{19212352159488483}{1485092730366850}.
\]
The explicit formulas give
\begin{equation}\label{n8a:eq:parameters}
 \begin{gathered}
 106<K<107,\qquad -0.51<W<-0.50,\\
 1.03376322<a<1.03376323,\qquad
 0.48005477<b<0.48005478,\\
 0.968938202<\gamma<0.968938203.
 \end{gathered}
\end{equation}
These bounds concern constants defined by explicit rational
quotients; the preceding formulas specify every input.
In particular, $0.96<\gamma<0.98$, and the same formulas give
\[
 0.22<Z_0<0.25,\quad 0.150<Z_1<0.153,\quad
 0.106<Z_2<0.109,\quad 0.081<Z_3<0.083.
\]
Also $d_0,d_1>13$, $d_2>12$ and
$\varepsilon g_j<0.001$ for $4\le j\le6$.
The definitions therefore imply
\[
 U_1>13(0.22)-14(0.153)-0.001>0,
\]
\[
 U_2>13(0.150)-0.25-12(0.109)-0.001>0,
\]
\[
 U_3>12(0.106)-2(0.153)-10(0.083)-0.001>0.
\]
For the remaining values, $G>0$ and
\[
 p(s)>1.03-1.4824723/16-17.4876379/256>0\quad(4\le s\le8),
\]
\[
 z(s)>0.96+2.0555724/11-15.9024312/25-211.5171688/625>0
                 \quad(5\le s\le11).
\]
This proves positivity of the profile, including $U_4,U_5$
by the interpolation identities.
\end{proof}

\subsubsection{Uniform estimates along the two chains}\label{n8a:sec:chains}
Let $B=\mathcal B_{8,9}$ and $E=\mathcal E_{8,9}$ be the residual
functions in \eqref{eq:generalchainresiduals}, with $\nu=\nu_8=3/2$,
exterior field $\beta_8(s,t)$ and the functions $p,z$ just constructed.
By the recurrence for $G$,
\[
 (Ah)_{B_j}=G(j)B(j)\quad(5\le j\le8),\qquad
 (Ah)_{E_k}=G(k+4)E(k+4)\quad(3\le k\le7).
\]

\begin{lemma}\label{n8a:lem:tails}
For the parameters just constructed,
\[
 B(s)>0\quad(5\le s\le8),\qquad E(s)>0\quad(7\le s\le11).
\]
\end{lemma}
\begin{proof}
Appendix~\ref{n8a:app:interior} proves both estimates by completing
squares in the leading Taylor terms and bounding the full remainders
by Cauchy's estimate. The bounds hold on the displayed real intervals.
\end{proof}

\subsubsection{The global field and the first-sphere reserve}

\begin{proof}[Proof of Theorem~\ref{thm:N8}]
Lemmas~\ref{n8a:lem:join} and \ref{n8a:lem:tails} give $h>0$ and
$Ah>0$ on $\mathcal P$. Proposition~\ref{n8a:prop:exterior} gives a
strictly positive exterior remainder. Hence
Proposition~\ref{prop:analyticcompletion}, with
$\eta=\delta\mathbf1_\Sigma$, applies to
$\widehat\beta=\mathscr P_{\mathcal P,h}\beta_8$.
Its remainder satisfies \eqref{eq:N8globalW}. More precisely,
\eqref{eq:generalcompletedremainder}, \eqref{n8a:eq:impose} and $g_1=1$
give
\begin{equation}\label{eq:N8actualreserve}
 W_8|_\Sigma=\frac1{119}+\frac{\varepsilon}{U_1}>c_*.
\end{equation}
The same proposition gives \eqref{eq:N8globalcoercivity} and strictness
for every nonzero competitor. Proposition~\ref{prop:scaling} supplies
$C(8)\le9$, completing the proof.
\end{proof}

The dimension-dependent input consists of the field
\eqref{eq:gaugefield8}, the explicit junction parameters and the interval
bounds proved in Appendix~\ref{app:fielddata}. The assembly uses only
the general completion principle. The reserve $c_*$ is a lower bound;
\eqref{eq:N8actualreserve} records the actual first-sphere remainder.

\section{Exterior optimization and the eight--nine transition}\label{sec:interface}

The defect of the critical power compares dimensions, while
nonnegativity of the operator depends on all possible profiles.
We now allow the exterior profile to adjust to prescribed values on
the first sphere. The resulting decrease in energy has an exact
expression in terms of the residual of the original profile. This
connects the explicit defect with the optimized stability margin.
The exterior remains subcritical in both dimensions eight and nine;
the margin changes sign through its coupling to the first sphere.
The source limits and energy spaces below make this minimization
precise even when the optimal exterior profile is not in $\ell^2$.

\subsection{The first sphere and its exterior response}

The first sphere $\Sigma_N=\{\pm e_j:1\le j\le N\}$ has no internal edges.
Its diagonal in $H_N$ is
\[
 \delta_N=2N-A_N.
\]
Let $Y_N=X_N\setminus\Sigma_N$, let $K_N$ denote the compression of $H_N$
to $\ell^2(Y_N)$, and define the adjacency map
\[
 (B_Nw)(x)=\sum_{a\in\Sigma_N,\,a\sim x}w(a).
\]
Then the following block identity is exact:
\begin{equation}\label{eq:structuralblock}
 H_N=\begin{pmatrix}\delta_NI&-B_N^*\\-B_N&K_N\end{pmatrix}.
\end{equation}
If $K_N\ge0$ and its source response at zero is finite, minimization over
the exterior subtracts the form
\[
 T_N=\lim_{z\downarrow0}B_N^*(K_N+z)^{-1}B_N
\]
from the first-sphere diagonal. Thus the effective interface form is
$\delta_NI-T_N$, with $T_N$ defined by the indicated source limit.

For $3\le N\le8$, positivity and Green finiteness follow from
Lemma~\ref{lem:N8subcritical}. Proposition~\ref{prop:K9} will establish
the same exterior properties in dimension nine. The sign of the
effective interface form can therefore be compared on both sides of
the transition.

Some elementary geometric quantities illustrate the competition. For the
normalized constant vector
$e_N=(2N)^{-1/2}\mathbf1_{\Sigma_N}$, let $b_N=B_Ne_N$.
There are $2N(N-1)$ points with squared norm two, each adjacent to two
points of $\Sigma_N$, and $2N$ points of the form $\pm2e_j$, each adjacent
to one. Hence
\[
 \|b_N\|^2=2N(N-1)\frac4{2N}+2N\frac1{2N}=4N-3.
\]
In the two dimensions at the transition:
\begin{center}
\begin{tabular}{@{}ccc@{}}
\toprule
$N$&$\delta_N=2N-A_N$&$\|b_N\|^2$\\
\midrule
8&$7$&$29$\\
9&$23/4$&$33$\\
\bottomrule
\end{tabular}
\end{center}
The first-sphere diagonal decreases while the source norm increases.
The degree grows linearly with $N$ and the critical coefficient quadratically.
The exact balance is determined by the full source response, which
depends on $K_N$. Proposition~\ref{prop:K9} establishes the exterior
positivity needed for the comparison in dimension nine.

A negative prescribed-profile defect detects instability. To obtain the
positivity criterion, we minimize over the exterior. The following
identities quantify this adjustment; its dimension-specific bounds are
proved below.

\subsection{Energy-space reduction and residual identities}\label{subsec:energytools}

The exact response is a trial variational value plus the squared dual
norm of its residual. This identity gives an exact error estimate in
the energy space. The following variational and Schur reductions are
independent of the lattice dimension.

The energy space allows exterior profiles to be limits in quadratic
energy even when their square-summable norm diverges. We record the
corresponding Hilbert-space identities. In this subsection spaces and
pairings are real;
the complex identities follow by taking real parts. Let $K$ be a
bounded, nonnegative, injective self-adjoint operator on $\ell^2(Y)$,
where $Y$ is countable, and let $\mathscr H_K$ be the completion of
$C_c(Y)$ in $\|u\|_K^2=Q_K[u]:=\langle u,Ku\rangle$;
$Q_K(g,u)$ denotes the associated bilinear form.
A source $f$ is energy-continuous when
\[
 \|f\|_{K,*}^2=
 \sup_{0\ne u\in C_c(Y)}\frac{|\langle f,u\rangle|^2}{Q_K[u]}<\infty.
\]
For $g\in\mathscr H_K$, $Kg$ denotes the energy functional
$u\mapsto Q_K(g,u)$; this notation does not require $g\in\ell^2(Y)$.
The next identity gives both a variational lower bound and its exact error.

\begin{proposition}\label{prop:energyresidual}
Let $f\in C_c(Y)$ be energy-continuous. Its Riesz representative
$v_f\in\mathscr H_K$ satisfies
\begin{equation}\label{eq:energyresponsegeneral}
 m_f:=\lim_{\varepsilon\downarrow0}
 \langle f,(K+\varepsilon)^{-1}f\rangle
 =\|f\|_{K,*}^2=Q_K[v_f].
\end{equation}
For every $g\in\mathscr H_K$,
\begin{equation}\label{eq:residualidentitygeneral}
 m_f=2\langle f,g\rangle-Q_K[g]+\|Kg-f\|_{K,*}^2.
\end{equation}
On every finite-dimensional subspace $U\subset\mathscr H_K$, the
variational maximizer $g_U$ exists and satisfies
$Q_K(g_U,u)=\langle f,u\rangle$ for all $u\in U$. Moreover,
\[
 m_f=Q_K[g_U]+\|Kg_U-f\|_{K,*}^2.
\]
\end{proposition}
\begin{proof}
The resolvent variational formula, with $C_c(Y)$ dense in $\ell^2(Y)$,
is
\[
 \langle f,(K+\varepsilon)^{-1}f\rangle
 =\sup_{u\in C_c(Y)}
 \{2\langle f,u\rangle-Q_K[u]-\varepsilon\|u\|_2^2\}.
\]
As $\varepsilon\downarrow0$ the limit is the supremum without the
last term: one inequality is immediate, and the other follows by fixing
$u$ before taking the limit. Riesz representation and completion of the
square identify that supremum with $Q_K[v_f]$.
The functional $Kg-f$ is represented by $g-v_f$, so expansion of
$Q_K[g-v_f]$ proves \eqref{eq:residualidentitygeneral}.
Orthogonal projection of $v_f$ onto $U$ gives the maximizer and its
Euler equations. Taking $u=g_U$ proves the final formula.
\end{proof}

The response matrix records the energy gained by adjusting the
exterior for each boundary pattern. Its eigenvector directions are
the interface channels. We next formulate the interface criterion.
Let $F$ be finite, let $C$
be a real symmetric matrix on $\R^F$, and suppose
$B:\R^F\to C_c(Y)$ has energy-continuous columns. Let $\mathcal Rw$
be the Riesz representative of $Bw$ and define $T$ by
$\langle w,Tw\rangle=Q_K[\mathcal Rw]$.

\begin{proposition}\label{prop:abstractschur}
For the bounded block operator
$\mathcal H=\left(\begin{smallmatrix}C&-B^*\\-B&K\end{smallmatrix}\right)$,
one has, for $w\in\R^F$ and $u\in C_c(Y)$,
\begin{equation}\label{eq:abstractschuridentity}
 Q_{\mathcal H}[w,u]
 =Q_K[u-\mathcal Rw]+\langle w,(C-T)w\rangle,
 \qquad \mathcal H\ge0\ \Longleftrightarrow\ C-T\ge0.
\end{equation}
Moreover $T=\lim_{\varepsilon\downarrow0}B^*(K+\varepsilon)^{-1}B$.
If $C=dI$, $B$ is entrywise nonnegative, the resolvent of $K$ preserves
positivity, and a permutation group acts transitively on $F$ and
commutes with every $B^*(K+\varepsilon)^{-1}B$, then, for
$e=|F|^{-1/2}\mathbf1_F$,
\begin{equation}\label{eq:transitivechannel}
 \lambda_{\max}(T)=\langle e,Te\rangle=\|Be\|_{K,*}^2,
 \qquad
 \mathcal H\ge0\ \Longleftrightarrow\ \|Be\|_{K,*}^2\le d.
\end{equation}
\end{proposition}
\begin{proof}
Riesz representation gives $Q_K(\mathcal Rw,u)=\langle Bw,u\rangle$;
expanding the square proves the first identity. If $C-T\ge0$ it
implies nonnegativity. Conversely approximate $\mathcal Rw$ in energy
by finite-support functions, for each fixed $w$. Nonnegativity then
forces $\langle w,(C-T)w\rangle\ge0$.
Equation~\eqref{eq:energyresponsegeneral}, applied to $Bw$ and polarized,
gives the matrix limit. Under the additional assumptions $T$ has
nonnegative entries and all row sums agree. Their common value is
$\langle e,Te\rangle$, and $e$ is an eigenvector with this eigenvalue.
The row-sum norm bounds the absolute value of every eigenvalue by the
same number, proving \eqref{eq:transitivechannel}.
\end{proof}

Thus transitivity selects the constant stability channel without a
full decomposition into irreducible representations. Bounds on the other
channels remain useful for counting eigenvalues beyond the first
threshold. These identities isolate the exterior response; estimating
that response is the dimension-specific step.

\subsection{The energy remaining after optimization}\label{subsec:optimizeddefect}

Recall the trace-normalized infimum introduced in
Section~\ref{sec:results}:
\begin{equation}\label{eq:optimizeddefinition}
 \mathfrak D_N=\inf\bigl\{\langle u,H_Nu\rangle:
 u\in C_c(X_N;\C),\ \|P_{\Sigma_N}u\|_2^2=2N\bigr\}.
\end{equation}
The constraint fixes no exterior norm, so $\mathfrak D_N$ may be
$-\infty$ even though $H_N$ is bounded. Its normalization agrees with
$h_N=1$ on $\Sigma_N$.

\begin{proposition}\label{prop:optimizedcriterion}
For every integer $N\ge3$,
\begin{align}
 \mathfrak D_N&\le\mathfrak d_N,\label{eq:optimizedbelowdefect}\\
 \mathfrak D_N\ge0
 &\quad\Longleftrightarrow\quad H_N\ge0
 \quad\Longleftrightarrow\quad C(N)=A_N.\label{eq:optimizedcriterion}
\end{align}
Consequently $\mathfrak D_N<0$, including the value $-\infty$, is
equivalent to $C(N)<A_N$.
Suppose in addition that $K_N\ge0$ is injective, its finite sources
are energy-continuous, and $\phi_N=h_N|_{Y_N}$ belongs to its energy
completion. Put
\[
 e_N=(2N)^{-1/2}\mathbf1_{\Sigma_N},\quad b_N=B_Ne_N,\quad
 m_N(0)=\lim_{\varepsilon\downarrow0}
 \langle b_N,(K_N+\varepsilon)^{-1}b_N\rangle,
\]
and set $r_N=K_N\phi_N-\sqrt{2N}\,b_N$. Then
\begin{equation}\label{eq:optimizedresponse}
 \mathfrak D_N
 =2N\bigl(2N-A_N-m_N(0)\bigr)
 =\mathfrak d_N-\|r_N\|_{K_N,*}^2.
\end{equation}
No assumption on the sign of $H_N$ is needed for this last identity.
\end{proposition}
\begin{proof}
The logarithmic cutoffs $h_N\chi_R$ in
\eqref{eq:defectgroundstate} have the prescribed trace for large $R$,
and their energies converge to $\mathfrak d_N$. This proves
\eqref{eq:optimizedbelowdefect}. Nonnegativity of $H_N$ implies
$\mathfrak D_N\ge0$. Conversely, a negative finite-support test with
nonzero trace can be rescaled to the constraint in
\eqref{eq:optimizeddefinition}. If its trace is zero, adding a sufficiently
small nonzero function on $\Sigma_N$ leaves its energy negative and
makes its trace nonzero. Thus $\mathfrak D_N\ge0$ excludes every
negative test. Boundedness and density, followed by
Proposition~\ref{prop:spectraltransition}, prove
\eqref{eq:optimizedcriterion}.

Under the additional hypotheses, let $v_N$ be the Riesz representative
of $b_N$ in $\mathscr H_{K_N}$. Minimizing the block form over the
exterior gives the matrix $(2N-A_N)I-T_N$, where
$T_N=\lim_{\varepsilon\downarrow0}B_N^*(K_N+\varepsilon)^{-1}B_N$.
The resolvent of $K_N$ preserves positivity: a sufficiently large
scalar shift makes its negative an entrywise nonnegative matrix, and
the exponential series and its Laplace integral give the assertion.
Signed-permutation symmetry is transitive on $\Sigma_N$.
Proposition~\ref{prop:abstractschur} therefore gives
$\lambda_{\max}(T_N)=m_N(0)$, with constant eigenvector $e_N$.
The trace minimum is $2N(2N-A_N-m_N(0))$.

For completeness, the energy of the fixed profile satisfies
\begin{equation}\label{eq:optimizedprofileenergy}
 \mathfrak d_N=2N(2N-A_N)
 -2\sqrt{2N}\langle b_N,\phi_N\rangle
 +\mathcal Q_{K_N}[\phi_N].
\end{equation}
This follows by cutting off $\phi_N$ logarithmically in the block form.
The exterior vertex sum is absolutely convergent, and the cutoff edge
cost tends to zero. The functional $r_N$ is represented by
$\phi_N-\sqrt{2N}v_N$. Expanding its squared energy gives
\[
 \|r_N\|_{K_N,*}^2=
 \mathcal Q_{K_N}[\phi_N]
 -2\sqrt{2N}\langle b_N,\phi_N\rangle+2Nm_N(0).
\]
Subtraction proves \eqref{eq:optimizedresponse}.
\end{proof}

Thus a negative prescribed-profile defect always detects instability.
A positive one is conclusive only after controlling the energy saved
by exterior adjustment. The constructed spatial remainders give a positive
trace margin in dimensions three through eight. Section~\ref{subsec:K9}
shows that the first-sphere exterior remains positive in dimension nine,
where that margin has the opposite sign. The definition
\eqref{eq:optimizeddefinition} remains valid even in dimensions where
this exterior is no longer nonnegative.

\subsection{The positive margin in dimension eight}\label{sec:N8}

Write $\Sigma=\Sigma_8$, $Y=Y_8$, $H=H_8$, $K=K_8$, and
$\mathcal Q[u]=\langle u,Ku\rangle$. The exterior operator is
\begin{equation}\label{eq:N8K}
 (Ku)(x)=\left(16-\frac9{|x|^2}\right)u(x)
       -\sum_{y\in Y,\,y\sim x}u(y).
\end{equation}
For $e=\frac14\mathbf1_\Sigma$ and the adjacency map $B=B_8$,
\begin{equation}\label{eq:N8source}
 b=Be=\frac12\mathbf1_{\{|x|^2=2\}}+
       \frac14\mathbf1_{\{\pm2e_i\}},\qquad
 m(z)=\langle b,(K+z)^{-1}b\rangle.
\end{equation}
The two source sets have 112 and 16 points, so $\|b\|^2=29$.
The block operator is
\begin{equation}\label{eq:N8block}
 H=\begin{pmatrix}7I&-B^*\\-B&K\end{pmatrix}.
\end{equation}

\begin{proposition}\label{prop:N8finalresponse}
The source response is finite and
\begin{equation}\label{eq:N8finalnumber}
 0<m(0)\le7-c_*<7,
 \qquad
 c_*\le7-m(0)<7.
\end{equation}
\end{proposition}
\begin{proof}
Lemma~\ref{lem:N8subcritical} gives Green finiteness and injectivity
of $K$. The remainder \eqref{eq:hardyremainder} implies
$H_8\ge c_*P_\Sigma$.
Apply Proposition~\ref{prop:abstractschur} to
$H_8-c_*P_\Sigma$, whose exterior block is still $K$.
The constant channel gives $m(0)\le7-c_*$.
Strict positivity of $m(0)$ follows from $b\ne0$ and positivity of
the resolvent quadratic form.
\end{proof}

\subsection{Dimension nine: a stable exterior and an unstable coupling}
\label{subsec:K9}

The first-sphere exterior remains positive at the first nonsharp
dimension. The comparison with eight therefore uses the same effective
interface quantity in a positive exterior energy space.

\begin{proposition}\label{prop:K9}
The compression $K_9$ on $Y_9=\{x\in\Z^9:|x|^2\ge2\}$ has a
positive reciprocal field with vertex remainder $W_9(x)>0$ everywhere.
It is nonnegative and injective, and its zero-energy Green kernel is
finite. For the source $b_9=B_9(18^{-1/2}\mathbf1_{\Sigma_9})$,
\begin{equation}\label{eq:K9margin}
 m_9(0)>\frac{623}{108},\qquad
 \frac{23}{4}-m_9(0)<-\frac1{54},\qquad
 \mathfrak D_9<-\frac13.
\end{equation}
\end{proposition}
\begin{proof}
Use the calibrated field with $\mu=7/4$, $\kappa=17/5$:
\begin{equation}\label{eq:field9}
 \beta_9(s,t)=
 \frac{123s^2-96st+53t^2}{53s^2-96st+123t^2}.
\end{equation}
It is positive and reciprocal. The calibration in Appendix~\ref{app:calibration9}, specifically
\eqref{eq:uniformfactor}--\eqref{eq:uniformR2} applies with $N=9$.
Writing $R_9=R_0+R_1\tau+R_2\tau^2$, direct substitution gives
\begin{align*}
 R_0={}&\frac{18}{5}s+\frac{8901}{20}+\frac{41679}{800s}
 -\frac{1759023}{1600s^2}-\frac{3676347}{3200s^3}
 -\frac{13026069}{32000s^4},\\
 R_1={}&\frac{417}{5}s-\frac{882999}{200}-\frac{601101}{400s}
 +\frac{1043901}{200s^2}+\frac{13026069}{4000s^3},\\
 R_2={}&\frac{166419}{25}-\frac{499257}{200s}
 -\frac{13026069}{2000s^2}.
\end{align*}
All three increase for $s\ge12$: indeed
\[
 R_0'>\frac{18}{5}-\frac{41679}{800\cdot12^2}>3,
 \qquad
 R_1'>\frac{417}{5}-\frac{1043901}{100\cdot12^3}
             -\frac{39078207}{4000\cdot12^4}>76,
 \qquad R_2'>0.
\]
At $s=12$, their values satisfy $R_0>484$, $R_1>-3502$,
$R_2>6403$. Thus, for $s\ge12$ and $\tau\ge0$,
\[
 R_9(s,\tau)>484-3502\tau+6403\tau^2>0,
 \qquad 4\cdot484\cdot6403-3502^2>0.
\]
The affine calibration proves a positive remainder in all these rows.

In the remaining layers, only five Boolean orbits need completion.
Appendix~\ref{app:gaugecore9} constructs their weights by a tridiagonal
Poisson problem and proves positivity on every other type by three
one-variable estimates. Completion on those five orbits therefore gives
$K_9\ge M_{W_9}$ with $W_9>0$ everywhere. Inverse form order then gives
\[
 \langle\delta_x,(K_9+\varepsilon)^{-1}\delta_x\rangle
 \le W_9(x)^{-1}.
\]
Cauchy--Schwarz for the positive resolvent proves Green finiteness at
every pair, and finite sources are energy-continuous.

The profile $\phi_9=h_9|_{Y_9}$ belongs to the exterior energy
completion. To see this, apply the ground-state expansion to its
logarithmic cutoffs: the edge cutoff term is $O(1/\log R)$ and
$\phi_9K_9\phi_9=O(|x|^{-11})$ is absolutely summable.
For differences of two cutoffs the same two terms tend to zero.
Continuity of point evaluations identifies the resulting energy limit
with $\phi_9$. Proposition~\ref{prop:optimizedcriterion} now gives
\[
 \mathfrak D_9=\mathfrak d_9-\|r_9\|_{K_9,*}^2
 =18\left(\frac{23}{4}-m_9(0)\right).
\]
Lemma~\ref{mono:lem:N9base} gives $\mathfrak d_9<-1/3$.
The nonnegativity of the squared norm proves \eqref{eq:K9margin}.
\end{proof}

The response matrix has nonnegative entries and constant row sum
$m_9(0)$, so its largest eigenvalue is $m_9(0)$. Consequently the same
Schur reduction gives, for $t\ge0$,
\begin{equation}\label{eq:repulsion9}
 H_9+tP_{\Sigma_9}\ge0
 \quad\Longleftrightarrow\quad t\ge t_9,
 \qquad t_9:=m_9(0)-\frac{23}{4}>\frac1{54}.
\end{equation} Thus dimension eight tolerates additional
attraction on the first sphere, whereas dimension nine requires
repulsion there. Both thresholds are determined by the full exterior
response.

\subsection{The defect sign law and the cost of exterior adjustment}
\label{subsec:defectstructure}

The exterior hypotheses now hold on both sides of the transition.
The optimized energy and the prescribed-profile defect have the same
sign in this family, and their difference has an exact energy meaning.

\begin{proposition}\label{prop:defectschur}
For $3\le N\le9$, $K_N$ has finite Green kernel, $\phi_N=h_N|_{Y_N}$
belongs to $\mathscr H_{K_N}$, and the quantities
\begin{equation}\label{eq:defectschurdata}
 \tau_N=2N-A_N-m_N(0),\qquad
 r_N=K_N\phi_N-\sqrt{2N}\,b_N
\end{equation}
are well-defined. They satisfy
\begin{equation}\label{eq:defectschurgeneral}
 \mathfrak D_N=2N\tau_N
 =\mathfrak d_N-\|r_N\|_{K_N,*}^2.
\end{equation}
For $3\le N\le8$, $\tau_N>0$, and
\begin{equation}\label{eq:schurmarginGreen}
 \tau_N=\langle e_N,G_Ne_N\rangle^{-1}.
\end{equation}
In dimension eight,
\begin{equation}\label{eq:defectschur8}
 \mathfrak d_8=16(7-m(0))+\|r_{\mathrm{cr}}\|_{K,*}^2,
 \qquad r_{\mathrm{cr}}=K(|x|^{-3}|_Y)-4b,
\end{equation}
and
\begin{equation}\label{eq:defectadjustmentgap}
 \mathfrak d_8-\|r_{\mathrm{cr}}\|_{K,*}^2
 \ge16c_*=\frac{16}{119}.
\end{equation}
In dimension nine, $\tau_9<-1/54$.
\end{proposition}
\begin{proof}
For $N\le8$, the compression inherits the pointwise positive weight
of $H_N$; Proposition~\ref{prop:K9} handles nine. The energy convergence
of the logarithmic cutoffs of $\phi_N$ follows in every case from
$\phi_NK_N\phi_N=O(|x|^{-N-2})$ and the cutoff edge cost
$O(1/\log R)$, exactly as in that proposition. Thus
Proposition~\ref{prop:optimizedcriterion} applies and proves
\eqref{eq:defectschurgeneral}.
For $3\le N\le8$, the constructed spatial remainders give
$\mathfrak D_N\ge2N\min_{\Sigma_N}w_N>0$.
For these dimensions, signed-permutation symmetry and the Schur identity
give
\[
 \langle e_N,(H_N+\varepsilon)^{-1}e_N\rangle
 =\bigl(2N-A_N+\varepsilon-m_N(\varepsilon)\bigr)^{-1}.
\]
Green finiteness permits passage to zero and proves
\eqref{eq:schurmarginGreen}. Formula~\eqref{eq:ownmargin8} proves
\eqref{eq:defectadjustmentgap}, and Proposition~\ref{prop:K9} gives
the negative margin in nine. The residuals are also in $\ell^2(Y_N)$,
and the energy norm has the equivalent resolvent expression
\[
 \|r_N\|_{K_N,*}^2=
 \lim_{\varepsilon\downarrow0}
 \langle r_N,(K_N+\varepsilon)^{-1}r_N\rangle.
\]
The variational formula proves this by the same completion of squares
as Proposition~\ref{prop:energyresidual}.
\end{proof}

The fixed-trace formulation makes the energy loss explicit:
\begin{align}
 &\inf_{v\in\mathscr H_{K_N}}
 \bigl\{2N(2N-A_N)-2\sqrt{2N}\operatorname{Re}\langle b_N,v\rangle
                         +\mathcal Q_{K_N}[v]\bigr\}\notag\\
 &\hspace{22mm}=\mathfrak D_N
 =\mathfrak d_N-\|r_N\|_{K_N,*}^2,\qquad 3\le N\le9.
 \label{eq:defectoptimizedminimum}
\end{align}
The prescribed value $v=\phi_N$ gives $\mathfrak d_N$; the minimum
occurs at $v=\sqrt{2N}v_N$, where $K_Nv_N=b_N$ in energy.
The difference of these values is the squared energy of
$\phi_N-\sqrt{2N}v_N$. In particular,
\begin{equation}\label{eq:defectstrictcriterion8}
 m(0)<7\quad\Longleftrightarrow\quad
 \mathfrak d_8>\|r_{\mathrm{cr}}\|_{K,*}^2.
\end{equation}
The positivity of this margin follows from the explicit remainder,
while the identity itself depends only on the exterior energy space.

\begin{corollary}\label{cor:structuraldichotomy}
For the prescribed operators and every integer $N\ge3$,
\begin{align}
 \mathfrak d_N>0
 &\quad\Longleftrightarrow\quad\mathfrak D_N>0
 \quad\Longleftrightarrow\quad C(N)=A_N
 \quad\Longleftrightarrow\quad3\le N\le8,
 \label{eq:positivedichotomy}\\
 \mathfrak d_N<0
 &\quad\Longleftrightarrow\quad\mathfrak D_N<0
 \quad\Longleftrightarrow\quad C(N)<A_N
 \quad\Longleftrightarrow\quad N\ge9.
 \label{eq:negativedichotomy}
\end{align}
\end{corollary}
\begin{proof}
Corollary~\ref{cor:defectsigncomplete} gives the signs of
$\mathfrak d_N$. Proposition~\ref{prop:defectschur} gives
$\mathfrak D_N>0$ for $N\le8$; for $N\ge9$,
$\mathfrak D_N\le\mathfrak d_N<0$ by
\eqref{eq:optimizedbelowdefect}. Proposition~\ref{prop:optimizedcriterion}
and Proposition~\ref{prop:hardyremainder} finish the equivalences.
\end{proof}

Green finiteness supplies a lower bound at any fixed vertex.
\begin{corollary}\label{cor:defectGreencontrol}
For $3\le N\le8$,
\begin{equation}\label{eq:defectGreenlower}
 \mathfrak d_N\ge G_N(e_1,e_1)^{-1}>0.
\end{equation}
\end{corollary}
\begin{proof}
Resolvent Cauchy--Schwarz and $\varepsilon\downarrow0$ give
$|u(e_1)|^2\le G_N(e_1,e_1)\langle u,H_Nu\rangle$ on finite
support. Apply this to $h_N\chi_R$ and let $R\to\infty$ in
\eqref{eq:defectgroundstate}.
\end{proof}

The negative directions become particularly small in higher dimensions.
For a function with value $v_j$ on the Boolean orbit $\mathcal O_j^{(N)}$
and zero elsewhere, the full diagonal gives
\begin{equation}\label{eq:booleannegativeenergy}
 \langle u,H_Nu\rangle
 =\sum_j2^j\binom Nj\left(2N-\frac{A_N}{j}\right)v_j^2
 -4\sum_j2^j\binom Nj(N-j)v_jv_{j+1}.
\end{equation}
The choices $(v_1,v_2,v_3)=(1,11/48,3/64)$ in dimension ten
and $(v_1,v_2)=(1,16/95)$ in dimension eleven have energies
$-5/2$ and $-1353/38$, respectively, by direct substitution.
For every $N\ge12$, the single-point test satisfies
$\langle\delta_{e_1},H_N\delta_{e_1}\rangle
=2N-(N-2)^2/4<0$; the expression equals $-1$ at twelve and decreases
thereafter. These values are the quadratic energies of the specified
tests.

The same exterior cannot define a positive energy space in all higher
dimensions. For $N\ge20$,
\[
 \langle\delta_{e_1+e_2},K_N\delta_{e_1+e_2}\rangle
 =2N-\frac{(N-2)^2}{8}<0.
\]
Fixing any trace in \eqref{eq:optimizeddefinition} and multiplying
this zero-trace negative test by a large scalar shows
$\mathfrak D_N=-\infty$ in that range. The original $H_N$ remains
bounded; the value $-\infty$ refers to trace normalization, which
places no bound on the exterior norm.

The positive local margins now lead to a continuous perturbation
problem: increase an attraction on finitely many vertices until
nonnegativity is lost. Section~\ref{sec:criticallaws} determines the
size and spatial spread of the emerging state for a general class of
critical-tail operators. Section~\ref{sec:N8application} returns to
the first-sphere couplings in dimensions eight and nine.

\begin{proof}[Proof of Theorem~\ref{thm:main}]
Proposition~\ref{prop:infinity} gives $C_\infty(N)=A_N$.
The local decay estimate in Section~\ref{sec:defect} proves absolute
summability. Corollary~\ref{cor:defectsigncomplete} and
Theorem~\ref{mono:thm:main} prove the sign law and dimensional decrease.
Proposition~\ref{prop:optimizedcriterion} proves the general variational
criterion. Lemma~\ref{lem:N8subcritical} and
Proposition~\ref{prop:K9} give the exterior hypotheses through nine;
Proposition~\ref{prop:defectschur} gives the energy identity.
The two margins are Propositions~\ref{prop:N8finalresponse}
and~\ref{prop:K9}. Applying the transitive Schur criterion to the
shifted first-sphere diagonals proves the attraction and repulsion
statements.
\end{proof}

\section{Critical laws under local perturbations}\label{sec:criticallaws}

The previous sections identify the stability margin through the
energy on a finite set and the response of its exterior. We now
determine what happens just beyond that margin. For an operator with
a critical inverse-square tail, the new negative eigenvalue is
exponentially small, and its eigenvector spreads over an exponentially
large spatial scale. The results in this section give both constants
in the eigenvalue law and apply beyond the Hardy family.

The proof follows the response to a source of finite support as its
energy tends to zero. Uniform spatial bounds give a precise expansion
for the initially subcritical operator and locate the critical
coupling. At that coupling, the resolvent has a logarithmic singularity;
its finite part determines the multiplicative prefactor. The same
estimates yield the mass bounds and the critical return law.
All constants may depend on the fixed operator and perturbation;
uniformity refers to sources on a fixed finite set.

\subsection{Uniform low-energy response}\label{sec:lowenergy}

For a nonnegative source of mass one, the scalar Green response
approaches its positive zero-energy limit at inverse-logarithmic order.
The estimates are uniform over all such sources supported on a fixed
finite set. This uniformity is needed because the eigenvalue equation
selects an energy-dependent source. The proof proceeds from fixed-boundary matching to exterior
barriers and then to lattice mass estimates. Its hypotheses concern a
fixed operator with a critical tail.

Fixed-pole bounds provide relevant zero-energy context.
At the critical parameter $\mu=-1$, Chen--Hua--Xu
\cite[Theorem~1.1(b)]{ChenHuaXu} work on the full lattice with
$-\Delta-H_0$, where $H_0$ lies strictly below a specified optimal
Hardy weight $\overline H_0$ and
$\overline H_0-H_0=O(|x|^{-2-\varepsilon_0})$ for some
$\varepsilon_0>0$. They obtain two-sided bounds of order
$|x-\xi|^{-(N-2)/2}$ for each fixed pole $\xi$.
The estimates below assume subcriticality and the critical tail, with
a finite core and a correction of either sign. They track the energy
parameter and give a matrix expansion for sources on a fixed finite
set. The subcritical source estimates and the subsequent critical
resolvent expansion together determine the eigenvalue asymptotics.

Let $N\ge3$, $\alpha=(N-2)/2$, and $X=\Z^N\setminus D$, where $D$ is
finite and may be empty. Let $c(x,y)=c(y,x)\ge0$, $c(x,x)=0$, be
conductances of a connected graph on $X$. Assume that there is a finite set
$K\subset X$ such that
\[
 c(x,y)=\ind_{\{|x-y|=1\}}
 \quad\text{whenever at least one of }x,y\text{ is outside }K.
\]
Thus only finitely many edges and their weights may be changed. In particular
the graph has bounded degree and bounded jump rates. Write
\[
 (L_cf)(x)=\sum_{y\in X}c(x,y)(f(x)-f(y)),\qquad H=L_c+q.
\]
The real potential $q$ is arbitrary on finite sets and satisfies, for some
$R_*>0$, $0\le M<\infty$ and $\eta>0$,
\begin{equation}\label{rob:eq:tail}
 q(x)=-\alpha^2|x|^{-2}+V(x),\qquad
 |V(x)|\le M|x|^{-2-\eta}\quad (|x|\ge R_*).
\end{equation}
All norms and inner products without a specified measure are in
$\ell^2(X)$. The operator $H$ is bounded and self-adjoint. We assume
\begin{equation}\label{rob:eq:subcritical}
 H\ge0,\qquad
 G_0(o,o):=\lim_{\varepsilon\downarrow0}
 \ip{\delta_o}{(H+\varepsilon)^{-1}\delta_o}<\infty
 \quad\text{for one }o\in X.
\end{equation}
Lemma~\ref{rob:lem:green} shows that Green finiteness at one vertex is
equivalent to finiteness at every pair. This is the subcritical starting
situation.

Deleting the origin with a Dirichlet condition fits this notation: the
edges to the deleted vertex are retained as a finite diagonal contribution
to $q$. A sufficiently large $R_*$ is always chosen so that all rows of
$L_c$ at $|x|\ge R_*$ are exactly the full integer lattice rows.

Throughout, $a\asymp b$ means two positive constant bounds in the stated
limit. Constants may depend on the fixed operator, the dimension, and the
fixed finite sets.
Put $L_\varepsilon=\log(1/\varepsilon)$ for $0<\varepsilon<1$.
The symbol $\BesselK$ denotes the modified Bessel function of the second
kind and order zero.

For a finite source $f$, its zero-energy response is the pointwise
Green sum $(G_0f)(x)=\sum_yG_0(x,y)f(y)$. We allow the unit mass
of a nonnegative source to be distributed arbitrarily over a fixed
finite set.

\begin{theorem}\label{rob:thm:response}
Assume \eqref{rob:eq:tail}--\eqref{rob:eq:subcritical}. For every nonempty finite
$S\subset X$, the simplex of these sources is
\[
 \cP_S=\{f\ge0:\supp f\subset S,\ \sum_{x\in S}f(x)=1\}.
\]
There are $R,C_-,C_+,\varepsilon_0>0$ such that, simultaneously for
$f\in\cP_S$, $0<\varepsilon<\varepsilon_0$, and $r=|x|\ge R$,
\begin{equation}\label{rob:eq:profile}
 \frac{C_-}{L_\varepsilon}r^{-\alpha}\BesselK(\sqrt{2\varepsilon}\,r)
 \le ((H+\varepsilon)^{-1}f)(x)
 \le\frac{C_+}{L_\varepsilon}r^{-\alpha}\BesselK(\sqrt{\varepsilon/2}\,r).
\end{equation}
In particular, uniformly for these sources,
\begin{align}
 \norm{(H+\varepsilon)^{-1}f}^2
 &\asymp\frac1{\varepsilon L_\varepsilon^2},\label{rob:eq:norm}\\
 \ip f{G_0f}-\ip f{(H+\varepsilon)^{-1}f}
 &\asymp\frac1{L_\varepsilon}.\label{rob:eq:scalarresponse}
\end{align}
The pointwise Green limit $G_0f$ satisfies
\begin{equation}\label{rob:eq:zeroresponse}
 H(G_0f)=f,\qquad (G_0f)(x)\asymp r^{-\alpha}\quad(r\ge R),
\end{equation}
again uniformly for $f\in\cP_S$.
\end{theorem}

The uniformity on the entire simplex, including its vertices, is essential:
the source selected by the finite-dimensional spectral problem is allowed
to move with energy.
\subsubsection{Green functions and matching data}

Finiteness and strict positivity of the Green kernel on a connected
subcritical graph are part of
\cite[Theorems~5.12 and~5.16]{KPP}.
The next lemma also records the resulting uniform bounds for
nonnegative sources of mass one on fixed finite sets.

\begin{lemma}\label{rob:lem:green}
Under \eqref{rob:eq:subcritical}, every $G_0(x,y)$ is finite and strictly
positive. For fixed finite nonempty sets $S,E$ there are $0<a\le b<\infty$
such that
\begin{equation}\label{rob:eq:boundary}
 a\le ((H+\varepsilon)^{-1}f)(x)\le b
 \quad(x\in E,\ f\in\cP_S,\ 0<\varepsilon\le1).
\end{equation}
\end{lemma}
\begin{proof}
Write $p_t(x,y)=\ip{\delta_x}{e^{-tH}\delta_y}$. Choose
$a_0\ge\sup_x H(x,x)$. The matrix $a_0I-H$ has nonnegative entries;
its exponential series and connectivity give $p_t(x,y)>0$ for every
$t>0$. Since $H\ge0$, the resolvent kernel is
\[
 G_\varepsilon(x,y)=\int_0^\infty e^{-\varepsilon t}p_t(x,y)\,dt.
\]
It increases as $\varepsilon\downarrow0$. For fixed $s>0$, positivity and
the semigroup law imply
\[
 p_{t+2s}(o,o)\ge p_s(o,x)^2p_t(x,x).
\]
Integration proves $G_0(x,x)<\infty$ from $G_0(o,o)<\infty$.
Cauchy--Schwarz for the positive resolvent gives
$G_0(x,y)\le\sqrt{G_0(x,x)G_0(y,y)}<\infty$.
Strict positivity follows from $G_1(x,y)>0$. Now take
\[
 a=\min_{x\in E,y\in S}G_1(x,y),\qquad
 b=\max_{x\in E,y\in S}G_0(x,y).
\]
Every response in \eqref{rob:eq:boundary} is a convex combination of these
kernel values, so the bounds hold throughout the simplex, including its
boundary.
\end{proof}

\subsubsection{Barriers for the full discrete operator}

We give the estimates uniformly in dimension-dependent exponents. The
classical integral representation of the modified Bessel function is
recorded in~\cite[Formula~8.432.1]{GradshteynRyzhik}
and~\cite[Eq.~10.32.9]{DLMF}. For $z>0$ it gives
\begin{equation}\label{rob:eq:besselint}
 \BesselK(z)=\int_0^\infty e^{-z\cosh t}\,dt
 =e^{-z}\int_0^\infty e^{-zu}\frac{du}{\sqrt{u(u+2)}}.
\end{equation}
The second integral follows from $u=\cosh t-1$. The function is
positive and decreasing, and the modified Bessel
equation~\cite[Eq.~10.25.1]{DLMF} gives
$\BesselK''+z^{-1}\BesselK'-\BesselK=0$.

\begin{lemma}\label{rob:lem:bessel}
For $0\le j\le4$ and $z>0$,
\begin{equation}\label{rob:eq:besselderivatives}
 \frac{|\BesselK^{(j)}(z)|}{\BesselK(z)}
 \le\sum_{\ell=0}^j\binom j\ell(1/2)_\ell z^{-\ell}
 \le7(1+z^{-1})^j.
\end{equation}
In particular $0<-\BesselK'(z)/\BesselK(z)\le1+(2z)^{-1}$. Moreover,
\begin{equation}\label{rob:eq:besselasymp}
 \BesselK(z)=-\log z+O(1)\ (z\downarrow0),\qquad
 \BesselK(z)\le e^{-z}\sqrt{\pi/(2z)}\ (z>0).
\end{equation}
\end{lemma}
\begin{proof}
Let $w(u)=[u(u+2)]^{-1/2}$ and
$M_\ell(z)=\int_0^\infty u^\ell e^{-zu}w(u)\,du$.
For $\ell\ge1$, integration by parts gives
\[
 zM_\ell\le(\ell-\tfrac12)M_{\ell-1},\qquad
 (u^\ell w)'=u^{\ell-1}w
 \left(\ell-\tfrac12-\frac{u}{2(u+2)}\right).
\]
The boundary terms vanish. Thus
$M_\ell\le(1/2)_\ell z^{-\ell}M_0$. Differentiating
\eqref{rob:eq:besselint} under the integral and expanding $(1+u)^j$ proves
\eqref{rob:eq:besselderivatives}, since $(1/2)_\ell\le7$ for $\ell\le4$.
For the small-argument estimate use
\[
 \BesselK(z)=\int_1^\infty\frac{e^{-zt}}{\sqrt{t^2-1}}\,dt.
\]
The difference between this integral and
$\int_1^\infty e^{-zt}dt/t$ is bounded uniformly as $z\downarrow0$;
the difference of the integrands without the exponential is integrable.
After $s=zt$, the latter integral is $-\log z+O(1)$.
Finally $w(u)\le(2u)^{-1/2}$ in \eqref{rob:eq:besselint} proves the
large-argument bound.
\end{proof}

Let $L_N$ be the full lattice Laplacian, with diagonal $2N$.

\begin{lemma}\label{rob:lem:consistency}
For each $p>0$ there is $C_{N,p}<\infty$ such that
\begin{equation}\label{rob:eq:consistency}
 |(L_N+\Delta_{\R^N})F_{p,k}(x)|
 \le C_{N,p}(r^{-4}+k^4)F_{p,k}(x),\qquad
 F_{p,k}(x)=r^{-p}\BesselK(kr),
\end{equation}
for all $r=|x|\ge2$ and $0<k\le1$.
\end{lemma}
\begin{proof}
For $f(r)=r^{-p}\BesselK(kr)$, Leibniz' rule and
\eqref{rob:eq:besselderivatives} give, for $0\le j\le4$,
\[
 |f^{(j)}(r)|\le D_p(k+r^{-1})^j f(r)
\]
with a finite constant $D_p$, for example
$7\max_{0\le j\le4}\sum_{\ell=0}^j\binom j\ell(p)_{j-\ell}$.
Here $(p)_0=1$. The elementary bounds
\[
 |\partial_i r|\le1,\quad |\partial_i^2r|\le r^{-1},\quad
 |\partial_i^3r|\le3r^{-2},\quad |\partial_i^4r|\le12r^{-3}
\]
and the fourth-order chain rule yield
$|\partial_i^4 f(|y|)|\le34D_p(k+|y|^{-1})^4f(|y|)$.
For $y=x+te_i$, $|t|\le1$, one has $|y|\ge r/2$. Integration of the
logarithmic derivative bound in Lemma~\ref{rob:lem:bessel} gives
\[
 f(|y|)\le e\,2^{p+1/2}f(r),\qquad
 k+|y|^{-1}\le2(k+r^{-1}).
\]
The error of each central second difference relative to the second
derivative is at most $\frac1{12}$ times the supremum of the fourth
derivative on the corresponding segment. Summing over $N$ coordinates
and using $(k+r^{-1})^4\le8(k^4+r^{-4})$ proves the claim.
All estimates are uniform in the direction of $x$ and in $kr$.
\end{proof}

Set $\beta=\min\{\eta,2\}>0$. For $g(r)=\BesselK(kr)$, direct radial
differentiation, using $2\alpha=N-2$, gives
\begin{align}
 (-\Delta-\alpha^2r^{-2}+\varepsilon)(r^{-\alpha}g)
  &=(\varepsilon-k^2)r^{-\alpha}g,\label{rob:eq:radialmain}\\
 (-\Delta-\alpha^2r^{-2}+\varepsilon)(r^{-\alpha-\beta}g)
  &=(\varepsilon-k^2)r^{-\alpha-\beta}g
    +2\beta r^{-\alpha-\beta-1}g'
    -\beta^2r^{-\alpha-\beta-2}g.\label{rob:eq:radialcorrection}
\end{align}
The coefficient $\beta^2$ is independent of $N$.

\begin{lemma}\label{rob:lem:barriers}
There are $a,R,\varepsilon_0>0$ such that the positive functions
\begin{align}
 U^-_\varepsilon(x)&=r^{-\alpha}(1+ar^{-\beta})
                          \BesselK(\sqrt{2\varepsilon}\,r),\label{rob:eq:subbarrier}\\
 U^+_\varepsilon(x)&=r^{-\alpha}(1-ar^{-\beta})
                          \BesselK(\sqrt{\varepsilon/2}\,r)\label{rob:eq:superbarrier}
\end{align}
are square-summable on $r\ge R-1$ and satisfy
\begin{equation}\label{rob:eq:barriersigns}
 (H+\varepsilon)U^-_\varepsilon\le0,\qquad
 (H+\varepsilon)U^+_\varepsilon\ge0\quad(r\ge R)
\end{equation}
for $0<\varepsilon<\varepsilon_0$.
\end{lemma}
\begin{proof}
Put $C=\max\{1,C_{N,\alpha},C_{N,\alpha+\beta}\}$ and choose
\begin{equation}\label{rob:eq:choosea}
 a\ge1,\quad a\beta^2\ge3(C+M),\quad
 R-1\ge\max\{2,(2a)^{1/\beta}\}.
\end{equation}
Enlarge $R$ to include all exceptional rows and the set where the tail
condition is not prescribed. For either sign in $(1\pm ar^{-\beta})F$
and $F=r^{-\alpha}g$, Lemma~\ref{rob:lem:consistency} bounds the discrete
error, plus the added potential $V$, by
\begin{equation}\label{rob:eq:totalerror}
 \frac32\big[(C+M)r^{-2-\beta}+Ck^4\big]F.
\end{equation}
Indeed $ar^{-\beta}\le1/2$, $r^{-4}\le r^{-2-\beta}$, and
$r^{-2-\eta}\le r^{-2-\beta}$.

For $k^2=2\varepsilon$, \eqref{rob:eq:radialmain}--\eqref{rob:eq:radialcorrection}
and $g'<0$ bound the continuous part of $U^-_\varepsilon$ above by
$[-\varepsilon-a\beta^2r^{-2-\beta}]F$.
Adding \eqref{rob:eq:totalerror} gives the upper bound
\[
 [-\varepsilon+6C\varepsilon^2
       -\tfrac12a\beta^2r^{-2-\beta}]F<0
\]
when $\varepsilon\le1/(12C)$. For $k^2=\varepsilon/2$, the continuous
part of $U^+_\varepsilon$ is at least
$[\varepsilon/4+a\beta^2r^{-2-\beta}]F$.
Subtracting \eqref{rob:eq:totalerror} leaves at least
\[
 [\varepsilon/4-3C\varepsilon^2/8
       +\tfrac12a\beta^2r^{-2-\beta}]F>0.
\]
Choose $\varepsilon_0\le\min\{1/2,1/(12C)\}$ so that also $k\le1$.
The large-argument Bessel bound proves square-summability.
\end{proof}

The continuous identities supply comparison functions for the exact
lattice operator. The estimate \eqref{rob:eq:totalerror} absorbs the
spatial and energy-dependent errors uniformly over the exterior,
including the $k^4$ term at arbitrarily large radii.

\begin{proof}[Proof of \eqref{rob:eq:profile}]
Enlarge the matching radius so that $S\subset\{|x|<R-1\}$ and put
\[
 \Omega=\{x\in X:|x|\ge R\},\quad
 \Gamma=\{y\in X\setminus\Omega:c(x,y)>0\text{ for some }x\in\Omega\}.
\]
This is a finite boundary contained in $R-1\le |y|<R$.
On $\Gamma$, Lemma~\ref{rob:lem:green} bounds every
$v_{\varepsilon,f}=(H+\varepsilon)^{-1}f$ between two fixed positive
constants. Meanwhile
$\BesselK(\sqrt{b\varepsilon}|y|)/L_\varepsilon\to1/2$ uniformly on this
finite boundary for $b=2,1/2$. Thus fixed constants $d_-,d_+>0$ give
\[
 \frac{d_-}{L_\varepsilon}U^-_\varepsilon
 \le v_{\varepsilon,f}\le
 \frac{d_+}{L_\varepsilon}U^+_\varepsilon\quad\text{on }\Gamma
\]
for every source in the simplex and sufficiently small $\varepsilon$.

These inequalities extend to $\Omega$. If $w\in
\ell^2(\Omega)$ has nonnegative boundary values and
$(H+\varepsilon)w\ge0$ on $\Omega$, then the negative part $w_-$,
extended by zero outside $\Omega$, satisfies
\[
 0\le\ip{(H+\varepsilon)w}{w_-}_{\ell^2(\Omega)}
 \le-\ip{(H_\Omega+\varepsilon)w_-}{w_-}
 \le-\varepsilon\norm{w_-}^2.
\]
Here $H_\Omega$ is the compression retaining the full diagonal, so
$H_\Omega\ge0$. The middle inequality uses the nonpositive off-diagonal
entries and the sign of the boundary data. All pairings are legitimate
because $H$ is bounded and $\Gamma$ is finite. Apply this argument to the
two differences. Both barriers belong to $\ell^2(\Omega)$ at positive
energy, so the negative-part argument applies on the whole exterior.
Finally discard $1+ar^{-\beta}\ge1$ in the lower bound and
$1-ar^{-\beta}\le1$ in the upper bound.
\end{proof}

\subsubsection{Mass and zero-energy response}

Write $\omega_{N-1}=|\mathbb S^{N-1}|$.

\begin{lemma}\label{rob:lem:mass}
For fixed $b>0$ and $R>0$,
\begin{equation}\label{rob:eq:masssum}
 \lim_{\varepsilon\downarrow0}\varepsilon
 \sum_{\substack{x\in X\\|x|\ge R}}
 |x|^{-2\alpha}\BesselK(b\sqrt\varepsilon|x|)^2
 =\omega_{N-1}\int_0^\infty t\BesselK(bt)^2\,dt
 =\frac{\omega_{N-1}}{2b^2}.
\end{equation}
The same convergence holds when scaled radii are restricted to $(0,s]$
or $[s,\infty)$ for each fixed $s>0$, with the corresponding integrals.
\end{lemma}
\begin{proof}
With $h=\sqrt\varepsilon$, the expression is the Riemann sum
\[
 h^N\sum F(hx),\qquad F(z)=|z|^{-N+2}\BesselK(b|z|)^2.
\]
On compact annuli away from zero, ordinary Riemann convergence applies.
The lattice point count $\#\{|x|\le T\}\le C_NT^N$ for $T\ge1$ and
\eqref{rob:eq:besselasymp} give, by dyadic annuli,
\[
 h^N\sum_{R\le|x|\le s/h}F(hx)
 \le C\sum_{j\ge0}(2^{-j}s)^2(1+|\log s|+j)^2
 \le C's^2(1+|\log s|)^2
\]
for small $s$, uniformly in $h$. Empty annuli below the lattice scale
contribute zero. The exponential Bessel bound gives a uniformly vanishing
tail as $s\to\infty$. This proves all asserted convergence statements.
Removing finitely many vertices does not affect the limit.

To evaluate the integral, Tonelli and \eqref{rob:eq:besselint} give
\[
 \int_0^\infty t\BesselK(t)^2\,dt
 =\int_0^\infty\int_0^\infty
        (\cosh u+\cosh v)^{-2}\,du\,dv.
\]
With $a=(u+v)/2$, $b_1=(u-v)/2$, this becomes
\[
 \frac12\int_0^\infty\frac{da}{\cosh^2a}
       \int_{-a}^a\frac{db_1}{\cosh^2b_1}
 =\int_0^\infty\frac{\tanh a}{\cosh^2a}\,da=\frac12.
\]
Rescaling proves \eqref{rob:eq:masssum}.
\end{proof}

\begin{proof}[Completion of Theorem~\ref{rob:thm:response}]
Square \eqref{rob:eq:profile} and apply Lemma~\ref{rob:lem:mass}. The finite
remaining set has uniformly bounded response by Lemma~\ref{rob:lem:green}.
This proves \eqref{rob:eq:norm}. For a fixed source the spectral theorem gives
\[
 -\frac d{ds}\ip f{(H+s)^{-1}f}=\norm{(H+s)^{-1}f}^2.
\]
Integrate from $s=\rho>0$ to $s=\varepsilon$ and then let
$\rho\downarrow0$. Finiteness of the Green pairing and
\[
 \int_0^\varepsilon\frac{ds}{s\log^2(1/s)}
 =\frac1{\log(1/\varepsilon)}
\]
prove \eqref{rob:eq:scalarresponse}, with uniform constants.

Monotonicity and finiteness give the pointwise limit $G_0f>0$.
Passing to the limit in the finite row of
$(H+\varepsilon)v_{\varepsilon,f}=f$ gives $H(G_0f)=f$.
For fixed $r$, the Bessel factor divided by $L_\varepsilon$ tends to
$1/2$. Passing to the limit in \eqref{rob:eq:profile} proves
\eqref{rob:eq:zeroresponse}, with constants valid at every $r\ge R$.
\end{proof}

The effective radial measure is independent of $N$:
\begin{equation}\label{rob:eq:effective2}
 r^{N-1}r^{-2\alpha}\,dr=r\,dr.
\end{equation}
This identity explains the order-zero Bessel function and the universal
form of the rate. Lemmas~\ref{rob:lem:barriers} and \ref{rob:lem:mass}
establish the corresponding lattice estimates.

\subsection{The exact finite-source coefficient}\label{sec:exactresponse}

The leading Green defect is a rank-one matrix with a coefficient fixed
by the source amplitudes at infinity. The proof uses two ingredients: an exterior
oscillation estimate gives each limiting amplitude, and a logarithmic
capacity calculation fixes the common constant. Write
$\omega_{N-1}=|\mathbb S^{N-1}|$.

\begin{theorem}\label{rob:thm:exactresponse}
Under \eqref{rob:eq:tail}--\eqref{rob:eq:subcritical}, for each vertex
$y\in X$ there is an amplitude
\begin{equation}\label{rob:eq:greenamplitude}
 a(y)=\lim_{|x|\to\infty}|x|^\alpha G_0(x,y)>0.
\end{equation}
For every fixed finite $S\subset X$, some $\sigma>0$ satisfies
\[
 G_0(x,y)=|x|^{-\alpha}\bigl(a(y)+O_S(|x|^{-\sigma})\bigr)
 \quad(y\in S).
\]
Let $a_S=(a(y))_{y\in S}$ and let $P_S$ denote restriction to $S$
and extension by zero, as appropriate. In matrix operator norm,
\begin{equation}\label{rob:eq:exactmatrix}
 P_S\bigl[G_0-(H+\varepsilon)^{-1}\bigr]P_S
 =\frac{2\omega_{N-1}}{L_\varepsilon}\,a_Sa_S^*
   +O_S(L_\varepsilon^{-2}).
\end{equation}
In particular, for a fixed nonnegative finite source $f\ne0$, put
$a(f)=\sum_y a(y)f(y)$. Then
\begin{equation}\label{rob:eq:exactscalar}
 \ip f{[G_0-(H+\varepsilon)^{-1}]f}
 =\frac{2\omega_{N-1}a(f)^2}{L_\varepsilon}
  +O_f(L_\varepsilon^{-2}).
\end{equation}
The matrix statement makes the error uniform on bounded sets of sources
supported on $S$, including signed and energy-dependent sources.
\end{theorem}

\subsubsection{A limiting amplitude without radial symmetry}

A bounded positive profile relative to the critical power has a limit
at infinity. The following lemma strengthens two-sided comparison to
an amplitude with a power-decaying error.

\begin{lemma}\label{rob:lem:amplitude}
Suppose the conductances and potential satisfy \eqref{rob:eq:tail}.
If $Hu=0$ outside a finite set and
$0<c\le |x|^\alpha u(x)\le C$ there, then, for some $a>0$ and
$\sigma>0$,
\begin{equation}\label{rob:eq:amplituderate}
 u(x)=|x|^{-\alpha}\bigl(a+O(|x|^{-\sigma})\bigr).
\end{equation}
\end{lemma}
\begin{proof}
Choose $\phi>0$ on $X$ equal to $r^{-\alpha}$ outside a fixed ball,
where $r=|x|$. Set $\mu=\phi^2$, $b_{xy}=c(x,y)\phi(x)\phi(y)$, and
\[
 (\mathscr L v)(x)=\mu(x)^{-1}\sum_y b_{xy}(v(x)-v(y)).
\]
The chain with generator $-\mathscr L$ has bounded jump rates and is
recurrent. Indeed, logarithmic cutoffs equal to one on $r\le R$ and
zero on $r\ge R^2$ have $b$-energy $O(1/\log R)$ and converge
pointwise to one. The electrical capacity criterion~\cite[Theorem~6.1]{KLW} therefore gives
recurrence; equivalently, these cutoffs are a null sequence for the
network form (see the graph criticality framework in \cite{KPP}).
Consequently, the chain hits any fixed nonempty finite set almost surely~\cite[Theorem~6.35]{KLW}.
A bounded $\mathscr L$-harmonic function outside such a set is its
harmonic extension from the hitting boundary, by bounded optional
stopping. In particular, its exterior range is contained in the range
of its boundary values.

Put $v=u/\phi$. Finite-difference Taylor expansion gives
\[
 H\phi/\phi=O(r^{-2-\gamma}),\qquad
 \mathscr L v=g:=-(H\phi/\phi)v=O(r^{-2-\gamma}),
 \qquad \gamma=\min\{2,\eta\}.
\]
The Taylor expansion acts on the smooth profile $\phi$, while $V$
enters through its pointwise bound. Fix $0<\beta<\gamma$. Applied
to smooth powers, the same expansion gives
\begin{equation}\label{rob:eq:amplitudebarrier}
 \mathscr L(r^{-\beta})
 =-\beta^2r^{-\beta-2}+O(r^{-\beta-4}).
\end{equation}
For large $R$, the bounded nonnegative function
$B_R(x)=(R-1)^{-\beta}-r^{-\beta}$ on $r\ge R-1$ satisfies
$\mathscr L B_R\ge(\beta^2/2)r^{-\beta-2}$ on $r>R$.
Its values farther inside the ball do not enter these rows and can be
set to zero. Write $T_R$ for the first hitting time of the ball
$\{r\le R\}$. Stopping the chain at that ball and then at a finite
exhaustion gives
\[
 \mathbb E_x\int_0^{T_R}|g(X_t)|\,dt\le C R^{-\beta}.
\]
Indeed, $|g|\le C\mathscr L B_R$ on $r>R$ and the stopped Dynkin
identity bounds the expected integral by $CB_R(x)$. Passing to the
limit first in the stopping time and then in the exhaustion uses
nonnegativity. Thus
\[
 p_R(x)=\mathbb E_x\int_0^{T_R}g(X_t)\,dt,
 \qquad p_R=0\quad(r\le R),
\]
is well-defined, solves $\mathscr Lp_R=g$ on $r>R$, and satisfies
$\|p_R\|_\infty\le CR^{-\beta}$. The equation follows either from
the first-jump decomposition or from the stopped finite-domain equations.

Let $m(R)=\inf_{r\ge R}v$, $M(R)=\sup_{r\ge R}v$, and
$O(R)=M(R)-m(R)$. The bounded harmonic function
$k=v-p_{2R}$ outside $r\le2R$ takes boundary values between $m(R)$
and $M(R)$, once $R>1$. Recurrence gives the same bounds throughout
the exterior. On the annulus $3R\le r\le5R$, a finite Harnack chain
therefore applies to the nonnegative function $k-m(R)$ and yields
\[
 \sup_{3R\le r\le5R}(k-m(R))
 \le C_H\inf_{3R\le r\le5R}(k-m(R)).
\]
Here $C_H$ is independent of $R$. To check this hypothesis, cover the
annulus by a number, depending only on $N$, of overlapping lattice
balls of radius proportional to $R$ whose doubled balls lie in $r>2R$.
On each doubled ball, division of all conductances by $R^{2-N}$
gives uniformly elliptic nearest-neighbour weights. Extend these weights
outside that ball by constant weights within the same ellipticity bounds.
The resulting lattice has uniform volume doubling and Poincar\'e constants.
Apply the elliptic consequence of \cite[Theorem~1.7]{Delmotte} to this
extension. More explicitly, for the associated transition kernel $P$,
the lazy kernel $(I+P)/2$ has the same harmonic functions and satisfies
the holding-probability hypothesis in the discrete-time formulation.
The Harnack constants depend only on the ellipticity bounds and $N$,
and hence are independent of $R$. The annulus is connected since $N\ge3$.

It follows that the oscillation of $k$ on this annulus is at most
$\vartheta O(R)$, where $0<\vartheta=1-C_H^{-1}<1$ after enlarging
$C_H$ if necessary. Apply the bounded harmonic extension property once
more, now outside $r\le4R$; its hitting boundary is in that annulus.
Adding back $p_{2R}$ gives
\begin{equation}\label{rob:eq:oscillation}
 O(4R)\le\vartheta O(R)+CR^{-\beta}.
\end{equation}
Iteration proves $O(R)=O(R^{-\sigma})$ for any
$0<\sigma<\min\{\beta,-\log\vartheta/\log4\}$.
The nested exterior ranges consequently shrink to a single limit $a$.
The positive lower bound on $v$ gives $a>0$, proving
\eqref{rob:eq:amplituderate}.
\end{proof}

The same oscillation argument will also be used for bounded Poisson
solutions of either sign. With the reference operator $\mathscr L$
above, if $v$ is bounded outside a finite set and
$\mathscr Lv=O(r^{-2-\beta_0})$ for some $\beta_0>0$, the stopped
potential estimate and \eqref{rob:eq:oscillation} give
$v=b+O(r^{-\sigma})$ for some real $b$ and $\sigma>0$.
One chooses a barrier exponent strictly below $\beta_0$; positivity
is needed only for the conclusion $b>0$, not for convergence.

\subsubsection{Logarithmic capacity with the exact constant}

The capacity associated with a quadratic energy is the least energy
needed to hold the value at a specified vertex at one. We use the
regularized energy $Q_\varepsilon$ below, which includes a positive
mass term. Its reciprocal capacity, denoted by $T_\varepsilon(o)$,
measures the largest squared point value attainable with unit
$Q_\varepsilon$-energy. The next lemma computes this reciprocal for
the network weighted by the critical profile. All edge sums count
each unoriented edge once.

\begin{lemma}\label{rob:lem:exactcapacity}
Let $h>0$ on $X$ satisfy
$h(x)=a|x|^{-\alpha}(1+O(|x|^{-\sigma}))$, with $a,\sigma>0$.
Set $d_{xy}=c(x,y)h(x)h(y)$, $\mu_h=h^2$, and
\[
 Q_\varepsilon[v]=\sum_{\{x,y\}}d_{xy}|v(x)-v(y)|^2
                  +\varepsilon\sum_x\mu_h(x)|v(x)|^2.
\]
For each fixed $o\in X$,
\begin{equation}\label{rob:eq:exactcapacity}
 T_\varepsilon(o):=
 \sup_{0\ne v\in C_c(X)}\frac{|v(o)|^2}{Q_\varepsilon[v]}
 =\frac{L_\varepsilon}{2\omega_{N-1}a^2}+O_o(1).
\end{equation}
\end{lemma}
\begin{proof}
Put $R=\varepsilon^{-1/2}$ and choose a fixed $R_0$ containing $o$
and all exceptional vertices and edges. Set $T=\log(R/R_0)$ and use
the cutoff
\[
 v_R(x)=\begin{cases}
 1,&r\le R_0,\\
 \log(R/r)/T,&R_0<r<R,\\
 0,&r\ge R.
 \end{cases}
\]
On a regular edge in direction $j$ at large radius,
$d_{x,x+e_j}=a^2r^{2-N}(1+O(r^{-\kappa}))$, where
$\kappa=\min\{1,\sigma\}>0$, and
$\log|x+e_j|-\log r=x_j/r^2+O(r^{-2})$.
Summing the squared leading differences over $j$ gives
$a^2r^{-N}/T^2$. Comparison with the integral on unit cubes yields
\[
 \sum_{R_0<|x|<R}|x|^{-N}=\omega_{N-1}\log(R/R_0)+O_{R_0}(1),
 \qquad
 \sum_{|x|>R_0}|x|^{-N-\kappa}<\infty.
\]
The edges crossing either cutoff boundary contribute $O(T^{-2})$;
the outer boundary contribution is in fact $O(R^{-1}T^{-2})$.
Consequently,
\begin{equation}\label{rob:eq:cutoffexact}
 \sum_{\{x,y\}}d_{xy}|v_R(x)-v_R(y)|^2
 =\frac{\omega_{N-1}a^2}{T}+O(T^{-2}),\qquad
 \sum_x h(x)^2v_R(x)^2=O(R^2T^{-2}).
\end{equation}
For the second bound, radial shell counting gives a constant multiple of
\[
 T^{-2}\int_{R_0}^R r\log^2(R/r)\,dr=O(R^2T^{-2}).
\]
The fixed inner ball is absorbed for large $R$.
Since $v_R(o)=1$, these estimates give
the lower bound in \eqref{rob:eq:exactcapacity}.

For the reverse bound we construct a unit flow with a diffuse terminal
mass. Choose a smooth nonincreasing cutoff $\chi$ equal to one on
$[0,1]$ and zero on $[2,\infty)$. The continuum field
\[
 F_R(z)=\frac{\chi(|z|/R)}{\omega_{N-1}}\frac{z}{|z|^N}
\]
has distributional divergence $\delta_0-\rho_R^{\rm cont}$, where
$\rho_R^{\rm cont}\ge0$, its integral is one, its support is in
$R\le|z|\le2R$, and $\rho_R^{\rm cont}\le CR^{-N}$.
On every regular oriented lattice edge, take the flux of $F_R$ across
the common face of the two unit cubes centered at its endpoints.
Outside a fixed finite core the divergence theorem gives divergence
$-\rho_R(x)$, with
\[
 \rho_R(x)=\int_{x+[-1/2,1/2]^N}\rho_R^{\rm cont}(z)\,dz,
 \quad \sum_x\rho_R(x)=1,\quad 0\le\rho_R(x)\le CR^{-N}.
\]
For large $R$, all cubes meeting this support are regular. The total
outward flux from the fixed core is one. Inside a fixed finite connected
subgraph containing the exceptional vertices and its boundary, complete
the currents along a spanning tree so that their divergence is
$\delta_o$. Such a completion exists because the prescribed divergences
sum to the outgoing flux. This adds a fixed finite energy, independent
of $R$: in the core the cutoff $\chi$ is identically one. Denote the
completed flow by $J_R$; then
$\operatorname{div}J_R=\delta_o-\rho_R$ on $X$.

In the region between the fixed core and $r<R-O(1)$, face integration
and the estimate for $d_{xy}$ give
\[
 J_R(x,x+e_j)=\frac{x_j}{\omega_{N-1}r^N}+O(r^{-N}),
 \qquad
 \sum_{j=1}^N\frac{J_R(x,x+e_j)^2}{d_{x,x+e_j}}
 =\frac{r^{-N}}{\omega_{N-1}^2a^2}+O(r^{-N-\kappa}).
\]
The sink annulus contributes $O(1)$, since each flux is
$O(R^{1-N})$ and it contains $O(R^N)$ edges. Thus
\begin{equation}\label{rob:eq:flowexact}
 \sum_{\{x,y\}}\frac{J_R(x,y)^2}{d_{xy}}
 =\frac{\log R}{\omega_{N-1}a^2}+O(1),\qquad
 \sum_x\frac{\rho_R(x)^2}{h(x)^2}=O(R^{-2}).
\end{equation}
For $v\in C_c(X)$, summation by parts and Cauchy--Schwarz in the
direct sum of the edge and vertex spaces give
\begin{align*}
 |v(o)|^2
 &=\left|\sum_{\{x,y\}}J_R(x,y)(v(x)-v(y))
                    +\sum_x\rho_R(x)v(x)\right|^2\\
 &\le\left(\sum_{\{x,y\}}\frac{J_R(x,y)^2}{d_{xy}}
       +\varepsilon^{-1}\sum_x\frac{\rho_R(x)^2}{h(x)^2}\right)
       Q_\varepsilon[v].
\end{align*}
An arbitrary orientation is fixed in each edge sum. Because
$\varepsilon R^2=1$, the factor in parentheses is
$L_\varepsilon/(2\omega_{N-1}a^2)+O(1)$. This proves the upper
bound with a bounded error, completing \eqref{rob:eq:exactcapacity}.
\end{proof}

\begin{proof}[Proof of Theorem~\ref{rob:thm:exactresponse}]
For $f\ge0$ of finite support, $f\ne0$, let $h=G_0f$.
Theorem~\ref{rob:thm:response} and Lemma~\ref{rob:lem:amplitude}
give $h=r^{-\alpha}(a(f)+O(r^{-\sigma}))$ with $a(f)>0$.
Applying this first to the finitely many point sources proves
\eqref{rob:eq:greenamplitude} and the asserted common rate on $S$.
Linearity of $G_0$ then gives $a(f)=\sum_y a(y)f(y)$.

The ground-state identity with $Hh=f$ reads, for $v\in C_c(X)$,
\[
 \ip{hv}{H(hv)}
 =\sum_{\{x,y\}}c(x,y)h(x)h(y)|v(x)-v(y)|^2
   +\sum_x f(x)h(x)|v(x)|^2.
\]
The resolvent variational principle consequently gives the exact identity
\begin{equation}\label{rob:eq:defectcapacity}
 D_\varepsilon(f):=\ip f{[G_0-(H+\varepsilon)^{-1}]f}
 =\inf_{v\in C_c(X)}
  \left\{Q_\varepsilon[v]+\sum_x f(x)h(x)|1-v(x)|^2\right\}.
\end{equation}
Indeed, substitute $u=hv$ into
$\ip f{(H+\varepsilon)^{-1}f}
=\sup_{u\in C_c(X)}(2\operatorname{Re}\ip f u-\ip u{(H+\varepsilon)u})$
and use $\ip f{G_0f}=\sum f h$. The source and test functions have
finite support, so all terms in this variational identity are finite.

Choose the fixed inner ball of $v_R$ to contain $\supp f$.
Then the last term of \eqref{rob:eq:defectcapacity} vanishes, and
\eqref{rob:eq:cutoffexact} gives
$D_\varepsilon(f)\le2\omega_{N-1}a(f)^2/L_\varepsilon
+O_f(L_\varepsilon^{-2})$.
For the reverse bound select $o$ with $p=f(o)h(o)>0$.
Lemma~\ref{rob:lem:exactcapacity} and another direct-sum
Cauchy--Schwarz inequality imply
\[
 1=|v(o)+(1-v(o))|^2
 \le\bigl(T_\varepsilon(o)+p^{-1}\bigr)
       \bigl(Q_\varepsilon[v]+p|1-v(o)|^2\bigr).
\]
Taking the infimum yields
\[
 D_\varepsilon(f)\ge
 \bigl(T_\varepsilon(o)+p^{-1}\bigr)^{-1}
 =\frac{2\omega_{N-1}a(f)^2}{L_\varepsilon}
   +O_f(L_\varepsilon^{-2}).
\]
This proves \eqref{rob:eq:exactscalar}. Finally apply that formula to
$\delta_x$, $\delta_y$ and $\delta_x+\delta_y$ for $x,y\in S$.
Polarization gives every matrix entry in \eqref{rob:eq:exactmatrix};
the finite number of entries gives its operator-norm remainder.
\end{proof}

\subsection{Eigenvalue asymptotics and spatial mass}\label{sec:birth}

The critical coupling is the reciprocal of the largest
eigenvalue of a finite positive Green matrix. Its leading channel and
the response estimates above determine the exact exponential rate,
with a bounded error in the logarithm of the energy. They also control
the threshold state and the mass of the bound state.
Section~\ref{sec:prefactor} will strengthen that bounded error to a
limit. The present subsection ends by showing that the hypotheses can
be realized for every allowed tail and connected finite core.

Let $W$ be a nonzero nonnegative multiplication operator with finite support
$F=\{x:W(x)>0\}$. Define the finite matrices and their largest eigenvalues by
\begin{equation}\label{rob:eq:BS}
 \B(\varepsilon)=W^{1/2}(H+\varepsilon)^{-1}W^{1/2}\big|_{\ell^2(F)},
 \quad \lambda(\varepsilon)=\lambda_{\max}(\B(\varepsilon)),
 \quad \tau_c=\lambda(0)^{-1}.
\end{equation}
The value at zero is the entrywise limit. All its entries are finite and
strictly positive.

\begin{theorem}\label{rob:thm:birth}
Under the assumptions of Theorem~\ref{rob:thm:response}, $0<\tau_c<\infty$,
and, among $\tau\ge0$, $H-\tau W\ge0$ precisely for $0\le\tau\le\tau_c$.
There is $\delta_0>0$ such that, for
$0<\delta=\tau-\tau_c<\delta_0$, the negative spectrum of $H-\tau W$
consists of exactly one simple eigenvalue $-\varepsilon(\tau)$, with
\begin{equation}\label{rob:eq:birth}
 \varepsilon(\tau)=\exp\!\left[-\frac{\mathfrak A(H,W)}{\delta}+O(1)\right].
\end{equation}
Let $\zeta_0>0$ be the unit Perron eigenvector of $\B(0)$ and put
$f_0=W^{1/2}\zeta_0$. With $a(f_0)$ from
Theorem~\ref{rob:thm:exactresponse},
\begin{equation}\label{rob:eq:lambdarate}
 \begin{split}
 \lambda(0)-\lambda(\varepsilon)
 &=\frac{2\omega_{N-1}a(f_0)^2}{L_\varepsilon}
    +O(L_\varepsilon^{-2}),\\
 -\lambda'(\varepsilon)&\asymp
       (\varepsilon L_\varepsilon^2)^{-1}.
 \end{split}
\end{equation}
At $\tau=\tau_c$ there is a strictly positive pointwise solution $h_c$ of
$(H-\tau_cW)h_c=0$, unique up to multiplication among positive solutions.
We measure its decay in the weighted spaces $\ell^2_{-\sigma}(X)$,
whose norm is
$\norm u_{\ell^2_{-\sigma}}^2=\sum_x(1+|x|^2)^{-\sigma}|u(x)|^2$.
For some $a_c,\sigma_0>0$, it satisfies
\begin{equation}\label{rob:eq:thresholdprofile}
 h_c(x)=a_c|x|^{-\alpha}(1+O(|x|^{-\sigma_0})),\qquad
 h_c\notin\ell^2(X),\qquad
 h_c\in\ell^2_{-\sigma}(X)\ \Longleftrightarrow\ \sigma>1,
\end{equation}
The rate in \eqref{rob:eq:birth} is
\begin{equation}\label{rob:eq:exactbirthcoefficient}
 \mathfrak A(H,W)=2\omega_{N-1}\tau_c^2 a(f_0)^2
 =\frac{2\omega_{N-1}a_c^2}{\sum_xW(x)h_c(x)^2}.
\end{equation}
The last expression is independent of the normalization of $h_c$.
The operator $H-\tau_cW$ has no
zero eigenvector in $\ell^2$ and is critical: subtracting any nonzero
nonnegative potential of finite support destroys nonnegativity.

Let $\Psi_\tau>0$ be the normalized eigenvector for $-\varepsilon(\tau)$.
For every nonempty finite $E\subset X$,
\begin{equation}\label{rob:eq:localmass}
 \sum_{x\in E}|\Psi_\tau(x)|^2
 \asymp_E\varepsilon L_\varepsilon^2
 \asymp_E\varepsilon/\delta^2.
\end{equation}
For $0<\theta<1$, define
\[
 \cL_\theta(\tau)=\inf\{R\ge0:
       \sum_{|x|\le R}|\Psi_\tau(x)|^2\ge\theta\}.
\]
Then
\begin{equation}\label{rob:eq:radius}
 \cL_\theta(\tau)\asymp_\theta\varepsilon^{-1/2},\qquad
 \log\cL_\theta(\tau)=\frac{\mathfrak A(H,W)}{2\delta}+O_\theta(1).
\end{equation}
Finally,
\begin{equation}\label{rob:eq:derivativebirth}
 \varepsilon'(\tau)=\ip{\Psi_\tau}{W\Psi_\tau}
 \asymp\varepsilon L_\varepsilon^2.
\end{equation}
\end{theorem}

This theorem gives the exact exponential rate and a bounded remainder
in its logarithm. Section~\ref{sec:prefactor} proves convergence of
that remainder and identifies the positive multiplicative prefactor.
The mass and derivative estimates are two-sided bounds for each fixed
operator and dimension.

\subsubsection{The first spectral channel}

Positivity isolates one simple spectral channel at the first
threshold. The matrix Green expansion determines its displacement
from zero energy.

\begin{proof}[Proof of the spectral assertions in Theorem~\ref{rob:thm:birth}]
The matrices $\B(\varepsilon)$ are positive definite and have strictly
positive entries. They converge in matrix norm to $\B(0)$.
Perron--Frobenius gives a simple largest eigenvalue and a positive unit
eigenvector $\zeta_\varepsilon$, chosen continuously down to zero.
For positive energy it is differentiable and Hellmann--Feynman yields
\begin{equation}\label{rob:eq:HF}
 -\lambda'(\varepsilon)
 =\norm{(H+\varepsilon)^{-1}W^{1/2}\zeta_\varepsilon}^2.
\end{equation}
Let $f_\varepsilon=W^{1/2}\zeta_\varepsilon$ and
$a_\varepsilon=\norm{f_\varepsilon}_1$. Since $W$ is strictly positive
on $F$ and $\zeta_\varepsilon\ge0$ has norm one,
\begin{equation}\label{rob:eq:sourcenormalization}
 \sqrt{\min_F W}\le a_\varepsilon
 \le\sqrt{\sum_F W}.
\end{equation}
Thus $f_\varepsilon/a_\varepsilon\in\cP_F$, and the uniform estimate
\eqref{rob:eq:norm} applies to the moving source. This proves the derivative
bound in \eqref{rob:eq:lambdarate}. For its first assertion, use
\eqref{rob:eq:exactmatrix} with $S=F$ and set
$b_F=W^{1/2}a_F$. It gives
\[
 \B(\varepsilon)=\B(0)
  -\frac{2\omega_{N-1}}{L_\varepsilon}b_Fb_F^*
  +O(L_\varepsilon^{-2}).
\]
Since $\lambda(0)$ is simple, the finite-dimensional eigenvalue
perturbation formula has a quadratic error in the matrix increment.
It yields the first line of \eqref{rob:eq:lambdarate}, because
$\ip{\zeta_0}{b_F}=a(f_0)$.

The finite-rank Birman--Schwinger reduction follows from
the form factorization
\[
 H-\tau W+\varepsilon
 =(H+\varepsilon)^{1/2}
 \big[I-\tau(H+\varepsilon)^{-1/2}W(H+\varepsilon)^{-1/2}\big]
 (H+\varepsilon)^{1/2}.
\]
The nonzero eigenvalues of the intervening finite-rank operator are those
of $\B(\varepsilon)$. Hence the number of eigenvalues below
$-\varepsilon$ equals the number of eigenvalues of $\B(\varepsilon)$
above $1/\tau$. This also identifies multiplicities at equality.
The function $\lambda(\varepsilon)$ is strictly decreasing, tends to
$\lambda(0)>0$ at zero and to zero at infinity. Therefore nonnegativity
holds exactly up to $\tau_c$.

Let $\lambda_2(0)$ be the second eigenvalue if $|F|\ge2$; set it to
zero if $|F|=1$. Perron--Frobenius gives $\lambda_2(0)<\lambda(0)$.
Choose $\delta_0>0$ so that
$(\tau_c+\delta_0)\lambda_2(0)<1$. Since
$\B(\varepsilon)\le\B(0)$ in form order, only the largest channel can
cross $1/\tau$ in this interval. There is exactly one simple negative
eigenvalue, determined by
\begin{equation}\label{rob:eq:thresholdidentity}
 \lambda(\varepsilon)=\frac1\tau,\qquad
 \lambda(0)-\lambda(\varepsilon)
       =\frac{\delta}{\tau\tau_c}.
\end{equation}
As $\tau\downarrow\tau_c$, its energy tends to zero. Equation
\eqref{rob:eq:lambdarate} first gives $L_\varepsilon\asymp\delta^{-1}$.
Put $D=2\omega_{N-1}a(f_0)^2$. Multiplication of
\eqref{rob:eq:thresholdidentity} by $L_\varepsilon$ gives
\[
 \frac{\delta L_\varepsilon}{\tau\tau_c}
 =D+O(L_\varepsilon^{-1}),\qquad
 L_\varepsilon=\frac{D\tau_c^2}{\delta}+O(1).
\]
In the second relation we used $\tau=\tau_c+\delta$ and the preceding
comparability; this proves \eqref{rob:eq:birth}.

All negative spectrum in this argument is discrete. Indeed a lattice with
finitely many deleted vertices and finitely many changed edges has the
same essential spectrum as the integer lattice Laplacian. One may extend
the operator by a finite-dimensional direct summand on $D$ and use finite
rank perturbation invariance. The potential tends to zero and so is a
compact multiplication operator. Consequently
$\spec_\ess(H-\tau W)=[0,4N]$.
\end{proof}

\subsubsection{The threshold state and spatial mass}

The zero-energy Green response produces the positive threshold
state. Its ground-state transform proves criticality, while the
finite-energy profile gives the mass and derivative estimates.

\begin{proof}[Threshold state, criticality, and uniqueness]
Set $f_0=W^{1/2}\zeta_0$ and $h_c=G_0f_0$. Then $h_c>0$,
$Hh_c=f_0$, and
\[
 W^{1/2}h_c=\B(0)\zeta_0=\lambda(0)\zeta_0.
\]
Thus $\tau_cWh_c=f_0$ and $(H-\tau_cW)h_c=0$ pointwise.
Theorems~\ref{rob:thm:response} and~\ref{rob:thm:exactresponse}
give \eqref{rob:eq:thresholdprofile}, with $a_c=a(f_0)$ for this
normalization. Moreover, $\sum W h_c^2=\lambda(0)^2=\tau_c^{-2}$,
which proves \eqref{rob:eq:exactbirthcoefficient} for any scalar
normalization.
For the integrability claim, the contribution of the annulus
$R\le |x|<2R$ to the weighted norm is comparable to $R^{2-2\sigma}$.
Dyadic summation converges exactly when $\sigma>1$.

Write $H_c=H-\tau_cW$. For every finitely supported $u$, the
ground-state identity, obtained by expansion, is
\begin{equation}\label{rob:eq:gst}
 \ip u{H_cu}=\frac12\sum_{x,y}c(x,y)h_c(x)h_c(y)
 \left|\frac{u(x)}{h_c(x)}-\frac{u(y)}{h_c(y)}\right|^2.
\end{equation}
Choose logarithmic cutoffs $\chi_R$ equal to one on $|x|\le R$, zero
on $|x|\ge R^2$, and equal to
$\log(R^2/|x|)/\log R$ in between. All exceptional edges are within
the region where the cutoff is constant when $R$ is large. On the other
edges in the transition region,
$|\chi_R(x)-\chi_R(y)|\le C/(|x|\log R)$, with the same bound near
its endpoints. The profile of $h_c$ and dyadic annular counting imply
\begin{equation}\label{rob:eq:nullsequence}
 \ip{h_c\chi_R}{H_c(h_c\chi_R)}
 \le\frac C{\log^2R}\sum_{R-1\le|x|\le R^2+1}|x|^{-N}
 \le\frac{C'}{\log R}\longrightarrow0.
\end{equation}
For any nonzero finite nonnegative potential $V_1$, the quantity
$\sum V_1h_c^2\chi_R^2$ tends to $\sum V_1h_c^2>0$.
Thus $H_c-V_1$ is not nonnegative. This proves criticality in the stated
sense.

If $k>0$ is another pointwise solution of $H_ck=0$, identity
\eqref{rob:eq:gst} also holds with $k$ in place of $h_c$. Apply it to
$h_c\chi_R$. By \eqref{rob:eq:nullsequence}, each fixed edge contribution
converges to zero, so $h_c(x)/k(x)=h_c(y)/k(y)$ on every edge.
Connectivity proves uniqueness up to a scalar.

Finally let $u\in\ell^2(X)$ satisfy $H_cu=0$. Finitely supported
truncations $u_n$ converge to $u$ in $\ell^2$, and, since $H_c$ is
bounded, $\ip{u_n}{H_cu_n}\to0$. Fatou in \eqref{rob:eq:gst} gives that
$u/h_c$ is constant on every edge. Since $h_c\notin\ell^2$, necessarily
$u=0$. Thus the kernel in $\ell^2(X)$ is trivial.
\end{proof}

\begin{proof}[Mass, radius, and derivative assertions]
At an energy satisfying \eqref{rob:eq:thresholdidentity}, set
$u_\varepsilon=(H+\varepsilon)^{-1}f_\varepsilon$ and
$J_\varepsilon=\norm{u_\varepsilon}^2$.
The finite matrix equation gives $Wu_\varepsilon=\lambda(\varepsilon)
f_\varepsilon$, so
\[
 (H-\tau W)u_\varepsilon=-\varepsilon u_\varepsilon,
 \qquad\Psi_\tau=u_\varepsilon/\sqrt{J_\varepsilon}.
\]
By \eqref{rob:eq:sourcenormalization}, $J_\varepsilon\asymp
(\varepsilon L_\varepsilon^2)^{-1}$. Lemma~\ref{rob:lem:green} bounds
$u_\varepsilon$ above and below by positive constants on any fixed
nonempty finite set. This proves \eqref{rob:eq:localmass}.

Normalizing \eqref{rob:eq:profile} gives constants $A_-,A_+>0$ with
\begin{equation}\label{rob:eq:normalizedprofile}
 A_-\sqrt\varepsilon r^{-\alpha}\BesselK(\sqrt{2\varepsilon}r)
 \le\Psi_\tau(x)\le
 A_+\sqrt\varepsilon r^{-\alpha}\BesselK(\sqrt{\varepsilon/2}r)
 \quad(r\ge R).
\end{equation}
The upper bound and Lemma~\ref{rob:lem:mass} show
\begin{align*}
 \limsup_{\varepsilon\downarrow0}
 \sum_{|x|\le s/\sqrt\varepsilon}|\Psi_\tau(x)|^2
 &\le A_+^2\omega_{N-1}\int_0^s t\BesselK(t/\sqrt2)^2\,dt,\\
 \limsup_{\varepsilon\downarrow0}
 \sum_{|x|\ge s/\sqrt\varepsilon}|\Psi_\tau(x)|^2
 &\le A_+^2\omega_{N-1}\int_s^\infty t\BesselK(t/\sqrt2)^2\,dt.
\end{align*}
The fixed core contributes $o(1)$. Choose a small $s_-$ making the
first integral strictly less than $\theta$, and a large $s_+$ making
the second strictly less than $1-\theta$. This proves
$s_-/\sqrt\varepsilon\le\cL_\theta\le s_+/\sqrt\varepsilon$
for small energy. Equation \eqref{rob:eq:birth} gives the logarithmic radius
statement.

Implicit differentiation in \eqref{rob:eq:thresholdidentity}, justified by
\eqref{rob:eq:HF}, yields
\[
 \varepsilon'(\tau)=\frac1{\tau^2J_\varepsilon}.
\]
Also $W^{1/2}u_\varepsilon=\lambda(\varepsilon)\zeta_\varepsilon$,
so $\ip{\Psi_\tau}{W\Psi_\tau}=\lambda(\varepsilon)^2/J_\varepsilon$.
This proves \eqref{rob:eq:derivativebirth}.
\end{proof}

\subsubsection{Admissible cores in every dimension}

The next construction verifies that the class in the general theorem
is nonempty in every dimension. It permits a prescribed connected finite
modification of the graph and a prescribed tail satisfying \eqref{rob:eq:tail}.

\begin{proposition}\label{rob:prop:nonempty}
Fix any connected finite modification of the integer lattice graph and any real
potential satisfying \eqref{rob:eq:tail}. There is a nonnegative potential
$U$ of finite support such that $H+U$ satisfies
\eqref{rob:eq:subcritical}.
\end{proposition}
\begin{proof}
Let $\beta=\min\{\eta,2\}$. For sufficiently large $a$ and then $R$,
\[
 \phi(x)=r^{-\alpha}(1-ar^{-\beta})>0\quad(r\ge R-1)
\]
is a strict supersolution on $r\ge R$.
Indeed, at energy zero the continuous expression equals
$a\beta^2 r^{-\alpha-\beta-2}$, by
\eqref{rob:eq:radialmain}--\eqref{rob:eq:radialcorrection} with $g=1$.
The fourth derivative argument of Lemma~\ref{rob:lem:consistency}, now for
powers alone, bounds the lattice error by
$C(r^{-\alpha-4}+ar^{-\alpha-\beta-4})$.
The potential remainder is bounded by
$Mr^{-\alpha-2-\eta}(1+ar^{-\beta})$.
First choose $a\beta^2>3(C+M)$, then choose $R$ with
$ar^{-\beta}\le1/2$. The positive term dominates both errors.

Extend $\phi$ positively to the finite remaining set. The formula for
rows at $r\ge R$ is unchanged by choosing the extension only at
$r<R-1$. On $r<R$ choose, for example,
$U(x)=\max\{0,1-(H\phi)(x)/\phi(x)\}$, and put $U=0$ elsewhere.
Then
\[
 \rho(x)=\frac{((H+U)\phi)(x)}{\phi(x)}>0\quad(x\in X).
\]
The ground-state expansion gives $H+U\ge M_\rho\ge0$ as quadratic
forms. The function $\rho$ is bounded, positive at each vertex, and
need not have a positive infimum. For $\varepsilon>0$, inverse form
order gives
\[
 \ip{\delta_x}{(H+U+\varepsilon)^{-1}\delta_x}
 \le\frac1{\rho(x)+\varepsilon}\le\frac1{\rho(x)}.
\]
This proves nonnegativity and finiteness of the Green diagonal.
\end{proof}

\subsection{The finite part and the asymptotic prefactor}\label{sec:prefactor}

The multiplicative prefactor is fixed by the constant term of the critical
resolvent. We therefore work at $H_c=H-\tau_cW$ and prove a remainder
that tends to zero after subtracting the logarithmic singularity.
The canonical Poisson normalization and Bessel matching identify this
constant; finite-dimensional perturbation then gives the prefactor.

Throughout the construction below normalize the positive state by
\begin{equation}\label{pref:eq:normalization}
 H_ch=0,\qquad h(x)=r^{-\alpha}(1+O(r^{-\sigma})),\qquad
 \omega=|\mathbb S^{N-1}|,
\end{equation}
with $\sigma>0$ decreased whenever necessary. We write
$\gamma_{\mathrm E}$ for Euler's constant and
$k_0=\log2-\gamma_{\mathrm E}$ for the constant in the
small-argument expansion of the order-zero Bessel function.
Pairings with $h$ below are finite sums over the support of the source.

\begin{proposition}\label{pref:prop:finitepart}
For every real finite source $f$, there is a unique pointwise solution
$\mathcal Gf$ of
\begin{equation}\label{pref:eq:canonicalpoisson}
 H_c\mathcal Gf=f,\qquad
 (\mathcal Gf)(x)=-\frac{\ip h f}{\omega}r^{-\alpha}\log r
                  +O_f(r^{-\alpha-\sigma}).
\end{equation}
In particular, the constant term after the logarithm is set to zero.
For each fixed finite set $S$, in matrix operator norm,
\begin{equation}\label{pref:eq:criticalfinitepart}
 P_S(H_c+\varepsilon)^{-1}P_S
 =\frac{L_\varepsilon}{2\omega}h_Sh_S^*+\mathcal R_S+o_S(1),
 \qquad
 \mathcal R_S=\mathcal G_S+\frac{k_0}{\omega}h_Sh_S^*,
\end{equation}
where $\mathcal G_S(x,y)=(\mathcal G\delta_y)(x)$ and $h_S=P_Sh$.
The matrices $\mathcal G_S$ and $\mathcal R_S$ are real symmetric.
\end{proposition}

\subsubsection{Pinning and the logarithmic Poisson solution}

The \emph{critical charge} of a source $f$ is its pairing
$\ip h f$ with the critical state; it determines the coefficient of
the logarithmic correction to the power tail. Adding a positive
potential at one vertex, called \emph{pinning}, gives a finite Green
kernel used to treat sources of zero charge. A solution with leading
term $r^{-\alpha}\log r$ supplies the remaining charge.

\begin{proof}[Construction in Proposition~\ref{pref:prop:finitepart}]
We first record two consequences of pinning at a vertex. Fix $o\in X$
and $\kappa>0$. Let $P_o=P_{\{o\}}$, put $H_+=H_c+\kappa P_o$,
and denote its zero-energy Green kernel by $G_+$. The ground-state
identity gives $H_+\ge\kappa P_o$, hence
$G_+(o,o)\le\kappa^{-1}$. Testing the zero-energy variational
principle with $h\chi_R$, using the logarithmic null sequence
\eqref{rob:eq:nullsequence}, proves the reverse inequality. Thus
\begin{equation}\label{pref:eq:pinnedcolumn}
 G_+(o,o)=\kappa^{-1},\qquad
 G_+(x,o)=\frac{h(x)}{\kappa h(o)}.
\end{equation}
Minimality of the Green function gives
$\kappa h(o)G_+(x,o)\le h(x)$, since
$H_+h=\kappa h(o)\delta_o$. This inequality follows by comparison
on finite Dirichlet exhaustions and passage to their Green limit.
The nonnegative difference is $H_+$-harmonic and vanishes at $o$;
the finite-row strong maximum principle makes it zero everywhere.
This proves the column identity. Green finiteness at $o$ and
Lemma~\ref{rob:lem:green} show that $H_+$ satisfies all hypotheses
of Theorems~\ref{rob:thm:response} and~\ref{rob:thm:exactresponse}.

If a finite source $t$ has $\ip h t=0$, then $v=G_+t$ satisfies
$v(o)=0$ by \eqref{pref:eq:pinnedcolumn}, so $H_cv=t$.
The signed-source consequence of Theorem~\ref{rob:thm:exactresponse}
also gives
\begin{equation}\label{pref:eq:zerocharge}
 v(x)=r^{-\alpha}(b_t+O_t(r^{-\sigma})).
\end{equation}
It remains to produce one source with nonzero $h$-charge and a
logarithmic solution.

Use the reference function $\phi=r^{-\alpha}$ outside a fixed ball,
extended positively inside, and the network operator $\mathscr L$
from the proof of Lemma~\ref{rob:lem:amplitude}. Set
$U=H_c\phi/\phi=O(r^{-2-\gamma})$, where $\gamma=\min\{2,\eta\}$.
Outside a sufficiently large ball $B_R$,
\[
 \mathscr L\log r=O(r^{-4}).
\]
The stopped potential operator integrates a source along the
recurrent reference chain until it first reaches $B_R$. Denote this
operator by $\mathscr T_R$. The barrier
\eqref{rob:eq:amplitudebarrier} proves, for $\beta>0$ sufficiently
small, that
\[
 |F(x)|\le C r^{-2-\beta}
 \ \Longrightarrow\ 
 \|\mathscr T_RF\|_\infty\le C'R^{-\beta}.
\]
Choose $0<\beta<\gamma$. The equation
\begin{equation}\label{pref:eq:logfixedpoint}
 p=\mathscr T_R\bigl[-\mathscr L\log r-U\log r-Up\bigr]
\end{equation}
is a contraction on bounded functions when $R$ is large, since
$\|\mathscr T_R|U|\|_\infty\le CR^{-\beta}<1$.
Its forcing is bounded by $C r^{-2-\beta}$ after decreasing $\beta$.
The resulting bounded $p$ satisfies
$\mathscr Lp=O(r^{-2-\beta})$.
The oscillation argument leading to \eqref{rob:eq:oscillation}
applies to every bounded solution of such a Poisson equation: subtract
its stopped potential and apply the same annular Harnack contraction.
Consequently $p=b+O(r^{-\sigma})$ for some $\sigma>0$.
We have constructed an exterior solution
\begin{equation}\label{pref:eq:logmode}
 w(x)=\phi(x)(\log r+b+O(r^{-\sigma})),\qquad H_cw=0
 \quad\text{outside }B_R.
\end{equation}
Extend $w$ arbitrarily across the finite remaining set and put $s=H_cw$.
Then $s$ has finite support and
\begin{equation}\label{pref:eq:logcharge}
 \ip h s=-\omega.
\end{equation}

We compute this charge through an annular energy estimate. Write $b^0_{xy}=c(x,y)\phi(x)\phi(y)$ and
$\mu_0=\phi^2$. If $e$ solves a reference Poisson equation on an
annulus, testing against a cutoff times $e$ gives
\begin{equation}\label{pref:eq:caccioppoli}
 \sum_{\{x,y\}\subset A_R}b^0_{xy}|e(x)-e(y)|^2
 \le C\left[R^{-2}\sum_{A_R'}\mu_0|e|^2
              +\sum_{A_R'}\mu_0|e\mathscr Le|\right],
\end{equation}
where $A_R=\{R\le r\le2R\}$ and
$A_R'=\{R/2\le r\le4R\}$; fixed-width boundary enlargements are
understood. Indeed, expand the edge differences in
$\sum\mu_0(\mathscr Le)\xi^2e$ with $\xi=1$ on $A_R$,
$|\nabla\xi|\le C/R$, and absorb the mixed term by
$2ab\le a^2/2+2b^2$.
Apply this estimate to $h/\phi-1$ and to
$w/\phi-\log r-b$. Their forcing terms are respectively
$O(r^{-2-\gamma})$ and $O(r^{-2-\gamma}\log r)$.
After decreasing $\sigma<\gamma$, both annular energies are
$O(R^{-2\sigma})$. The product and quotient rules for edge differences
therefore give, with $d_{xy}=c(x,y)h(x)h(y)$,
\[
 \sum_{\{x,y\}\subset A_R}d_{xy}
 \left|\nabla_{xy}\left(w/h-\log r-b\right)\right|^2
 =O(R^{-2\sigma}\log^2R).
\]
Take a smooth radial cutoff $\chi_R=\chi(r/R)$ equal to one on
$B_R$ and zero outside $B_{2R}$. Its $d$-energy is $O(1)$.
The ground-state identity and summation by parts give
\[
 \ip h s
 =\sum_{\{x,y\}}d_{xy}
       \nabla_{xy}(w/h)\nabla_{xy}\chi_R.
\]
The preceding energy bound makes the error from $w/h-\log r-b$
tend to zero. In the remaining sum replace $d_{xy}$ by
$r^{2-N}(1+O(r^{-\min\{1,\sigma\}}))$ and compare with the radial
integral. Its limit is
$\omega\int_1^2\chi'(t)\,dt=-\omega$.
This proves \eqref{pref:eq:logcharge}; all edge sums count each
unoriented edge once.

Given $f$, set $Q=\ip h f$ and $t=f+(Q/\omega)s$. Then
$\ip h t=0$, and
\[
 \widetilde g=-\frac Q\omega w+G_+t
\]
solves $H_c\widetilde g=f$. Equations \eqref{pref:eq:zerocharge}
and~\eqref{pref:eq:logmode} give
$\widetilde g=\phi[-(Q/\omega)\log r+b_f+O(r^{-\sigma})]$.
Define $\mathcal Gf=\widetilde g-b_fh$ to obtain
\eqref{pref:eq:canonicalpoisson}. If two such solutions exist, their
difference divided by $h$ is a bounded harmonic function on the
recurrent $h$-network and tends to zero at infinity. Recurrence follows
from the same logarithmic null sequence. Bounded optional
stopping at any vertex makes it constant, hence zero. This proves
uniqueness, linearity, and independence of the pinning choices.
\end{proof}

\subsubsection{Matching with a vanishing resolvent error}

The zero-energy Poisson solution and the decaying Bessel profile have
matching logarithmic and constant terms. A weighted estimate for the
cutoff residual makes their difference tend to zero on each fixed
finite set, which is the precision needed for a prefactor.

\begin{proof}[Resolvent expansion in Proposition~\ref{pref:prop:finitepart}]
First we need an upper bound that is uniform in the second vertex.
Let $R_+(\varepsilon)=(H_++\varepsilon)^{-1}$. The rank-one
resolvent identity is
\[
 (H_c+\varepsilon)^{-1}(x,y)=R_+(\varepsilon)(x,y)
 +\frac{\kappa R_+(\varepsilon)(x,o)R_+(\varepsilon)(o,y)}
        {1-\kappa R_+(\varepsilon)(o,o)}.
\]
Theorem~\ref{rob:thm:exactresponse} gives
$1-\kappa R_+(\varepsilon)(o,o)\asymp L_\varepsilon^{-1}$.
Also $R_+(\varepsilon)\le G_+$ entrywise,
\eqref{pref:eq:pinnedcolumn} controls its column at $o$, and
Theorem~\ref{rob:thm:response} controls every fixed column of $G_+$
by a constant multiple of $h$. Thus, for fixed $x$ and all $y$,
\begin{equation}\label{pref:eq:weightedkernelbound}
 0<(H_c+\varepsilon)^{-1}(x,y)
       \le C_x L_\varepsilon h(y).
\end{equation}

Fix a finite source $f$, put $Q=\ip h f$, $c=Q/\omega$,
$g=\mathcal Gf$, and $s_\varepsilon=L_\varepsilon/2+k_0$.
Choose $R=\varepsilon^{-1/4}$ and a smooth radial cutoff
$\chi_R$ equal to one on $B_R$ and zero outside $B_{2R}$.
Use the following square-integrable approximate solution:
\begin{equation}\label{pref:eq:matchedsolution}
 u_\varepsilon=\chi_R(c s_\varepsilon h+g)
 +(1-\chi_R)c\phi\BesselK(\sqrt\varepsilon r).
\end{equation}
The classical series in~\cite[Formula~8.447.3]{GradshteynRyzhik}
and~\cite[Eq.~10.31.2]{DLMF} gives the small-argument expansion
\begin{equation}\label{pref:eq:besselconstant}
 \BesselK(t)=-\log(t/2)-\gamma_{\mathrm E}
                +O\bigl(t^2(1+|\log t|)\bigr).
\end{equation}
It follows on every fixed enlargement of $A_R$ that, if
$d=c s_\varepsilon h+g-c\phi\BesselK(\sqrt\varepsilon r)$,
\begin{equation}\label{pref:eq:matchingdifference}
 |d/\phi|\le C_f A_\varepsilon,\qquad
 A_\varepsilon=L_\varepsilon(R^{-\sigma}+\varepsilon R^2).
\end{equation}

We now estimate the cutoff residual to prove convergence on each
fixed finite set. On the transition annulus,
\[
 \left|\phi^{-1}(H_c+\varepsilon)d\right|
 \le C_f L_\varepsilon(\varepsilon+R^{-2-\gamma}).
\]
The inner term follows from $H_ch=0$, $H_cg=f$, with $f=0$ there.
For the outer term, the continuum identity
$(-\Delta-\alpha^2r^{-2}+\varepsilon)
(r^{-\alpha}\BesselK(\sqrt\varepsilon r))=0$
and fourth-order Taylor remainders give the bound.
Since
$\phi^{-1}(H_c+\varepsilon)(\phi e)=
\mathscr Le+(U+\varepsilon)e$, apply
\eqref{pref:eq:caccioppoli} to $e=d/\phi$.
The estimates $\sum_{A_R'}\mu_0=O(R^2)$,
$R^{-\gamma}\le R^{-\sigma}$, and~\eqref{pref:eq:matchingdifference}
give
\[
 \begin{split}
 \sum_{\{x,y\}\subset A_R}b^0_{xy}|e(x)-e(y)|^2
 &\le C_f\bigl\{A_\varepsilon^2
       +A_\varepsilon L_\varepsilon(\varepsilon R^2+R^{-\gamma})
       +(\varepsilon R^2+R^{-\gamma})A_\varepsilon^2\bigr\}\\
 &\le C_f A_\varepsilon^2.
 \end{split}
\]
Consequently the cutoff commutator satisfies
\begin{equation}\label{pref:eq:commutator}
 \sum_x h(x)|[H_c,\chi_R]d(x)|\le C_f A_\varepsilon.
\end{equation}
To see this directly, divide the commutator by $\phi$ and use
\[
 \phi^{-1}[H_c,\chi_R](\phi e)
 =e\mathscr L\chi_R
 -\mu_0^{-1}\sum_y b^0_{xy}
          (\chi_R(x)-\chi_R(y))(e(x)-e(y)).
\]
The first term has weighted absolute sum $O(A_\varepsilon)$ since
$\mathscr L\chi_R=O(R^{-2})$. Cauchy--Schwarz bounds the second
by the product of the two annular energy norms, also
$O(A_\varepsilon)$. Here $h/\phi$ is bounded above and below.

The exterior Taylor remainder is summable with the required weight.
More explicitly, with $t=\sqrt\varepsilon r$, it is bounded by
\[
 C\phi r^{-4}(1+t)^4
       [\BesselK(t/2)+t\mathrm K_1(t/2)]
       +C\phi r^{-2-\eta}\BesselK(t).
\]
This follows from the Bessel derivative
recurrences~\cite[Formulas~8.486.10--8.486.13]{GradshteynRyzhik}
and Taylor's formula on unit edges, as in
Lemma~\ref{rob:lem:consistency}.
Multiplication by $h$ and radial shell summation, split at
$r=\varepsilon^{-1/2}$, gives $O(L_\varepsilon R^{-\gamma})$:
below that radius use $\BesselK(t)=O(1+|\log t|)$ and
$t\mathrm K_1(t)=O(1)$, and above it use exponential decay.
The inner mass error has weighted absolute sum
$O_f(\varepsilon R^2L_\varepsilon)$.
Combining these estimates with \eqref{pref:eq:commutator}, for
$r_\varepsilon=(H_c+\varepsilon)u_\varepsilon-f$ we obtain
\begin{equation}\label{pref:eq:residualsmall}
 \sum_x h(x)|r_\varepsilon(x)|
 \le C_fL_\varepsilon(R^{-\sigma}+\varepsilon R^2).
\end{equation}
The residual is also square integrable, so the resolvent identity and
\eqref{pref:eq:weightedkernelbound} imply, at each fixed $x$,
\[
 |(H_c+\varepsilon)^{-1}f(x)-u_\varepsilon(x)|
 \le C_{x,f}L_\varepsilon^2
       (\varepsilon^{\sigma/4}+\varepsilon^{1/2})=o(1).
\]
For small energy this fixed $x$ lies in the inner region of
\eqref{pref:eq:matchedsolution}. Hence
\[
 (H_c+\varepsilon)^{-1}f(x)
 =\frac{\ip h f}{\omega}
      (L_\varepsilon/2+k_0)h(x)+\mathcal Gf(x)+o(1).
\]
Apply this to the finitely many point sources in $S$ to obtain
\eqref{pref:eq:criticalfinitepart} in matrix norm. Symmetry of the
resolvent and of its leading term proves symmetry of the finite part.
\end{proof}

\subsubsection{The prefactor from a finite critical matrix}

The constant term of the principal Birman--Schwinger eigenvalue
determines the multiplicative prefactor in the eigenvalue asymptotics. The next theorem
expresses that constant through the normalized Poisson solution and
then records the formula for any normalization of the critical state.

\begin{theorem}\label{pref:thm:prefactor}
Under the assumptions of Theorem~\ref{rob:thm:birth}, the magnitude
$\varepsilon(\tau)$ of the unique negative eigenvalue near $\tau_c$
has the asymptotic formula
\begin{equation}\label{pref:eq:eigenvalueasymptotics}
 \varepsilon(\tau)=\mathfrak P(H,W)
       \exp\!\left[-\frac{\mathfrak A(H,W)}{\tau-\tau_c}\right]
       (1+o(1)),\qquad \mathfrak P(H,W)>0.
\end{equation}
Normalize $h$ by \eqref{pref:eq:normalization}, put
$M=\sum_xW(x)h(x)^2$, and let $g=\mathcal G(Wh)$ be the unique
solution specified by \eqref{pref:eq:canonicalpoisson}. Then
\begin{equation}\label{pref:eq:prefactorformula}
 \mathfrak A(H,W)=\frac{2\omega}{M},\qquad
 \mathfrak P(H,W)=4e^{-2\gamma_{\mathrm E}}
       \exp\!\left[\frac{2\omega}{M^2}\sum_xW(x)h(x)g(x)\right].
\end{equation}
For an arbitrary normalization $h_c\sim a_c r^{-\alpha}$, put
$M_c=\sum W h_c^2$ and let $g_c$ solve
\[
 H_cg_c=Wh_c,\qquad
 g_c(x)=-\frac{M_c}{\omega a_c}r^{-\alpha}\log r
                 +O(r^{-\alpha-\sigma}).
\]
The same prefactor is
\begin{equation}\label{pref:eq:invariantprefactor}
 \mathfrak P(H,W)=4e^{-2\gamma_{\mathrm E}}
   \exp\!\left[\frac{2\omega a_c^2}{M_c^2}
                     \sum_xW(x)h_c(x)g_c(x)\right].
\end{equation}
The normalization at infinity fixes the additive multiple of $h_c$
in $g_c$.
\end{theorem}
\begin{proof}
Write $\delta=\tau-\tau_c$, so $H-\tau W=H_c-\delta W$.
On $F=\{W>0\}$, Proposition~\ref{pref:prop:finitepart} gives
\[
 W^{1/2}(H_c+\varepsilon)^{-1}W^{1/2}
 =\frac{L_\varepsilon}{2\omega}ww^*+C_F+o(1),
 \qquad w=W^{1/2}h,\quad
 C_F=W^{1/2}\mathcal R_FW^{1/2}.
\]
Here $\|w\|^2=M>0$. The largest eigenvalue therefore equals
\[
 \mu(\varepsilon)=\frac{M}{2\omega}L_\varepsilon+d+o(1),
 \qquad d=\frac{\ip w{C_Fw}}M
         =\frac{\ip{Wh}{\mathcal G(Wh)}}M+\frac{k_0M}{\omega}.
\]
For example, divide the matrix by $L_\varepsilon$ and apply the
simple-eigenvalue expansion at the rank-one matrix $ww^*/(2\omega)$;
its quadratic error becomes $O(L_\varepsilon^{-1})$ after rescaling.
The first Birman--Schwinger equation is $\mu(\varepsilon)=\delta^{-1}$.
Consequently
\[
 L_\varepsilon=\frac{2\omega}{M\delta}
                       -\frac{2\omega d}{M}+o(1).
\]
Exponentiation proves \eqref{pref:eq:eigenvalueasymptotics} and
\eqref{pref:eq:prefactorformula}, since $e^{2k_0}=4e^{-2\gamma_{\mathrm E}}$.
For an arbitrary normalization use $h=h_c/a_c$, $g=g_c/a_c$ and
$M=M_c/a_c^2$ to obtain \eqref{pref:eq:invariantprefactor}.
\end{proof}

The prefactor is determined by the normalized zero-energy Poisson
problem and depends on the finite core through that solution. The
derivative, spatial-mass and return estimates are two-sided bounds,
as stated in their respective results.

\subsection{Critical return probabilities}\label{sec:probability}

The low-energy norm estimate also determines the return law of the
chain transformed by the threshold state. The proof uses a scalar spectral
measure and positive finite-time bridges for the connected operators
specified in Section~\ref{sec:lowenergy}.

\begin{corollary}\label{rob:cor:returns}
Let $H_c=H-\tau_cW$ and let $h_c$ be the state in
Theorem~\ref{rob:thm:birth}. The Doob chain with generator
\begin{equation}\label{rob:eq:doob}
 (\mathscr G_cf)(x)=\sum_y c(x,y)\frac{h_c(y)}{h_c(x)}[f(y)-f(x)]
\end{equation}
is conservative and null recurrent. For each fixed $x\in X$, as
$t\to\infty$,
\begin{equation}\label{rob:eq:returns}
 \mathbb P_x^c(X_t=x)\asymp_x t^{-1},\qquad
 \mathbb P_x^c(T_x>t)\asymp_x(\log t)^{-1},
\end{equation}
where $T_x$ is the first return after the first departure, including the
initial holding time. Its reversible measure is $\mu_c(x)=h_c(x)^2$;
in particular $\mu_c(B_R)\asymp R^2$.
\end{corollary}
\begin{proof}
Keep the fixed positive vector $f_0=W^{1/2}\zeta_0$. The finite-rank
resolvent identity gives
\begin{equation}\label{rob:eq:criticalresolvent}
 (H_c+\varepsilon)^{-1}f_0
 =(H+\varepsilon)^{-1}W^{1/2}
        (I-\tau_c\B(\varepsilon))^{-1}\zeta_0.
\end{equation}
The finite inverse has nonnegative entries by its Neumann series, since
$\tau_c\lambda(\varepsilon)<1$. Its action on $\zeta_0$ has norm
comparable to $L_\varepsilon$. In fact, the component along
$\zeta_\varepsilon$ is
\[
 \frac{\ip{\zeta_\varepsilon}{\zeta_0}}
 {1-\tau_c\lambda(\varepsilon)}\zeta_\varepsilon,
\]
whose norm is comparable to $L_\varepsilon$ by
\eqref{rob:eq:lambdarate}; the other components remain bounded by the
strict spectral gap at zero. Thus the positive source on the right of
\eqref{rob:eq:criticalresolvent}, normalized in $\ell^1$, belongs to
$\cP_F$, and its normalizing factor is comparable to $L_\varepsilon$.
Theorem~\ref{rob:thm:response} gives
\begin{equation}\label{rob:eq:criticalderivative}
 \norm{(H_c+\varepsilon)^{-1}f_0}^2\asymp\varepsilon^{-1}.
\end{equation}

Let $\nu$ be the spectral measure of $f_0$ for $H_c$, and put
$A(s)=\nu([0,s])$. Equation \eqref{rob:eq:criticalderivative} says
\[
 \frac c\varepsilon\le
 \int_0^\infty(\lambda+\varepsilon)^{-2}\,d\nu(\lambda)
 \le\frac C\varepsilon.
\]
Restricting the integral at $\varepsilon=s$ to $[0,s]$ gives
$A(s)\le4Cs$. Enlarge the constant so that $A(s)\le C_1s$ for all
$s>0$. Integration by parts then yields
$\int_{(a,\infty)}\lambda^{-2}\,d\nu(\lambda)\le2C_1/a$.
Evaluate the lower derivative bound at $\varepsilon=a/M_1$ with
fixed $M_1$ sufficiently large. The contribution of $(a,\infty)$ is
at most half of $c/\varepsilon$. The contribution of $[0,a]$ is at
most $A(a)/\varepsilon^2$, so $A(a)\ge ca/(2M_1)$.
Therefore $A(s)\asymp s$ near zero. Integrating against $e^{-t\lambda}$
gives
\begin{equation}\label{rob:eq:Qreturn}
 Q(t):=\ip{f_0}{e^{-tH_c}f_0}\asymp t^{-1}.
\end{equation}
For the lower bound use $e^{-1}A(1/t)$; for the upper bound integrate
by parts and use $A(s)\le C_1s$.

Write $p_t^c(x,y)$ for the kernel of $e^{-tH_c}$. On $F$ the source
$f_0$ is strictly positive, hence
$Q(t)\ge f_0(y)^2p_t^c(y,y)$ for each $y\in F$.
For arbitrary fixed $x$, positivity of $p_1^c$ and finiteness of $F$
give a constant $b_x>0$ such that
\[
 p_{t+2}^c(x,x)\ge b_x
       \sum_{y,z\in F}p_t^c(y,z)\ge
       \frac{b_x}{\max_F f_0^2}Q(t).
\]
Conversely, fixing $y\in F$ gives
$p_{t+2}^c(y,y)\ge p_1^c(y,x)^2p_t^c(x,x)$.
Together with \eqref{rob:eq:Qreturn}, these inequalities prove
$p_t^c(x,x)\asymp_x t^{-1}$ for every fixed vertex.

The total jump rate in \eqref{rob:eq:doob} is, by $H_ch_c=0$,
\[
 q_c(x)=\sum_yc(x,y)\frac{h_c(y)}{h_c(x)}
       =\sum_yc(x,y)+q(x)-\tau_cW(x).
\]
It is strictly positive and uniformly bounded above. The chain is
therefore nonexplosive and conservative. Its transition probabilities
are $P_t^c(x,y)=h_c(y)p_t^c(x,y)/h_c(x)$, so the diagonal is unchanged.
Detailed balance gives $\mu_c=h_c^2$, and annular counting gives
$\mu_c(B_R)\asymp R^2$.

Fix $x$, write $q_x=q_c(x)$, and put
$U_x(s)=\int_0^\infty e^{-st}P_t^c(x,x)dt\asymp\log(1/s)$.
Decomposition at the first return after departure, with the initial
exponential holding time included, gives
\[
 U_x(s)=\frac1{q_x+s}+U_x(s)\mathbb E_x^c e^{-sT_x},\qquad
 1-\mathbb E_x^c e^{-sT_x}
 =\frac1{(q_x+s)U_x(s)}\asymp\frac1{\log(1/s)}.
\]
We use $e^{-s\infty}=0$. The limit at zero shows $T_x<\infty$ almost
surely. For its tail, use
\[
 (1-e^{-1})\mathbb P(T_x>t)
 \le1-\mathbb E e^{-T_x/t},\qquad
 1-\mathbb E e^{-T_x/t^2}\le t^{-1}+\mathbb P(T_x>t).
\]
These prove the two sides of the second estimate in \eqref{rob:eq:returns}.
Its integral diverges, so the mean return time is infinite. With
irreducibility, this proves null recurrence.
\end{proof}

The reversible measure $h_c^2$ has infinite total mass, whereas the
perturbation-weighted mass $\sum W h_c^2$ is finite because $W$ has
finite support.

\section{Critical couplings across the dimensional transition}
\label{sec:N8application}

Lemma~\ref{lem:N8subcritical} applies the general threshold theorem
to $H_N-\tau W$ for every $3\le N\le8$ and every nonzero nonnegative
finitely supported $W$, without a symmetry assumption on $W$.
For first-sphere perturbations, the Schur reduction also identifies
the threshold, the critical state and the interval with one negative
eigenvalue. We treat attraction in dimension eight and repulsion in
dimension nine.

Use $K,\mathcal Q,\Sigma,Y,e,b$ from Section~\ref{sec:N8}, and put
$V=w_8|_Y$. The constructed spatial remainder and compression of the free
lattice Laplacian give
\begin{equation}\label{eq:lebase}
 0\le K\le32I,\qquad K\ge M_V,\quad V(x)>0,\qquad
 m_0:=m(0)\le7-c_*<7.
\end{equation}

\subsection{The eight-dimensional exterior response}

The vector $v_\varepsilon$ is the exterior profile produced by the
constant first-sphere input. We record its scalar response
$m(\varepsilon)$, its squared mass $J(\varepsilon)$, and the response
lost as $\varepsilon$ increases from zero, $d(\varepsilon)$.
For $\varepsilon>0$, these quantities are
\begin{equation}\label{eq:ledefinitions}
 \begin{split}
 v_\varepsilon&=(K+\varepsilon)^{-1}b,\qquad
 m(\varepsilon)=\langle b,v_\varepsilon\rangle,\qquad
 J(\varepsilon)=\|v_\varepsilon\|_2^2,\\
 d(\varepsilon)&=m_0-m(\varepsilon),\qquad
 L_\varepsilon=\log(1/\varepsilon).
 \end{split}
\end{equation}

\begin{proposition}\label{thm:lowenergy}
As $\varepsilon\downarrow0$, the responses $v_\varepsilon$
converge pointwise and in $\mathscr H_K$ to the energy
representative $v_0$ of $b$.
The response $v_\varepsilon$ satisfies the spatial estimates of
\eqref{rob:eq:profile} with $N=8$, $\alpha=3$ and constants rescaled for
the fixed source $b$. In particular,
\begin{equation}\label{eq:leresponse}
 J(\varepsilon)\asymp\frac1{\varepsilon L_\varepsilon^2},\qquad
 d(\varepsilon)\asymp\frac1{L_\varepsilon}.
\end{equation}
In addition, with $a_b=\lim_{|x|\to\infty}|x|^3v_0(x)>0$,
\begin{equation}\label{eq:leexactresponse}
 d(\varepsilon)=\frac{2\pi^4a_b^2}{3L_\varepsilon}
                 +O(L_\varepsilon^{-2}).
\end{equation}
The limiting profile satisfies
\begin{equation}\label{eq:lezero}
 Kv_0=b,\qquad v_0(x)\asymp|x|^{-3},\qquad
 v_0\in\ell^2_{-\sigma}(Y)\ \Longleftrightarrow\ \sigma>1.
\end{equation}
\end{proposition}
\begin{proof}
Apply Theorem~\ref{rob:thm:response} to $K$ and the source
$b/\|b\|_1$, then restore the fixed normalization. This gives all
pointwise and norm estimates; the identity
$d(\varepsilon)=\int_0^\varepsilon J(s)\,ds$ gives the scalar comparison.
The zero-energy profile and annular counting give the weighted-space
classification.

It remains to identify the pointwise limit with the energy representative
used in the Schur calculation. The inequality $K\ge V>0$ gives
$\ker K=\{0\}$ and continuity of vertex evaluations in
$\mathscr H_K$. If $\mu_b$ is the spectral measure of $b$, then
$\int\lambda^{-1}d\mu_b(\lambda)=m_0<\infty$.
The zero-energy representative therefore satisfies
\[
 \|v_\varepsilon-v_0\|_{\mathscr H_K}^2
 =\int_{(0,32]}\frac{\varepsilon^2}
 {\lambda(\lambda+\varepsilon)^2}\,d\mu_b(\lambda)
 \longrightarrow0.
\]
Dominated convergence applies with majorant $\lambda^{-1}$.
Continuity of evaluations identifies this energy limit with the
pointwise Green limit. Passing to the limit in the finite row gives
$Kv_0=b$ at every vertex. Theorem~\ref{rob:thm:exactresponse}
applied to $K$ and $b$, together with
$|\mathbb S^7|=2\pi^4/\Gamma(4)=\pi^4/3$, gives
\eqref{eq:leexactresponse} and the existence of $a_b$.
\end{proof}
\subsection{The first-sphere coupling}

The same exterior response controls a family of perturbations supported on
the first sphere. Let $P_e$ be the orthogonal projection onto $e$, extended
by zero from $\Sigma$ to $X_8$, and let $P_\Sigma$ be multiplication by
$\mathbf1_\Sigma$. The rank-one perturbation changes only the
constant interface channel; the local perturbation is a diagonal
potential on the sixteen points of $\Sigma$. Denote the two families by
\begin{equation}\label{eq:birthfamilies}
 H_8^{\mathrm r}(\tau)=H_8-\tau P_e,\qquad
 H_8^{\mathrm{loc}}(\tau)=H_8-\tau P_\Sigma,
 \qquad\tau\ge0.
\end{equation}
Both preserve the essential spectrum $[0,32]$. Their first bound-state
branches coincide, as the next proposition shows.

Define
\begin{equation}\label{eq:birththreshold}
 \tau_*=7-m_0.
\end{equation}
Proposition~\ref{prop:defectschur} identifies this quantity in the
structural formula $\mathfrak d_8=16\tau_*+\|r_{\mathrm{cr}}\|_{K,*}^2$.
The next response eigenvalue determines how far the coupling can
increase before another interface direction reaches zero. Denote the
zero-energy response matrix by $T_0$ and its second largest
eigenvalue, counted with multiplicity, by $\lambda_2(T_0)$.
The corresponding coupling is $\tau_2$:
\[
 T_0=\lim_{\varepsilon\downarrow0}B^*(K+\varepsilon)^{-1}B,
 \qquad \tau_2=7-\lambda_2(T_0),
\]
Every entry of $T_0$ is strictly positive, since the
exterior is connected and every column of $B$ is a nonzero nonnegative
source. Perron--Frobenius and transitivity imply
$\lambda_2(T_0)<m_0$, hence $\tau_2>\tau_*$.

\begin{proposition}\label{prop:birthschur}
The threshold satisfies
\begin{equation}\label{eq:birththresholdbounds}
 c_*\le\tau_*<7,\qquad \tau_2>\tau_*.
\end{equation}
Both operators in \eqref{eq:birthfamilies} are nonnegative for
$0\le\tau\le\tau_*$. For every $\tau>\tau_*$,
$H_8^{\mathrm r}(\tau)$ has exactly one negative eigenvalue, which is simple.
For $\tau_*<\tau<\tau_2$, the local family has exactly the same negative
eigenvalue. Writing it as $-\varepsilon(\tau)$, one has
\begin{equation}\label{eq:birthequation}
 \tau-\tau_*=\varepsilon+d(\varepsilon).
\end{equation}
Its strictly positive normalized eigenvector, in the decomposition
$\ell^2(X_8)=\ell^2(\Sigma)\oplus\ell^2(Y)$, is
\begin{equation}\label{eq:birtheigenvector}
 \Psi_\tau=\frac{(e,v_{\varepsilon(\tau)})}
 {\sqrt{1+J(\varepsilon(\tau))}}.
\end{equation}
Moreover,
\begin{equation}\label{eq:birthderivative}
 \varepsilon'(\tau)=\frac1{1+J(\varepsilon(\tau))}
 =\|\mathbf1_\Sigma\Psi_\tau\|_2^2.
\end{equation}
\end{proposition}

\begin{proof}
The bounds on $\tau_*$ follow from
Proposition~\ref{prop:N8finalresponse}, and the strict channel gap was
proved above.

At energy $-\varepsilon<0$, the constant-channel Schur complement of
both families is
\begin{equation}\label{eq:birthscalar}
 f_\tau(\varepsilon)=7+\varepsilon-\tau-m(\varepsilon).
\end{equation}
On the orthogonal complement of $e$, the rank-one family retains
the Schur form $(7+\varepsilon)I-T_\varepsilon$, where
$T_\varepsilon=B^*(K+\varepsilon)^{-1}B\le T_0$.
It is bounded below by $\tau_2+\varepsilon$ there.
For the local family the lower bound is
$\tau_2+\varepsilon-\tau$, which is positive for $\tau<\tau_2$.

The scalar function in \eqref{eq:birthscalar} is strictly increasing:
\[
 f_\tau'(\varepsilon)=1+J(\varepsilon)>0.
\]
Its limit at zero is $\tau_*-\tau$, and it tends to infinity as
$\varepsilon\to\infty$, since $m(\varepsilon)\le29/\varepsilon$.
For $\tau\le\tau_*$ every regularized Schur complement is positive,
which proves nonnegativity. For $\tau>\tau_*$ there is one and only one
positive root of \eqref{eq:birthscalar}. The exterior block is invertible
at that energy and the interface kernel is one-dimensional. The Schur
identity thus proves existence, uniqueness and simplicity of the negative
eigenvalue, with the indicated range for the local family.

The signed permutation symmetry makes
$B^*v_\varepsilon=m(\varepsilon)e$. The block equations then give
\eqref{eq:birtheigenvector}; strict positivity follows from
Lemma~\ref{rob:lem:green}. Since
$P_\Sigma(e,v_\varepsilon)=P_e(e,v_\varepsilon)$, the two branches and
these eigenvectors agree. Subtracting $7-\tau_*-m_0=0$ from the root
equation gives \eqref{eq:birthequation}. Implicit differentiation proves
\eqref{eq:birthderivative}.
\end{proof}

\subsection{Threshold state, mass and critical returns}

The symmetry of the first-sphere problem identifies the branch
selected by Theorem~\ref{rob:thm:birth}. The same state determines
the mass bounds, return laws and canonical prefactor below.

\begin{proof}[Proof of Corollary~\ref{thm:birthintro}]
Lemma~\ref{lem:N8subcritical} permits application of
Theorem~\ref{rob:thm:birth} to $H_8$ and $W=P_\Sigma$.
Proposition~\ref{prop:birthschur} identifies its first threshold as
$\tau_*=7-m_0$, gives its quantitative bounds, and extends uniqueness
of the negative eigenvalue to the interval $\tau_*<\tau<\tau_2$.
The general theorem proves the absence of an $\ell^2$ zero eigenvector
and all the mass and radius estimates. The exact coefficient can also
be read directly from the Schur equation: by
\eqref{eq:birthequation} and~\eqref{eq:leexactresponse},
\[
 \log(1/\varepsilon(\tau))=
 \frac{\mathfrak A_8}{\tau-\tau_*}+O(1),\qquad
 \mathfrak A_8=\frac{2\pi^4a_b^2}{3}.
\]
The threshold vector $(e,v_0)$ equals $1/4$ on $\Sigma$ and has
squared mass one there. Thus normalization to one on $\Sigma$ gives
$a_8=4a_b$ and $\mathfrak A_8=\pi^4a_8^2/24$, as asserted.
Taking logarithms of the radius comparison yields the bounded remainder
in \eqref{eq:introdelocalization}.
The general theorem also gives
\[
 \varepsilon'(\tau)=\|\mathbf1_\Sigma\Psi_\tau\|_2^2
 \asymp\varepsilon\log^2(1/\varepsilon)
 \asymp\frac{\varepsilon}{(\tau-\tau_*)^2}.
\]
The explicit representative of the threshold state is
\begin{equation}\label{eq:birthstate}
 \Phi_*=(e,v_0),\qquad
 \Phi_*(x)\asymp |x|^{-3},\qquad
 \Phi_*\in\ell^2_{-\sigma}(X_8)\ \Longleftrightarrow\ \sigma>1.
\end{equation}
Indeed, $Kv_0=b$ and $B^*v_0=m_0e$ give the two zero-energy block
equations. Positivity and uniqueness in Theorem~\ref{rob:thm:birth}
identify it up to normalization.

To specify the prefactor, put $H_*=H_8-\tau_*P_\Sigma$ and let
$g_8$ be the canonical solution
\begin{equation}\label{pref:eq:N8poisson}
 H_*g_8=(e,0),\qquad
 g_8(x)=-\frac1{|\mathbb S^7|a_b}|x|^{-3}\log|x|
                   +O(|x|^{-3-\sigma}).
\end{equation}
Its existence and uniqueness follow from
Proposition~\ref{pref:prop:finitepart} by rescaling the state.
The signed permutation symmetries preserve this normalization, so
$g_8|_\Sigma=t_8e$, where $t_8=\ip e{g_8|_\Sigma}$.
Since $\sum_\Sigma\Phi_*^2=1$, formula
\eqref{pref:eq:invariantprefactor} yields
\begin{equation}\label{pref:eq:N8prefactor}
 \mathfrak P_8=4e^{-2\gamma_{\mathrm E}}
                     \exp(\mathfrak A_8t_8)>0.
\end{equation}
This proves the full asymptotic in \eqref{eq:introbirthlaw}.
It also gives a convergent second term in the exterior response.
Indeed, the constant-channel Schur identity at criticality and
\eqref{pref:eq:criticalfinitepart} give
\[
 \frac1{\varepsilon+d(\varepsilon)}
 =\ip{(e,0)}{(H_*+\varepsilon)^{-1}(e,0)}
 =\frac{L_\varepsilon}{\mathfrak A_8}
       +t_8+\frac{2k_0}{\mathfrak A_8}+o(1).
\]
Consequently,
\begin{equation}\label{pref:eq:N8secondresponse}
 d(\varepsilon)=\frac{\mathfrak A_8}{L_\varepsilon}
 -\frac{\mathfrak A_8\log\mathfrak P_8}{L_\varepsilon^2}
 +o(L_\varepsilon^{-2}).
\end{equation}
Here $\varepsilon=o(L_\varepsilon^{-2})$, and
$\log\mathfrak P_8=\mathfrak A_8t_8+2k_0$.
The amplitude $a_b$ and the normalized Poisson datum $t_8$
determine the prefactor.
\end{proof}

The rank-one family in \eqref{eq:birthfamilies} has the same positive
bound-state branch, hence the same prefactor, and the same pointwise
threshold state. At the
threshold it also has no $\ell^2$ zero eigenvector: since
\[
 H_8^{\mathrm r}(\tau_*)
 =H_8^{\mathrm{loc}}(\tau_*)+\tau_*(P_\Sigma-P_e)
 \ge H_8^{\mathrm{loc}}(\tau_*)\ge0,
\]
a zero eigenvector for the left side would belong to the kernel of the
local operator, which is trivial.

The critical return comparisons follow from
Corollary~\ref{rob:cor:returns} with $h_c=\Phi_*$ and reversible
measure $\Phi_*^2$. The additional information specific to this example
is the interface characterization of the unique negative eigenvalue and the scalar second-order
response \eqref{pref:eq:N8secondresponse}.

\subsection{Extinction at the ninth-dimensional repulsion threshold}
\label{subsec:extinction9}

The same threshold law describes the removal of the negative spectrum
in dimension nine. The stable exterior makes it possible to cross the
coupling threshold while keeping the exterior operator fixed.

\begin{corollary}\label{cor:extinction9}
Let $t_9=m_9(0)-23/4>1/54$ and put
$J_9(t)=H_9+tP_{\Sigma_9}$ for $t\ge0$.
Then $J_9(t)$ is nonnegative exactly for $t\ge t_9$.
At $t=t_9$ it is critical, has no $\ell^2$ zero eigenvector, and its
positive zero-energy state, normalized to have first-sphere mass one, is
\[
 \Phi_9=(e_9,v_9),\qquad
 v_9=G_{K_9}b_9
     =\lim_{\varepsilon\downarrow0}(K_9+\varepsilon)^{-1}b_9,
\]
where the limit is pointwise and in the energy space $\mathscr H_{K_9}$.
In particular, $K_9v_9=b_9$ and
\[
 \Phi_9(x)=a_9|x|^{-7/2}(1+O(|x|^{-\sigma})),
 \qquad a_9,\sigma>0.
\]
For $t<t_9$ sufficiently close to $t_9$, there is exactly one negative
eigenvalue, which is simple, and
\begin{equation}\label{eq:extinction9}
 -\varepsilon_9(t),\qquad
 \varepsilon_9(t)=\mathfrak P_9
 \exp\!\left[-\frac{\mathfrak A_9}{t_9-t}\right](1+o(1)),
 \qquad \mathfrak A_9=2|\mathbb S^8|a_9^2.
\end{equation}
The prefactor is specified by
\eqref{pref:eq:invariantprefactor} with
$H_c=J_9(t_9)$, $W=P_{\Sigma_9}$, $h_c=\Phi_9$, and $M_c=1$.
\end{corollary}
\begin{proof}
The positivity criterion is \eqref{eq:repulsion9}.
Use the subcritical reference operator
$\widehat H_9=J_9(t_9+1)=J_9(t_9)+P_{\Sigma_9}$.
It has the exact critical tail and satisfies
$\widehat H_9\ge P_{\Sigma_9}$. The variational resolvent formula
therefore gives $\widehat G_9(o,o)\le1$ at a first-sphere vertex.
Lemma~\ref{rob:lem:green} propagates Green finiteness to every pair.
The first-sphere Schur criterion shows that
$\widehat H_9-\tau P_{\Sigma_9}\ge0$ exactly for $0\le\tau\le1$.
Apply Theorems~\ref{rob:thm:birth} and~\ref{pref:thm:prefactor}
with $\tau_c=1$, and set $\tau=t_9+1-t$.
This gives $\tau-\tau_c=t_9-t$ and proves the spectral assertions.
The critical state is invariant under signed permutations by uniqueness.
Its trace is therefore a multiple of $e_9$; normalization to
$\|\Phi_9|_{\Sigma_9}\|_2=1$ fixes that multiple to one.
The exterior block equation and its energy representative give $v_9$.
Equivalently, $B_9^*v_9=m_9(0)e_9$ verifies the first-sphere equation
directly. Since $M_c=1$, the general rate formula gives
$\mathfrak A_9=2|\mathbb S^8|a_9^2$.
\end{proof}

\section{Consequences and further questions}\label{sec:conclusion}

The lattice Hardy family shows that preservation of the continuum
threshold at infinity does not guarantee preservation of the global
stability threshold. The total defect makes the dimensional change
explicit, and exterior optimization measures the correction needed
to pass from a prescribed profile to the least energy. At the first
strict dimension, the exterior itself is still subcritical. The
instability occurs in its coupling to the first sphere and leads to
loss of a uniform spectral lower bound under mesh refinement.

Local perturbations turn this stability comparison into a threshold
problem with a continuously varying parameter. The exponential
eigenvalue law holds for the stated class of critical inverse-square
tails, with constants determined by zero-energy data. Finite changes
of the core and power-decaying tail corrections may change those
constants and the critical coupling. The two-dimensional return
scale persists throughout that class.

Several questions remain. The scalar theta interpolation has at least
one zero in $(8,9)$; its uniqueness and monotonicity in the real
parameter are open beyond the integer-dimensional comparison proved
here. Determining the nonsharp Hardy constants $C(N)$ requires
further optimization; at sufficiently large dimensions the
trace-normalized infimum is already $-\infty$. For the general
threshold problem, one can ask whether the spatial mass and return
comparisons admit exact asymptotic constants. Tail corrections of
order $|x|^{-2}(\log|x|)^{-2}$ also lie outside the present power-decay
hypothesis and provide a natural extension of the spectral analysis.

\appendix
\addtocontents{toc}{\protect\setcounter{tocdepth}{1}}
\section{Theta representation and dimensional estimates}\label{app:defecttheta}

We derive Proposition~\ref{prop:defecttheta}. The point of the
regularization is that the energy and inverse-square mass of the critical
profile diverge logarithmically when separated, whereas their discrete
combination in \eqref{eq:defectdefinition} is absolutely summable.

For $p>0$, let $V_p\sim\operatorname{Beta}(p,p)$ and define
\begin{equation}\label{eq:appPp}
 P_p(t)=\mathbb E e^{-tV_p},\qquad
 B_p(t)=1+e^{-t}-2P_p(t).
\end{equation}
For $a,b>0$, apply the classical beta parameter
identity~\cite[Formula~3.197.4]{GradshteynRyzhik} with both shape
parameters equal to $p$ and affine parameter $a/b-1>-1$.
Combining it with the Mellin--Laplace
formula~\cite[Formula~3.381.4]{GradshteynRyzhik} gives
\begin{align}
 a^{-p}b^{-p}
 &=\frac{1}{\Gamma(2p)}\int_0^\infty t^{2p-1}
 \mathbb E e^{-t(V_pa+(1-V_p)b)}dt,\label{eq:appFeynman}\\
 (a^{-p}-b^{-p})^2
 &=\frac{1}{\Gamma(2p)}\int_0^\infty t^{2p-1}
 e^{-t\min\{a,b\}}B_p(t|a-b|)dt.\label{eq:appedgeMellin}
\end{align}
In \eqref{eq:appedgeMellin} the second formula is used only when
$a$ and $b$ are the squared norms of neighbouring nonzero lattice
points, so $|a-b|$ is a positive odd integer after choosing the smaller
squared norm.

Fix a coordinate direction. Pair the two edges crossing the levels
$k$ and $-k-1$. The two edges incident with the origin contribute the
separate term $e^{-t}$. Summing the transverse coordinates gives
$\Theta(t)^{N-1}$, and excluding the transverse origin on the first
pair gives $\Theta(t)^{N-1}-1$. Thus at the critical value
$p=\nu_N$ the edge kernel is exactly \eqref{eq:defectJ}. This is the
origin of the three terms in that formula.

To keep the cancellation at zero intact, introduce
\begin{equation}\label{eq:appFN}
 \mathcal F_N(t)=\int_t^\infty(\Theta(u)^N-1)du.
\end{equation}
Put $p=\nu_N+\delta$, $\delta>0$, and let
$h_p(x)=|x|^{-2p}$ for $x\ne0$, with $h_p(0)=0$.
The energy and inverse-square mass of $h_p$ now converge separately.
Summation by parts on expanding balls has a boundary error
$O(R^{-4\delta})$, and gives their difference as
$\sum_{x\ne0}h_p(x)(H_Nh_p)(x)$. Uniformly for
$p\in[\nu_N,\nu_N+\delta_0]$, Taylor's formula outside a fixed ball gives
\begin{equation}\label{eq:appregularizeddefect}
 h_p(x)(H_Nh_p)(x)
 =-4\delta^2|x|^{-N-4\delta}
   +O_N(|x|^{-N-2-4\delta}).
\end{equation}
The sum of the first term is $O_N(\delta)$, since annular counting
bounds $\sum_{x\ne0}|x|^{-N-4\delta}$ by $O_N(1+\delta^{-1})$.
The remainder is dominated by a fixed multiple of $|x|^{-N-2}$.
Consequently the regularized difference tends to $\mathfrak d_N$.

Let $\mathcal J_p$ denote the kernel in \eqref{eq:defectJ} with
$\mathcal B_N$ replaced by $B_p$. The preceding absolutely convergent
identities give
\[
 \sum_{x\ne0}h_p(x)(H_Nh_p)(x)
 =\frac1{\Gamma(2p)}\int_0^\infty t^{2p-1}
       [\mathcal J_p(t)-s_N^2\mathcal F_N(t)]\,dt.
\]
Here is a uniform small-$t$ expansion of its integrand. Write
$\Theta_v(t)=\sum_{k\in\Z}e^{-t(k+v)^2}$. Completing the square in
\eqref{eq:appFeynman} and correcting the two origin edges in each
coordinate direction gives the exact identity
\begin{align*}
 \mathcal J_p(t)=N\big[2\Theta(t)^N
 &-2\Theta(t)^{N-1}\mathbb E\{
       e^{-tV_p(1-V_p)}\Theta_{V_p}(t)\}\\
 &-2+4P_p(t)\big].
\end{align*}
Jacobi's transformation~\cite[Eq.~20.7.32]{DLMF} and the theta
Fourier series~\cite[Formula~8.180.4]{GradshteynRyzhik} give
\begin{equation}\label{eq:appshiftedJacobi}
 \Theta_v(t)=\sqrt{\frac\pi t}
 \left[1+2\sum_{n\ge1}e^{-\pi^2n^2/t}\cos(2\pi nv)\right],
 \qquad t>0.
\end{equation}
The absolute value of the nonconstant term in brackets is at most
$2e^{-\pi^2/t}/(1-e^{-\pi^2/t})$. Hence
$\Theta_v(t)=(\pi/t)^{1/2}[1+O(e^{-\pi^2/t})]$
uniformly in $v\in[0,1]$. Since $0\le V_p(1-V_p)\le1/4$ and
$\mathbb E[V_p(1-V_p)]=p/[2(2p+1)]$, it follows that
\begin{equation}\label{eq:appJpuniform}
 \mathcal J_p(t)=\frac{Np}{2p+1}\pi^{N/2}t^{1-N/2}
       +O_N(t^{2-N/2}+1),\qquad 0<t\le1,
\end{equation}
with a constant uniform on the indicated compact interval of $p$.
The same Jacobi formula gives
$\mathcal F_N(t)=\pi^{N/2}t^{-s_N}/s_N+O_N(1)$.
Thus the singular coefficient of the difference is
\[
 c(p)=\frac{Np}{2p+1}-s_N,
 \qquad c(\nu_N)=0,\qquad c'(\nu_N)=\frac4N.
\]
Its contribution on $(0,1)$ is
\[
 \frac{\pi^{N/2}c(p)}{\Gamma(2p)}
 \int_0^1t^{2\delta-1}\,dt
 \longrightarrow \frac{2\pi^{N/2}}{N\Gamma(s_N)}.
\]
After this term is subtracted, the integrand is bounded by
$C_N(1+t^{s_N-1})$, uniformly in $\delta\in[0,\delta_0]$.
At infinity $0\le B_p\le1$ bounds $\mathcal J_p(t)$ and
$\mathcal F_N(t)$ by $C_Ne^{-t}$, uniformly in $p$.
Dominated convergence, together with \eqref{eq:appregularizeddefect},
therefore proves
\begin{equation}\label{eq:appolddefect}
 \mathfrak d_N=
 \frac{2\pi^{N/2}}{N\Gamma(s_N)}
 +\frac1{\Gamma(s_N)}\int_0^\infty t^{s_N-1}
 \big[\mathcal J_N(t)-s_N^2\mathcal F_N(t)\big]dt.
\end{equation}

It remains to eliminate the auxiliary tail \(\mathcal F_N\).
Setting $v=0$ in \eqref{eq:appshiftedJacobi} gives the Jacobi
transformation
\begin{equation}\label{eq:appJacobi}
 \Theta(t)=\sqrt{\frac\pi t}\,\Theta(\pi^2/t).
\end{equation}
Consequently, as $t\downarrow0$,
\begin{equation}\label{eq:appFasymptotic}
 \mathcal F_N(t)\sim\frac{\pi^{N/2}}{s_N}t^{-s_N}.
\end{equation}
For $\varepsilon>0$, integration by parts therefore gives
\begin{align}
 \int_\varepsilon^\infty t^{s_N-1}\mathcal F_N(t)dt
 ={}&-\frac{\varepsilon^{s_N}\mathcal F_N(\varepsilon)}{s_N}
 +\frac1{s_N}\int_\varepsilon^\infty
 t^{s_N}(\Theta(t)^N-1)dt.\label{eq:appIBP}
\end{align}
The boundary term in \eqref{eq:appIBP}, multiplied by $-s_N^2$,
contributes another $\pi^{N/2}$ in the limit. Substitution into
\eqref{eq:appolddefect} gives exactly \eqref{eq:defecttheta}.

The same modular calculation also gives
\begin{equation}\label{eq:appJsmall}
 \mathcal J_N(t)=s_N\pi^{N/2}t^{1-N/2}
 +O_N(t^{2-N/2}+1),\qquad t\downarrow0.
\end{equation}
The term $s_Nt(\Theta(t)^N-1)$ has the same leading part, so the
integrand in \eqref{eq:defecttheta}, including the factor
$t^{s_N-1}$, is $O_N(1+t^{s_N-1})$ at zero. At infinity,
$0\le\mathcal B_N\le1$ gives
$\mathcal J_N(t)=O_N(e^{-t})$ and the mass term is $O_N(te^{-t})$.
This proves the endpoint convergence claimed in
Proposition~\ref{prop:defecttheta}.

\subsection{The single base estimate in dimension nine}\label{app:defectsmallchecks}

We prove Lemma~\ref{mono:lem:N9base} using the integrated kernel
$I=I_9$ from \eqref{mono:eq:I}. First, with $x=2^{1/4}$,
\[
 Q_9=\frac{88\pi^4}{135}+\frac{212403}{2048}-144x+\frac{175}{8}x^2
 <\frac{282239}{10240}<\frac{138}{5},
\]
by $\pi<22/7$, $x>19/16$ and $x^2<10/7$.

The remainder admits a uniform analytic lower bound. For $\lambda,b>0$,
put $s=7/2$, $a=\lambda+b/2$ and $z=b/(2a)$. By \eqref{mono:eq:kappa},
$\kappa_{m,9}\ge(77/117)(m-1)$. For
$d_m=(m-1)(s)_{2m}/(2m)!$, $m\ge2$,
\[
 \frac{d_{m+1}}{d_m}
 =\frac{m(2m+7/2)(2m+9/2)}{(m-1)(2m+1)(2m+2)}
 \ge\frac{m+2}{m-1},
\]
since $2m^2+7m/4-4>0$. Integrating \eqref{eq:defectXiseries}
and summing the binomial minorant therefore gives
\begin{align}
 I(\lambda,b)
 &=2a^{-s}\sum_{m\ge2}\frac{\kappa_{m,9}(s)_{2m}}{(2m)!}z^{2m}
 \notag\\
 &\ge2\frac{77}{117}d_2a^{-s}\frac{z^4}{(1-z^2)^4}
 =\frac{5929}{3072}\frac{b^4\sqrt a}{[\lambda(\lambda+b)]^4}
 >\frac{193}{100}\frac{b^4\sqrt a}{[\lambda(\lambda+b)]^4}.
 \label{eq:N9integratedlower}
\end{align}
For $9/4\le a\le9$, $(a-9/4)(a-9)\le0$ implies
$\sqrt a\ge(2a+9)/9$. Let $b_j=2^j\binom8j$.
Retaining the terms through $e^{-4t}$ in
$\Theta^8\ge(1+2e^{-t})^8$, and using \eqref{mono:eq:R}, yields
\begin{align*}
 R_9&\ge18\left[\sum_{j=2}^4 b_jI(j,1)
                  +\sum_{j=1}^4 b_jI(j+1,3)\right]\\
 &>\frac{193}{50}\left[
 \sum_{j=2}^4\frac{b_j(2j+10)}{[j(j+1)]^4}
 +81\sum_{j=1}^4\frac{b_j(2j+14)}{[(j+1)(j+4)]^4}\right]\\
 &>8+\frac{14}{3}+6+\frac43+\frac92+\frac9{20}+3
 =\frac{559}{20}.
\end{align*}
In the last line, the seven terms are ordered by increasing $j$, with
$I(j,1)$ before $I(j+1,3)$ when both occur; their respective numerators
after subtracting the displayed lower bounds are
$64$, $7$, $1$, $2$, $99$, $909$, $3741$, over the positive denominators
$15625$, $2025$, $225$, $2025$, $1715$, $25000$, $400000$.
Consequently
\[
 \mathfrak d_9=Q_9-R_9<\frac{138}{5}-\frac{559}{20}
 =-\frac7{20}<-\frac13.
\]
\subsection{A uniform analytic bound for the dimensional remainder}\label{app:defectremainder}

We prove Lemma~\ref{mono:lem:550}. Recall the remainder
\eqref{mono:eq:R}, write $s=(N-2)/2$ and $n=N-1$, and use
$\Xi_\infty$ from \eqref{mono:eq:Xiinf}. The estimate is uniform for
all integer $N\ge9$.

\begin{proof}
Use \eqref{mono:eq:betaorder} in \eqref{mono:eq:R}. Its prefactor becomes
$2(N-2)/\Gamma(s)=4s/\Gamma(s)$, and every $\Xi_N$ is replaced by
$\Xi_\infty$. Split the integral at $t=3$.

\smallskip\noindent
\emph{The interval $0<t\le3$.}
The remainder integrand is bounded by the full positive sum
\[
 \Theta(t)^n\sum_{k\ge0}e^{-k^2t}\Xi_\infty((2k+1)t).
\]
Let $\Theta_{1/2}(t)=\sum_{k\in\Z}e^{-t(k+1/2)^2}$.
Telescoping the formula \eqref{mono:eq:Xiinf} gives
\[
 \sum_{k\ge0}e^{-k^2t}\Xi_\infty((2k+1)t)
   =(t/4-1)\Theta(t)+e^{-t/4}\Theta_{1/2}(t)
   \le\Theta(t)(t/4-1+e^{-t/4}).
\]
The inequality $\Theta_{1/2}\le\Theta$ follows directly from the
Poisson formula. Since $t/4-1+e^{-t/4}\le t^2/32$, Jacobi's identity
and $\Theta(3)<H:=11/10$ bound this part of $R_N$ by
\begin{equation}\label{mono:eq:Us}
 U(s):=\frac{3s}{8\Gamma(s)}\left(\frac{363}{100}\right)^{s+1}.
\end{equation}
Indeed, the monotonicity of $\Theta$ and Jacobi's identity give
\[
 \sqrt t\,\Theta(t)=\sqrt\pi\,\Theta(\pi^2/t)
 \le\sqrt\pi\,\Theta(\pi^2/3)
 =\sqrt3\,\Theta(3)<\sqrt3\,H,
 \qquad 0<t\le3.
\]
Thus $\Theta(t)^N\le (3/t)^{N/2}H^N$, and the remaining power of $t$
in the integral is zero.
Put $q_*=363/100$. On each of the two admissible progressions of $s$,
\[
 \frac{U(s+1)}{U(s)}=q_*\frac{s+1}{s^2}.
\]
This ratio decreases; hence the maxima for integer dimensions $N\ge9$
are at $s=9/2$ and $s=5$. Both are less than $180$. Indeed
$q_*^6<2300$ and $\sqrt{\pi q_*}>33/10$ give
\[
 U(9/2)=\frac{9q_*^{11/2}}{35\sqrt\pi}
       <\frac{6}{77}\,2300<180,
 \qquad U(5)=\frac{5}{64}q_*^6<180.
\]

\smallskip\noindent
\emph{The interval $t\ge3$.}
Write $q=e^{-t}$ and fix
\[
 H=\frac{11}{10},\qquad a=\frac{201}{100},\qquad b=\frac{501}{100}.
\]
Geometric series give the elementary bounds
\begin{gather}
 \Theta(t)\le H,\qquad
 \Theta(t)-1\le\frac{2001}{1000}q,\qquad
 \Theta(t)-1-2q\le a q^4,\notag\\
 \sum_{k\ge2}(2k+1)e^{-k^2t}\le bq^4.\label{mono:eq:thetabounds}
\end{gather}
For example, use $k^2\ge4+5(k-2)$ for $k\ge2$; the bounds reduce to
$2q^4/(1-q^5)$ and
$q^4[5/(1-q^5)+2q^5/(1-q^5)^2]$ at $q\le e^{-3}$.
Taylor's formula for $x^n$ consequently gives
\[
 \Theta^n-1-2nq\le anq^4+an(n-1)H^{n-2}q^2,
 \qquad
 \Theta^n-1\le anH^{n-1}q.
\]
Here $(2001/1000)^2/2<a$.
Extend the resulting positive integrals to $(0,\infty)$. Define
\begin{align*}
 J_s(\lambda,c)
 &:=\frac{1}{\Gamma(s)}\int_0^\infty t^{s-1}e^{-\lambda t}\Xi_\infty(ct)dt\\
 &=\frac{cs}{4}[\lambda^{-s-1}-(\lambda+c)^{-s-1}]
    -\lambda^{-s}-(\lambda+c)^{-s}
    +2(\lambda+c/2)^{-s}.
\end{align*}
Since $\Xi_\infty(t)\le t/4$, the large-$t$ contribution is at most
$L_1(s)+L_2(s)+T(s)$, where, with $r=121/200$,
\begin{align}
 L_1(s)&=4sa(2s+1)(2s)H^{-1}r^s
          [s/8-1+2(4/5)^s],\notag\\
 L_2(s)&=4sa(2s+1)r^s(3s/8),\notag\\
 T(s)&=\frac{s^2}{4^{s+1}}[a(2s+1)+bH^{2s+1}].\label{mono:eq:large}
\end{align}
In fact these follow by using
\[
 J_s(2,1)\le2^{-s}[s/8-1+2(4/5)^s],\quad
 J_s(2,3)\le2^{-s}(3s/8),\quad
 J_s(4,1)\le\frac{s}{4\,4^{s+1}}.
\]
For the middle inequality $2(4/7)^s<1$ when $s\ge7/2$.

Here are uniform bounds for all the terms in \eqref{mono:eq:large}.
Use
\[
 \log(200/121)>\frac12,\qquad
 \log(250/121)>\frac{18}{25},\qquad e>\frac{19}{7},
\]
and $\sup_{x>0}x^m e^{-cx}=(m/(ce))^m$.
For $7/2\le s\le8$, drop $s/8-1$ and use $2s+1\le16s/7$ to obtain
\[
 L_1(s)\le
 \left(4a\,\frac{16}{7}\,2H^{-1}\,2\right)
 \left(\frac{25}{6}\right)^3\left(\frac{7}{19}\right)^3<250.
\]
For $s\ge8$, use $s/8-1+2(4/5)^s\le s/8$ and
$2s+1\le17s/8$. This gives
\[
 L_1(s)\le
 \left(4a\,\frac{17}{8}\,2H^{-1}/8\right)
 8^4\left(\frac{7}{19}\right)^4<294.
\]
On the whole range $s\ge7/2$,
\[
 L_2(s)\le
 \left(4a\,\frac{16}{7}\,\frac{3}{8}\right)
 6^3\left(\frac{7}{19}\right)^3<75.
\]
Both summands of $T(s)$ decrease for $s\ge7/2$.
At that endpoint, $H^8<11/5$ gives
\[
 T(s)\le \frac{49}{2048}[8a+bH^8]
       <\frac{49}{2048}\left[8a+\frac{11b}{5}\right]<1.
\]
The derivative checks here use $\log4>4/3$ and
$\log(400/121)>1$.
Thus the large-$t$ part is less than $294+75+1=370$.
Together with \eqref{mono:eq:Us}, this proves $R_N<180+370=550$.
All exponential and logarithmic constants used above can, for example,
be checked from $e>\sum_{j=0}^5 1/j!=163/60$ and the positive series
$\log x=2\sum_{j\ge0}((x-1)/(x+1))^{2j+1}/(2j+1)$.
\end{proof}

\section{Scalar estimates for quadratic fields}\label{app:quadraticfields}\label{app:exterior9}

The same algebraic calibration proves the exterior estimates through
dimension seven and supports the completion in dimension nine. We first
derive the common scalar formulas and then give the estimates for the
five lower dimensions.

\subsection{The affine calibration}\label{app:calibration9}

The next algebraic formulas apply in every dimension $N\ge3$. Set
$\mu=(N-2)/4$ and choose $\kappa$ with $4\kappa>\mu^2$. For $s,t>0$, put
\begin{equation}\label{eq:uniformbase}
 a_{\mu,\kappa}(s,t)=
 \frac{1-\mu v+\kappa v^2}{1+\mu v+\kappa v^2},\qquad
 v=\frac{t-s}{t+s}.
\end{equation}
The numerator and denominator are positive, since
$1\pm\mu v+\kappa v^2\ge1-\mu^2/(4\kappa)>0$.
Interchanging $s,t$ changes the sign of $v$, so the field is reciprocal.
The formulas below give the scalar remainder whose sign will be
estimated for the dimension-specific parameter choices.

Set
\begin{equation}\label{eq:uniformAB}
 c=1+\mu+\kappa,\qquad
 \mathsf A=4s^2+2(\mu+2)s+c,\qquad
 \mathsf B=4(\mu+2)s+4c.
\end{equation}
Writing $t=s+\delta$, direct subtraction gives
\begin{equation}\label{eq:uniformincrement}
 1-a_{\mu,\kappa}(s,s+\delta)
 =\frac{2\mu\delta(2s+\delta)}
        {4s^2+2(\mu+2)s\delta+c\delta^2}.
\end{equation}
The denominator is positive for all real $\delta$ and $s>0$, because
\[
 4s^2+2(\mu+2)s\delta+c\delta^2
 =c\left(\delta+\frac{\mu+2}{c}s\right)^2
       +\frac{4\kappa-\mu^2}{c}s^2>0.
\]
For $q\ge0$, define
\begin{align}
 \mathsf D_s(q)&=(\mathsf A+4cq)^2-\mathsf B^2q,
 \label{eq:uniformD}\\
 g_s(q)&=\frac{4\mu\big[(2s+1+4q)(\mathsf A+4cq)
                         -4(s+1)\mathsf Bq\big]}{\mathsf D_s(q)}.
 \label{eq:uniformg}
\end{align}
Here $\mathsf D_s(q)>0$: it is the product of the two denominators in
\eqref{eq:uniformincrement} at $\delta=1\pm2\sqrt q$.
Adding those two fractions gives \eqref{eq:uniformg}. At a lattice point
with $s=|x|^2\ge2$, no neighbour is the origin, so the exact increments
$1\pm2x_j$ yield
\begin{equation}\label{eq:uniformsumg}
 \W_{a_{\mu,\kappa}}^{(N)}(x)=\sum_{j=1}^N g_s(x_j^2).
\end{equation}

An affine function with the correct coordinate sum provides the
minorant. Put
\begin{align}
 T&=4(\mu+1)s^2+2(1-\mu^2)s-\mu c,
 \label{eq:uniformT}\\
 \ell_s(q)&=\frac{4\mu(2s+1)}{\mathsf A}
                -\frac{4\mu T}{s^2\mathsf A}\,q.
 \label{eq:uniformell}
\end{align}
Indeed, $N=4\mu+2$ and
$T=N(2s+1)s-\mu\mathsf A$. Thus $\ell_s(0)=g_s(0)$, and
\begin{equation}\label{eq:uniformcalibration}
 \sum_{j=1}^N\ell_s(q_j)=\frac{4\mu^2}{s}=\frac{A_N}{s}
 \quad\text{whenever}\quad \sum_{j=1}^Nq_j=s.
\end{equation}
It remains to show that $\ell_s$ is a minorant of $g_s$ on the required
range of $q$.

The scalar difference has the factorization
\begin{equation}\label{eq:uniformfactor}
 g_s(q)-\ell_s(q)
 =\frac{4\mu q s^2}{\mathsf A\mathsf D_s(q)}
       R_N\left(s,\frac qs;\kappa\right),
\end{equation}
where
\begin{equation}\label{eq:uniformR}
 R_N(s,\tau;\kappa)=R_{N,0}(s;\kappa)+R_{N,1}(s;\kappa)\tau
                            +R_{N,2}(s;\kappa)\tau^2.
\end{equation}
For a coefficient-by-coefficient verification, write
\[
 d_1=8c\mathsf A-\mathsf B^2,\qquad d_2=16c^2,\qquad
 J=4c(2s+1)+4\mathsf A-4(s+1)\mathsf B.
\]
Then $\mathsf D_s(q)=\mathsf A^2+d_1q+d_2q^2$, and the numerator
of $g_s(q)/(4\mu)$ is $(2s+1)\mathsf A+Jq+16cq^2$.
Subtracting the affine term cancels the constant coefficient and gives
\begin{align}
 R_{N,0}&=
 \frac{s^2\mathsf A J-(2s+1)s^2d_1+T\mathsf A^2}{s^4},
 \label{eq:uniformR0}\\
 R_{N,1}&=
 \frac{16cs^2\mathsf A-(2s+1)s^2d_2+Td_1}{s^3},
 \label{eq:uniformR1}\\
 R_{N,2}&=\frac{Td_2}{s^2}.
 \label{eq:uniformR2}
\end{align}
The following estimates use these three scalar coefficients.

\section{Scalar exterior estimates in dimensions three through seven}
\label{app:lowalgebra}

We prove Lemma~\ref{lem:lowexterior} for the rational parameters
\eqref{eq:lowfixedparameters}. The coefficients
\eqref{eq:uniformR0}--\eqref{eq:uniformR2} reduce the exterior
estimate to positive quadratics and monotonicity in the squared radius.

\subsection{Dimensions three through five}\label{app:affine345}

Here $s\ge2$ and $\kappa=N/4$. We bound the coefficients of
$R_N(s,\tau)=R_0(s)+R_1(s)\tau+R_2(s)\tau^2$ separately.

\paragraph{Dimension three.}
Substitution in \eqref{eq:uniformR0}--\eqref{eq:uniformR2} gives
\begin{align*}
 R_0&=18s+\frac{579}{4}+\frac{4351}{32s}+\frac{285}{8s^2}
                      -\frac{3}{2s^3}-\frac2{s^4}>200,\\
 R_1&=43s-\frac{3007}{8}-\frac{573}{2s}-\frac{24}{s^2}
                      +\frac{16}{s^3}>-480,\\
 R_2&=320+\frac{120}{s}-\frac{32}{s^2}\ge320.
\end{align*}
The negative terms of $R_0$ have total absolute value at most $5/16$,
and $18s+4351/(32s)>98$ by the arithmetic--geometric mean inequality.
For $R_1$, use $43s\ge86$, $1/s\le1/2$ in its negative terms, and
discard the last positive term. Hence, for $\tau\ge0$,
\[
 R_3(s,\tau)>40(5-12\tau+8\tau^2)>0,\qquad
 32(5-12\tau+8\tau^2)=(16\tau-12)^2+16.
\]

\paragraph{Dimension four.}
The coefficients are
\begin{align*}
 R_0&=24s+190+\frac{335}{2s}+\frac{75}{4s^2}
                      -\frac{175}{8s^3}-\frac{125}{16s^4}>280,\\
 R_1&=40s-630-\frac{425}{s}+\frac{50}{s^2}
                      +\frac{125}{2s^3}>-780,\\
 R_2&=600+\frac{150}{s}-\frac{125}{s^2}\ge600.
\end{align*}
The negative terms of $R_0$ have total absolute value less than four,
whereas $24s+335/(2s)>126$. Also
$R_1\ge80-630-425/2=-1525/2>-780$. Therefore
\[
 R_4(s,\tau)>20(14-39\tau+30\tau^2)>0,\qquad
 120(14-39\tau+30\tau^2)=(60\tau-39)^2+159.
\]

\paragraph{Dimension five.}
We obtain
\begin{align*}
 R_0&=34s+\frac{953}{4}+\frac{5935}{32s}-\frac{483}{16s^2}
                      -\frac{531}{8s^3}-\frac{81}{4s^4}>360,\\
 R_1&=17s-\frac{7759}{8}-\frac{2253}{4s}+\frac{234}{s^2}
                      +\frac{162}{s^3},\\
 R_2&=1008+\frac{126}{s}-\frac{324}{s^2}\ge990.
\end{align*}
For $R_0$, the negative terms have total absolute value less than
eighteen and $34s+5935/(32s)>158$. Moreover,
\[
 R_1'=17+\frac1{s^2}\left(\frac{2253}{4}-\frac{468}s-
                                  \frac{486}{s^2}\right)>0,
 \qquad R_1(2)=-\frac{4555}{4}>-1140,
\]
and $R_2-990=18(s-2)(s+9)/s^2\ge0$. Thus
\[
 R_5(s,\tau)>30(12-38\tau+33\tau^2)>0,\qquad
 132(12-38\tau+33\tau^2)=(66\tau-38)^2+140.
\]
For all three dimensions, \eqref{eq:uniformfactor} and
\eqref{eq:uniformcalibration} now prove the exterior assertion in
Lemma~\ref{lem:lowexterior}.

\subsection{The six-dimensional simplex estimate}\label{app:affine6}

Take $\kappa=3/2$ and $s\ge3$. The same coefficient formulas give
\begin{align*}
 R_0'&>48-\frac{182}{s^2}>10,\qquad
 R_2'=\frac{1372}{s^3}>0,\\
 R_0'+R_1'&=16+\frac{490-924/s-588/s^2}{s^2}
                            +\frac{343}{2s^5}>16.
\end{align*}
The last inequality uses $490-924/3-588/9=350/3>0$. Hence
\[
 \partial_sR_6=(1-\tau)R_0'+\tau(R_0'+R_1')+\tau^2R_2'>0
 \qquad(0\le\tau\le1).
\]
At the initial radius,
\[
 R_0(3)=\frac{306305}{648}>472,\quad
 R_1(3)=-\frac{44333}{27}>-1642,\quad
 R_2(3)=\frac{13426}{9}>1491.
\]
It follows that $R_6(s,\tau)>472-1642\tau+1491\tau^2>18$, since
\[
 4\cdot454\cdot1491-1642^2=11492>0.
\]
The restriction $\tau=q/s\le1$ holds for every squared coordinate.
Summing the affine calibration proves $V_6(x)>0$ for $s(x)\ge3$.

\subsection{The seven-dimensional simplex estimate}\label{app:affine7}

For $\kappa=9/4$ and $s\ge5$, we have
\begin{align*}
 R_0'&>18-\frac{6111}{32s^2}>10,\qquad
 R_2'=\frac{729}{2s^2}+\frac{3645}{s^3}>0,\\
 R_0'+R_1'&=81+\frac{27621}{32s^2}-\frac{37593}{16s^3}
                      -\frac{53703}{32s^4}+\frac{3645}{8s^5}>81.
\end{align*}
The correction in the last line is greater than
$(700-2400/s-1700/s^2)/s^2>0$. Thus, as in dimension six,
$\partial_sR_7>0$ for $0\le\tau\le1$. At $s=5$,
\[
 R_0(5)=\frac{225999}{500}>451,\quad
 R_1(5)=-\frac{208647}{100}>-2087,\quad
 R_2(5)=\frac{13851}{5}>2770.
\]
Consequently $R_7(s,\tau)>451-2087\tau+2770\tau^2>18$, because
\[
 4\cdot433\cdot2770-2087^2=442071>0.
\]
This proves $V_7(x)>0$ for $s(x)\ge5$.

At $s=4$, a sum of integer squares has, up to signs and permutations,
the form $(2)$ or $(1^4)$. Formula~\eqref{eq:rationalfield} gives
\[
 a_7(4,1)=\frac{128}{53}<\frac{13}{5},\quad
 a_7(4,5)=\frac{16}{21},\quad
 a_7(4,9)=\frac{288}{613}<\frac12,
\]
\[
 a_7(4,3)=\frac{24}{17}<\frac{99}{70},\qquad
 a_7(4,7)=\frac{40}{73}<\frac{11}{20}.
\]
The axial point has one, twelve and one neighbours of squared radii
one, five and nine. The Boolean point has four, six and four neighbours
of squared radii three, five and seven. Their remainders are therefore
greater than, respectively,
\[
 \frac{199}{16}-\frac{13}{5}-12\frac{16}{21}-\frac12
       =\frac{109}{560}>0,
 \qquad
 \frac{199}{16}-4\frac{99}{70}-6\frac{16}{21}-4\frac{11}{20}
       =\frac1{112}>0.
\]
This completes the proof of Lemma~\ref{lem:lowexterior}.

\section{Analytic estimates for the eighth-dimensional field}\label{app:fielddata}

This appendix proves the uniform estimates used in
Section~\ref{subsec:N8gauge}. The interior is controlled by two
one-variable inequalities. The exterior is divided into $s\ge19$,
$12\le s\le19$, and the two possible low-radius families outside
$\mathcal P$. The normalization and remainder estimates in
Section~\ref{subsec:generalcauchy} give quantitative margins for the
chosen field. All terminating decimals
are exact rationals.

Table~\ref{tab:N8proofguide} lists the estimates used to assemble the
field. Here $s=|x|^2$, $q$ is a positive coordinate square, and
$F_B,F_E$ are the normalized numerators defined in
Appendix~\ref{n8a:app:interior}. The first two rows control the
interior chains; the last four cover every exterior vertex.

\begin{table}[H]
\centering
\small
\renewcommand{\arraystretch}{1.25}
\begin{tabular}{@{}>{\raggedright\arraybackslash}p{.24\textwidth}p{.46\textwidth}p{.22\textwidth}@{}}
\toprule
Region & Positive margin & Proof \\
\midrule
Boolean chain\newline $5\le s\le8$
 & $F_B>1/6000$, hence $B(s)>0$
 & Appendix~\ref{n8a:app:interior} \\
Second chain\newline $7\le s\le11$
 & $F_E>1/5000$, hence $E(s)>0$
 & Appendix~\ref{n8a:app:interior} \\
$s\ge19$, $0<q\le s$
 & $R(s,q)>5/(7s^3)$
 & Appendix~\ref{n8a:app:large} \\
$12\le s\le19$
 & $R(s,q)>1/(200s^3)$ for $q=1,4$ or $9\le q\le s$
 & Appendix~\ref{n8a:app:compact} \\
Two coordinates of magnitude $2$\newline $8\le s\le11$
 & $L_2(s)>0$
 & Lemma~\ref{n8a:e39:thm:low} \\
One coordinate of magnitude $3$\newline $9\le s\le11$
 & $L_3(s)>0$
 & Lemma~\ref{n8a:e39:thm:low} \\
\bottomrule
\end{tabular}
\caption{Estimates for the eighth-dimensional field. The remaining
interior junctions have the prescribed positive residuals
\eqref{n8a:eq:impose}; Appendix~\ref{n8a:app:denominators} supplies
the denominators and circle bounds for the first two rows.}
\label{tab:N8proofguide}
\end{table}

\subsection{Uniform positivity from Taylor and Cauchy bounds}\label{subsec:generalcauchy}

Uniform scalar estimates on real intervals control the required integer
radii simultaneously. First remove a denominator known to be positive. For a rational
function $Q(s)=P(s)/D(s)$ on $[a,b]\subset(0,\infty)$, assume $P$ is a
polynomial and $D>0$. Choose $\sigma\in[a,b]$ and an integer
$n\ge\deg P$. The normalized polynomial
\begin{equation}\label{eq:generalpolynomialnormalization}
 F(t)=\frac{(1+t)^nP(\sigma/(1+t))}{D(\sigma)},\qquad
 t=\frac\sigma s-1\in I:=\left[\frac\sigma b-1,\frac\sigma a-1\right]
\end{equation}
has the same sign as $Q(s)$, since $1+t>0$ on $I$.
Using the numerator polynomial allows a Cauchy circle larger than the
distance to possible complex zeros of $D$.
The finite binomial theorem~\cite[Formula~1.111]{GradshteynRyzhik}
gives a direct coefficient formula. If
$P(s)=\sum_{j=0}^{n}p_js^j$, with zero coefficients added when needed, then
\begin{equation}\label{eq:generalnormalizedcoefficients}
 [t^k]F(t)=\frac1{D(\sigma)}
       \sum_{j=0}^{n-k}p_j\sigma^j\binom{n-j}{k},\qquad 0\le k\le n.
\end{equation}
Indeed, the $j$th term becomes
$p_j\sigma^j(1+t)^{n-j}/D(\sigma)$. The initial coefficients can
therefore be recovered directly from the factored numerator, without
expanding the full transformed polynomial. Products of these polynomials
are expanded by the coefficient convolution
formula~\cite[Formula~0.316]{GradshteynRyzhik}.

We use the following elementary estimate repeatedly. Let $\T_mF$ denote
the Taylor polynomial at zero through degree $m\in\mathbb N_0$.

\begin{lemma}\label{lem:generaltaylorpositivity}
Suppose $F$ is holomorphic in a neighbourhood of $|t|\le R$, is real
on a real interval $I\subset[-r,r]$ with $0<r<R$, and satisfies
$|F(t)|\le M$ on $|t|=R$. Then
\begin{equation}\label{eq:generalcauchytail}
 |F(t)-\T_mF(t)|\le
 M\frac{(r/R)^{m+1}}{1-r/R}\qquad(t\in I).
\end{equation}
In particular, if $\T_mF\ge\mu$ on $I$ and
$\mu>M(r/R)^{m+1}/(1-r/R)$, then $F>0$ on $I$.
\end{lemma}
\begin{proof}
Cauchy's derivative formula~\cite[Eq.~1.9.31]{DLMF} bounds the
coefficient of $t^j$ by $MR^{-j}$. Summing the geometric series
over $j\ge m+1$~\cite[Formula~0.231.1]{GradshteynRyzhik} gives
\eqref{eq:generalcauchytail}. Subtract this bound from the lower
bound for $\T_mF$.
\end{proof}

The same argument is uniform in auxiliary parameters whenever $M$ and
the leading lower bound are uniform. The initial Taylor terms can be
kept together and estimated by completed squares, convexity or concavity;
only their joint lower bound matters.

Factored expressions give convenient circle bounds without expanding
all high-degree coefficients. Write $P(s)=\sum_jP_j(s)$ and factor a
summand as $P_j(s)=c\prod_\ell f_\ell(s)^{m_\ell}$. Set
$d_\ell=\deg f_\ell$, $d=\sum_\ell m_\ell d_\ell$ and
\[
 \widehat f_\ell(t)=\frac{(1+t)^{d_\ell}}{\sigma^{d_\ell}}
 f_\ell\left(\frac\sigma{1+t}\right),\qquad
 \nu_R\!\left(\sum_jp_jt^j\right)=\sum_j|p_j|R^j.
\]
For $R>1$, the normalized summand on $|t|=R$ is bounded by
\begin{equation}\label{eq:generalfactorbound}
 \frac{|c|\sigma^d}{D(\sigma)}\,
 \eta_R(n-d)\prod_\ell\nu_R(\widehat f_\ell)^{m_\ell},\qquad
 \eta_R(a)=\begin{cases}(R+1)^a,&a\ge0,\\(R-1)^a,&a<0.\end{cases}
\end{equation}
This follows from $R-1\le|1+t|\le R+1$ and the triangle inequality.
The degree of an individual $P_j$ may exceed $n$ if its leading terms
cancel in the sum. Such a normalized summand can have a pole at $t=-1$;
the circle estimate is still valid, and Cauchy's formula is applied
only to the full polynomial $F$ in \eqref{eq:generalpolynomialnormalization}.

\subsection{Uniform positivity on the two interior chains}\label{n8a:app:interior}

The functions $B,E$ are defined in Section~\ref{n8a:sec:chains},
with the exact parameters of Section~\ref{n8a:sec:profile}. We prove positivity on the complete
intervals in Lemma~\ref{n8a:lem:tails}.

\begin{proof}[Proof of Lemma~\ref{n8a:lem:tails}]
Write $B=N_B/D_B$, $E=N_E/D_E$, with the positive denominators
in Appendix~\ref{n8a:app:denominators}. The numerators are polynomials
of degrees $22$ and $31$ with rational coefficients.
Use the polynomial normalization \eqref{eq:generalpolynomialnormalization}
to define $F_B,F_E$, with
$(\sigma_B,n_B)=(13/2,22)$ and $(\sigma_E,n_E)=(9,31)$.
The transformed intervals are $I_B=[-3/16,3/10]$ and
$I_E=[-2/11,2/7]$. Thus it suffices to prove $F_X>0$ on $I_X$.
All terminating decimals below are exact rational numbers.
Writing $N_X(s)=\sum_jp_{X,j}s^j$, every coefficient used below is
given by \eqref{eq:generalnormalizedcoefficients}:
\[
 f_{X,k}:=[t^k]F_X(t)
 =\frac1{D_X(\sigma_X)}\sum_{j=0}^{n_X-k}
    p_{X,j}\sigma_X^j\binom{n_X-j}{k}.
\]
The numerator is specified by the four or six summands of
\eqref{eq:generalchainresiduals} and the denominators in
Appendix~\ref{n8a:app:denominators}. This formula with the exact rational parameters in
\eqref{n8a:eq:gamma}--\eqref{n8a:eq:ab} reconstructs the coefficient
comparisons used in the proof.

\paragraph{The Boolean chain.}
Dividing $T_B=\T_6F_B$ by
$(t+3/50)^2(t-3/10)^2$ gives
\[
 T_B=(t+3/50)^2(t-3/10)^2(a_Bt^2+b_Bt+c_B)
                   +d_3t^3+d_2t^2+d_1t+d_0,
\]
where
\[
 a_B>0.85,\quad0<b_B<0.88,\quad c_B>0.35,
\]
\[
 d_0>0.000247,\quad0<d_1<0.000269,\quad
 d_2>0.00322,\quad-0.00712<d_3<0.
\]
For example, the monic divisor expands as
\[
 (t+3/50)^2(t-3/10)^2
 =t^4-\frac{12}{25}t^3+\frac{27}{1250}t^2
       +\frac{27}{3125}t+\frac{81}{250000}.
\]
Thus the three quotient coefficients are recovered by
\[
 a_B=f_{B,6},\qquad
 b_B=f_{B,5}+\frac{12}{25}a_B,\qquad
 c_B=f_{B,4}+\frac{12}{25}b_B-\frac{27}{1250}a_B;
\]
subtracting the product determines $d_0,\ldots,d_3$.
Substitution of \eqref{n8a:eq:gamma}--\eqref{n8a:eq:ab} then gives
the displayed rational bounds. The quadratic factor is positive, since
$c_B-b_B^2/(4a_B)>0.35-0.88^2/3.4>0$.
As $t\le3/10$, the remainder exceeds
$t^2/1000+d_1t+0.000247$. Completing the square gives
\[
 T_B>0.000247-\frac{0.000269^2}{4/1000}>\frac9{40000}.
\]
The coefficients of degrees $7$ through $10$ are positive; the
odd ones are less than $3/4$ and $18/125$. Their possible negative
contribution, confined to $t\le0$, has magnitude less than
\[
 \frac34(3/16)^7+\frac{18}{125}(3/16)^9<\frac1{100000}.
\]
Thus $\T_{10}F_B>9/40000-1/100000$ throughout $I_B$.

\paragraph{The second chain.}
Let $T_E=\T_6F_E$. For $t=-x$, $0\le x\le2/11$, its
initial coefficients give
\[
 T_E'''(-x)>0.48+6.78x-82.3x^3>0,\qquad T_E'(0)<0.
\]
Hence $T_E'$ is convex, and its nonpositive sublevel set is an
interval containing the right endpoint. Thus $T_E$ can increase
and then decrease, or decrease throughout; its minimum is attained
at an endpoint. Both endpoint values satisfy
$T_E(-2/11)>3/10000$, $T_E(0)>3/10000$.

For $0\le t\le2/7$, put $x=t-1/8\in[-1/8,9/56]$. Then
\[
 T_E=b_0+b_1x+x^2(b_2+b_3x+b_4x^2+b_5x^3+b_6x^4),
\]
where
\[
 b_0>0.000264,\quad |b_1|<0.000082,\quad b_2>0.01045,
\]
\[
 -0.03049<b_3<0,\quad-0.1095<b_4<0,
 \quad0<b_5<0.53421,\quad b_6>0.
\]
The expression in parentheses exceeds
\[
 0.01045-0.03049\frac9{56}-0.1095\left(\frac9{56}\right)^2
                       -\frac{0.53421}{512}>\frac1{625}.
\]
Consequently
\[
 T_E>0.000264-\frac{0.000082^2}{4/625}>\frac1{4000}.
\]
When adding degrees $7$ through $12$, the terms of degrees $8$
and $12$ are nonnegative; the other coefficients $a_j$ satisfy
\[
 0<a_7<0.458,\quad-0.109<a_9<0,\quad-0.019<a_{10}<0,
 \quad0<a_{11}<0.019.
\]
The negative contribution has magnitude less than
$0.458(2/11)^7+0.019(2/11)^{10}+0.019(2/11)^{11}$ if $t\le0$,
and less than $0.109(2/7)^9+0.019(2/7)^{10}$ if $t\ge0$.
Both bounds are less than $1/200000$. Therefore
$\T_{12}F_E>1/4000-1/200000$ throughout $I_E$.

\paragraph{The full Taylor remainders.}
The factor estimates in Appendix~\ref{n8a:app:denominators} give
\[
 |F_B(t)|<36000\ (|t|=2),\qquad
 |F_E(t)|<300000000\ (|t|=3).
\]
Lemma~\ref{lem:generaltaylorpositivity} gives
\[
 |F_B-\T_{10}F_B|\le36000\frac{(3/20)^{11}}{1-3/20}<\frac1{25000},
\]
\[
 |F_E-\T_{12}F_E|\le300000000\frac{(2/21)^{13}}{1-2/21}
                                                <\frac1{50000}.
\]
Finally,
\[
 F_B>\frac9{40000}-\frac1{100000}-\frac1{25000}>\frac1{6000},
 \qquad
 F_E>\frac1{4000}-\frac1{200000}-\frac1{50000}>\frac1{5000}.
\]
This proves both uniform inequalities.
\end{proof}

\subsection{The continuous exterior: \texorpdfstring{$s\ge19$}{s >= 19}}\label{n8a:app:large}
We prove $\Rcal(s,q)>5/(7s^3)$ for $s\ge19$ and $0<q\le s$.
The angular variables in this argument may be real.

\subsubsection{Polynomial normalization and a uniform remainder}
Write $u=z^2$, $\tau=X^2$ and, for either sign,
\[
 d_\pm=z^2\pm2zX,\qquad a_\pm=2+d_\pm.
\]
Define the normalized polynomials
\begin{align*}
 N_\pm=\frac{d_\pm}{15360}\bigl(&720a_\pm^4-4134a_\pm^3u
 +3885d_\pm^2a_\pm^2-3440a_\pm^2u^2\\
 &+7688d_\pm^2a_\pm u-4500d_\pm^4\bigr),\\
 D_\pm&=\frac{160}{15360}a_\pm^3
                  (3a_\pm^2-20a_\pm u+24d_\pm^2),
 \qquad H_\pm=D_\pm+N_\pm.
\end{align*}
Then $\Psi(v_\pm,w_\pm)=N_\pm/D_\pm$ when
$v_\pm=d_\pm/a_\pm$ and $w_\pm=u/a_\pm$.
Put $n=N_\pm|_{X=0}$, $h=H_\pm|_{X=0}$,
$J=H_+H_-$ and $\Bc=Jh$. Define
\[
 K=2h(N_+H_-+N_-H_+)-4nJ,
\]
\begin{equation}\label{n8a:e37:eq:Adef}
 \Ac(u,\tau)=\frac{K/\tau+32nJ-9u\Bc-2u^4\Bc}{u^2}.
\end{equation}
Despite the displayed divisions, $\Ac$ is a polynomial in $u,\tau$.
Indeed, the products are even in $z$ and $X$; $K$ vanishes at
$X=0$, so it is divisible by $\tau$.
Near $u=0$, $H_\pm(0,X)=1$ and
\[
 \frac{2N_+}{H_+}+\frac{2N_-}{H_-}
                  =3(1-5\tau)u+O(u^2).
\]
The terms of order $u$ in the numerator of \eqref{n8a:e37:eq:Adef}
are $-15u+24u-9u=0$, and the constant term vanishes.
The numerator is therefore divisible by $u^2$.
Substitution in the definition of $\Rcal$ gives
\begin{equation}\label{n8a:e37:eq:identity}
 \Rcal(1/u,\tau/u)=\frac{\Ac(u,\tau)}{\Bc(u,\tau)}
                              +\frac{u^2}{4\tau}.
\end{equation}
Working with the polynomial $\Ac$ prevents complex zeros of
$\Bc$ from restricting the radius in Cauchy's estimate.

\paragraph{A bound for the entire remainder.}
On $|u|=|\tau|=1$, take $|z|=|X|=1$. Then
$|d_\pm|\le3$, $|a_\pm|\le5$, and the preceding formulas give
\[
 |N_\pm|<516,\qquad |H_\pm|<1025.
\]
For example, the bound for $N_\pm$ is
\[
 \frac{3}{15360}\bigl(720\cdot5^4+4134\cdot5^3
 +3885\cdot9\cdot25+3440\cdot25+7688\cdot9\cdot5
 +4500\cdot81\bigr)<516.
\]
For $H_\pm$, add
$160\cdot5^3(3\cdot25+20\cdot5+24\cdot9)/15360$.
Also,
\begin{align*}
 n&=\frac34u-\frac{209}{320}u^2-\frac{7637}{3840}u^3
                  -\frac{59}{480}u^4+\frac{73}{5120}u^5,\\
 h&=1-\frac1{12}u-\frac{2707}{960}u^2-\frac{10517}{3840}u^3
                  +\frac{11}{480}u^4+\frac{1339}{15360}u^5.
\end{align*}
Thus $|n|<4$ and $|h|<7$ on $|u|=1$. Applying these bounds
before dividing in \eqref{n8a:e37:eq:Adef} gives
\begin{align*}
 |\Ac(u,\tau)|
 &\le4\cdot7\cdot516\cdot1025
              +(36\cdot4+11\cdot7)1025^2\\
 &=246997325<250000000=:M.
\end{align*}
The maximum modulus principle~\cite[\S1.10(v)]{DLMF} extends
this bound to $|\tau|\le1$.
Writing $\Ac(u,\tau)=\sum_{\ell\ge0}c_\ell(\tau)u^\ell$, apply
\eqref{eq:generalcauchytail} with $R=1$, $r=1/s$ and $m=9$,
uniformly for $0\le\tau\le1$:
\begin{equation}\label{n8a:e37:eq:tail}
 \left|\sum_{\ell\ge10}c_\ell(\tau)s^{-\ell}\right|
 \le\frac{250000000}{s^{10}(1-1/s)}\qquad(s>1).
\end{equation}
This single bound controls the full Taylor remainder.

\subsubsection{A joint estimate for the leading terms}
The leading angular terms must be kept together to retain their
positive joint contribution. Put
\[
 P(s,\tau)=s^3c_0(\tau)+s^2c_1(\tau)+sc_2(\tau)+c_3(\tau).
\]
The useful square decomposition is
\begin{equation}\label{n8a:e37:eq:square}
 P=a\bigl(\tau^2+b(s)\tau+k(s)\bigr)^2
                  +D_0(s)\tau^2+E_0(s)\tau+F_0(s),
\end{equation}
where
\[
 a=\frac{12883825}{512},\qquad
 b(s)=-\frac{14338500s+2773706087}{3092118000},\qquad
 k(s)=\frac{s+2}{125},
\]
\begin{align*}
 D_0(s)&=\frac{5991011889750000s^2-225078213152067000s
                    +2357806647364246831}{379959459840000},\\
 E_0(s)&=-\frac{369436500s^2-13164488661s+123219409043}{23040000},\\
 F_0(s)&=\frac{2592000s^3+75008952s^2-3570362892s+26748831983}
                      {69120000}.
\end{align*}
Expand \eqref{n8a:e37:eq:Adef} through order $u^3$ and complete
the fourth- and third-degree terms in $\tau$. The affine choice
of $k(s)$ leaves a quadratic remainder with simple positive bounds.

We shall use the bound proved below,
\begin{equation}\label{n8a:e37:eq:c4bound}
 c_4(\tau)\ge-9000(\tau-1/4)^2-50
                     \qquad(0\le\tau\le1).
\end{equation}
For $s\ge19$, the right-hand side is nonpositive; hence
\[
 P+\frac{c_4}{s}\ge
 a(\tau^2+b\tau+k)^2+d\tau^2+e\tau+f,
\]
where we set
\[
 d=D_0-\frac{9000}{19},\qquad
 e=E_0+\frac{4500}{19},\qquad
 f=F_0-\frac{1225}{38}.
\]
With $r=s-19\ge0$, the explicit formulas give
\begin{align*}
 d&\ge D(r):=\frac{63}{4}r^2+\frac{27}{4}r+168,\\
 |e|&\le E(r):=\frac{65}{4}r^2+38r+44,\\
 f&\ge F(r):=\frac3{80}r^3+\frac{16}{5}r^2+30r+22.
\end{align*}
Here $e<0$; substitute $s=19+r$ in $D_0,E_0,F_0$ and bound
their coefficients in the indicated directions.
Since $\tau\ge0$, completing a second square gives
\[
 d\tau^2+e\tau+f\ge D\tau^2-E\tau+F
                            \ge F-\frac{E^2}{4D}.
\]
The remaining scalar inequality has a short proof:
\begin{align*}
 4D(F-10)-E^2
  ={}&r^3\left(\frac{189}{80}r^2-\frac{1229}{20}r
                                      +\frac{3833}{5}\right)
          +\frac{4212}{5}r^2+17140r+6128\\
 \ge{}&r^3(2r^2-62r+766)+840r^2+17000r+6000>0,
\end{align*}
since
\[
 2r^2-62r+766=2(r-31/2)^2+571/2>0.
\]
We have proved the joint estimate
\begin{equation}\label{n8a:e37:eq:joint}
 s^3c_0+s^2c_1+sc_2+c_3+\frac{c_4}{s}>10
            \qquad(s\ge19,\ 0\le\tau\le1).
\end{equation}

\subsubsection{Angular bounds for the initial corrections}
The identities for $c_4,\ldots,c_9$ are displayed below.
The fifth term retains its angular dependence; constant bounds
suffice for the following terms.

\paragraph{A parabolic bound for the fifth term.}
To prove \eqref{n8a:e37:eq:c4bound}, put
$Q=c_4+9000(\tau-1/4)^2+50$.
The coefficient of $\tau^5$ in $c_4$ is negative.
Replacing that term by the same coefficient times $\tau^4$
gives a lower bound for $Q$, since $\tau^5\le\tau^4$.
With $y=\tau-1/4$, this lower bound is
\[
 y^2T(y)+\delta y+\varepsilon,
\]
\[
 \begin{gathered}
 T(y)=-\frac{680238563}{30720}y^2
       +\frac{6671439823}{2949120}y
       +\frac{1185682025503}{98304000},\\
 \delta=-\frac{2278148363921}{3538944000},\qquad
 \varepsilon=\frac{96126342407}{4718592000}.
 \end{gathered}
\]
In particular, $|\delta|<644$ and $\varepsilon>20$.
On $-1/4\le y\le1/4$, concavity of $T$ and
\[
 T(-1/4)=\frac{1491058471567}{147456000}>10000,\qquad
 T(1/4)=\frac{828922233571}{73728000}>10000
\]
imply
\[
 y^2T(y)+\delta y+\varepsilon
       >10000y^2-644|y|+20
       \ge20-\frac{644^2}{40000}>0.
\]
On $1/4\le y\le3/4$, the function $U(y)=yT(y)$ is concave:
its second derivative is decreasing and
$U''(1/4)=-42305736713/1474560<0$.
Moreover,
\[
 U(1/4)=\frac{828922233571}{294912000}>975,\qquad
 U(3/4)=\frac{64019303839}{65536000}>975.
\]
Thus $U(y)>975$ and
$y^2T(y)+\delta y+\varepsilon=y(U(y)+\delta)+\varepsilon>0$.
This proves \eqref{n8a:e37:eq:c4bound} throughout $[0,1]$.

\paragraph{Two useful factorizations.}
For $c_5$, use the identity
\begin{align*}
884736000\,[c_5(\tau)-c_5(0)]
 =\tau\bigl[1515070440000\tau^4
                       +(1-\tau)H(\tau)+48592374149\bigr],\\
 H(\tau)=70172871887880\tau^2-68575340184363\tau
                                      +20003977080782
\end{align*}
and the bound
\[
 H(\tau)\ge10^{12}(70\tau^2-69\tau+20)
       \ge10^{12}\left(20-\frac{69^2}{280}\right)>0.
\]
Thus $c_5(\tau)\ge c_5(0)>-1400$.
For $c_6$, consider the polynomial
$[c_6(\tau)-c_6(0)]/\tau$, extended at $\tau=0$. Discard
its positive fourth-degree term and use $\tau^3\le\tau^2$
in its negative cubic term. This gives
\begin{align*}
3538944000\,[c_6(\tau)-c_6(0)]
 &\ge\tau(198885330077232\tau^2
                -258059035997625\tau+85138567220739)\\
 &\ge10^{12}\tau(198\tau^2-259\tau+85)\ge0,
\end{align*}
since $259^2<4\cdot198\cdot85$. Hence
$c_6(\tau)\ge c_6(0)>-2000$.
Finally, retain only the negative terms of the next three
polynomials and use $0\le\tau\le1$. Altogether,
\begin{equation}\label{n8a:e37:eq:tailinitial}
 c_5>-1400,\quad c_6>-2000,\quad c_7>-26000,\quad
 c_8>-34000,\quad c_9>-40000.
\end{equation}

\subsubsection{Uniform bounds for the numerator and quotient}
By \eqref{n8a:e37:eq:tail}, \eqref{n8a:e37:eq:joint} and \eqref{n8a:e37:eq:tailinitial},
\[
 s^3\Ac(1/s,\tau)>10-\Lambda(s),
\]
\[
 \Lambda(s)=\frac{1400}{s^2}+\frac{2000}{s^3}
 +\frac{26000}{s^4}+\frac{34000}{s^5}+\frac{40000}{s^6}
 +\frac{250000000}{s^7(1-1/s)}.
\]
Every summand decreases for $s\ge19$, and
\[
 \Lambda(19)=\frac{1981154600}{423412929}<5.
\]
Consequently,
\begin{equation}\label{n8a:e37:eq:Apositive}
 \Ac(1/s,\tau)>\frac5{s^3}
                \qquad(s\ge19,\ 0\le\tau\le1).
\end{equation}

It remains to bound the real denominator. If $s\ge19$, $0\le q\le s$,
then $s+1\pm2\sqrt q>0$ and
\[
 0<w_\pm<\frac1{30},\qquad |v_\pm|<\frac13.
\]
For the first bound use
$2s+1-2\sqrt s\ge39-2\sqrt{19}>30$;
for the second, each branch satisfies
$|v_\pm|\le1/(\sqrt s-1)<1/3$.
Thus $1+8v_\pm^2-20w_\pm/3>7/9$, and the triangle inequality
applied to $\Psi$ gives
\[
 |\Psi(v_\pm,w_\pm)|
 <\frac13\left[\frac32+\frac{111}{2400}+\frac{125}{288}
 +\frac{\frac{25}{10800}+\frac{169}{2160}+\frac{175}{648}}{7/9}
 \right]=\frac{2269}{2800}<1.
\]
Consequently, $D_\pm,H_\pm,h$ are positive and $\Bc>0$.
The even part $F_H=(H_++H_-)/2$ is
\begin{align*}
 F_H={}&1+(24\tau-1/12)u
 +\frac{49140\tau^2+41973\tau-2707}{960}u^2\\
 &+\frac{89916\tau^2+175996\tau-10517}{3840}u^3
 +\frac{18027\tau+44}{1920}u^4+\frac{1339}{15360}u^5.
\end{align*}
For $0<u\le1/19$, $0\le\tau\le1$,
\[
 0<F_H\le1+24u+95u^2+70u^3+10u^4+u^5<\frac{13}{5},
 \qquad 0<h<\frac{1001}{1000}.
\]
Since $J\le F_H^2$, it follows that
\[
 0<\Bc=Jh<\left(\frac{13}{5}\right)^2\frac{1001}{1000}<7.
\]
Identities \eqref{n8a:e37:eq:identity} and \eqref{n8a:e37:eq:Apositive} prove
$\Rcal(s,q)>5/(7s^3)$. With $q_i=x_i^2$ and
$n_0=\#\{i:q_i=0\}$, we obtain
\[
 V_{\beta_8}(x)=s^{-3}\sum_{q_i>0}q_i\Rcal(s,q_i)
                                +\frac{n_0}{4s^4}
                       >\frac5{7s^5},
\]
since $\sum_iq_i=s$. This proves the estimate for $s\ge19$.

\subsubsection{Identities for the initial corrections}
Identity \eqref{n8a:e37:eq:square} specifies $c_0,c_1,c_2,c_3$ jointly.
The six further identities used above follow directly from
\eqref{n8a:e37:eq:Adef}:
\begin{align*}
221184000\,c_4={}&-281961000000\tau^5-4615756653600\tau^4\\
 &+5398075640325\tau^3-1534784052762\tau^2\\
 &-81024075575\tau+44413522196,\\[2pt]
884736000\,c_5={}&1515070440000\tau^5-70172871887880\tau^4\\
 &+138748212072243\tau^3-88579317265145\tau^2\\
 &+20052569454931\tau-1219425440305,\\[2pt]
3538944000\,c_6={}&11508081570000\tau^5-15907810811256\tau^4\\
 &+214793140888488\tau^3-258059035997625\tau^2\\
 &+85138567220739\tau-6723112848829,\\[2pt]
14155776000\,c_7={}&-22515839010000\tau^5+156094387026672\tau^4\\
 &-293965298863956\tau^3-35723437906544\tau^2\\
 &+103410903294389\tau-14825254206959,\\[2pt]
56623104000\,c_8={}&-2005056000000\tau^5+632852753068464\tau^4\\
 &-1622170183932224\tau^3+1333626758467276\tau^2\\
 &-236041202593912\tau-17516203206927,\\[2pt]
56623104000\,c_9={}&-23362975170000\tau^5+756424454144640\tau^4\\
 &-1791986582074888\tau^3+1492984087546880\tau^2\\
 &-414958155540623\tau-1295701900318.
\end{align*}
For the last three polynomials, the sums of the absolute values
of their negative coefficients are less than $26000$, $34000$ and
$40000$, respectively. These bound three initial corrections.
All terms of order ten and higher are covered by the single
Cauchy estimate \eqref{n8a:e37:eq:tail}.

\subsection{The compact exterior: \texorpdfstring{$12\le s\le19$}{12 <= s <= 19}}\label{n8a:app:compact}
We prove $\Rcal(s,q)>1/(200s^3)$ for $q=1$, $q=4$ and
$9\le q\le s$. These regions contain every positive integer square.

\subsubsection{A regular numerator and Cauchy bounds}
Use the preceding identities and bounds.
With $u=1/s$, $\tau=q/s$, write
\begin{equation}\label{n8a:e38:eq:AB}
 \Rcal(s,q)=\frac{\Ac(u,\tau)}{\Bc(u,\tau)}+\frac1{4sq},
 \qquad \Ac(u,\tau)=\sum_{j\ge0}c_j(\tau)u^j.
\end{equation}
Here $\Ac,\Bc$ are the polynomials in \eqref{n8a:e37:eq:Adef}--\eqref{n8a:e37:eq:identity}.
Their exact divisibility was proved above.

For $0\le\tau\le1$, the preceding bounds and the initial expansion give
\begin{align}
 c_0&=3/80,\quad
 c_1=\frac{4175}{256}\left(\tau-\frac{11457}{20875}\right)^2
             -\frac{5553221}{2505000}>-\frac{111}{50},\label{n8a:e38:eq:c1}\\
 c_4&\ge-9000(\tau-1/4)^2-50,\label{n8a:e38:eq:c4old}\\
 c_5&>-1400,\quad c_6>-2000,\quad c_7>-26000,
       \quad c_8>-34000,\quad c_9>-40000.\label{n8a:e38:eq:oldbounds}
\end{align}
The proof of \eqref{n8a:e38:eq:c4old} uses concavity and a completed
square; the bounds for $c_5,c_6$ use their factorizations.
All coefficient identities are those of Appendix~\ref{n8a:app:large}.

For small angles we improve the Cauchy constant.
On $|u|=1$, $|\tau|=\alpha$, we have $|d_\pm|\le1+2\sqrt\alpha$,
$|a_\pm|\le3+2\sqrt\alpha$ and $|n|<4$, $|h|<7$.
If $|N_\pm|<N_\alpha$, $|H_\pm|<H_\alpha$, then
\begin{equation}\label{n8a:e38:eq:M}
 |\Ac|\le\frac{28N_\alpha H_\alpha}{\alpha}
 +\left(\frac{16}{\alpha}+205\right)H_\alpha^2.
\end{equation}
For $\alpha=1/12$, insert $|d_\pm|<8/5$, $|a_\pm|<18/5$.
Then $N_\alpha=62$, $H_\alpha=145$, and the right-hand side
is $11367565<12000000$.
For $\alpha=1/3$, use $|d_\pm|<13/6$, $|a_\pm|<25/6$;
then $N_\alpha=162$, $H_\alpha=349$, and the right-hand side is
$35564845<36000000$. Appendix~\ref{n8a:app:large} gives $M_1=250000000$.
The maximum modulus principle in $\tau$ and
\eqref{eq:generalcauchytail} in $u$ give
\begin{equation}\label{n8a:e38:eq:tail}
 \left|\sum_{j\ge10}c_j(\tau)s^{-j}\right|
 \le\frac{M_\alpha}{s^{10}(1-1/s)}\qquad(0\le\tau\le\alpha),
\end{equation}
where $M_{1/12}=12000000$, $M_{1/3}=36000000$ and
$M_1=250000000$. This bounds the entire remaining tail.

\subsubsection{The denominator and two lower bounds}
For $s\ge12$, $0\le q\le s$,
\[
 0<w_\pm<1/18,\qquad |v_\pm|<1/3.
\]
Indeed, $2s+1-2\sqrt s\ge25-2\sqrt{12}>18$;
the two branches of $v$ are bounded using
$4\sqrt s<s+2$ and $2\sqrt s+1<s$ for $s\ge12$.
The triangle inequality gives
\[
 |\Psi|<\frac13\left[\frac32+\frac{111}{1440}+\frac{125}{288}
 +\frac{\frac{25}{3888}+\frac{169}{1296}+\frac{175}{648}}{17/27}
 \right]=\frac{1807}{2040}<1.
\]
Thus $H_\pm,h$ and $\Bc$ are positive. The even expression
$F_H=(H_++H_-)/2$ in Appendix~\ref{n8a:app:large} satisfies, for $u\le1/12$,
\[
 F_H\le1+24u+95u^2+70u^3+10u^4+u^5<4,
 \qquad h<1001/1000.
\]
Consequently,
\begin{equation}\label{n8a:e38:eq:Bupper}
 0<\Bc\le F_H^2h<17.
\end{equation}

For $q=1,4$, retain the positive term in \eqref{n8a:e38:eq:AB}.
We shall prove, for $12\le s\le19$,
\begin{equation}\label{n8a:e38:eq:Blower}
 \Bc(1/s,1/s)>9/10,\qquad \Bc(1/s,4/s)>1.
\end{equation}
When $d=cu$, the polynomial $H$ is
\begin{align*}
 \mathsf H_c(u)={}&1+(13c/4-10/3)u+(6c^2-8467c/960)u^2\\
 &+(1635c^3/256-5267c^2/640-43c/48)u^3\\
 &+(819c^4/256-2919c^3/1280-43c^2/48)u^4\\
 &+(295c^5/1024+59c^4/2560-43c^3/192)u^5.
\end{align*}
Thus $h=\mathsf H_1$, $J_1=\mathsf H_3\mathsf H_{-1}$,
$J_4=\mathsf H_5\mathsf H_{-3}$ and $\Bc(u,qu)=J_q(u)h(u)$.
For $0\le u\le1/12$,
\[
 h\ge1-u/12-(16/5)u^2>19/20.
\]
Multiplying the two quintics gives
\begin{align*}
 J_1&=1-u/6+\frac{169}{1440}u^2-\frac{6119}{1440}u^3
                -\frac{106671667}{921600}u^4+u^5T_1(u),\\
 J_4&=1-u/6+\frac{25009}{1440}u^2-\frac{2843}{1440}u^3
                -\frac{11740087}{36864}u^4+u^5T_4(u).
\end{align*}
The polynomials $T_1,T_4$ are displayed below. Bounding their
negative terms gives
\[
 T_1\ge-113-1188u^2-39u^3-3u^5>-122,
 \qquad T_4\ge4200-164000u^2-56000u^5>3000.
\]
Therefore
\[
 J_1\ge1-u/6-2u^2>19/20,\qquad
 J_4\ge1-u/6+14u^2.
\]
The first bound gives $J_1h>(19/20)^2>9/10$. For the second,
\begin{align*}
 J_4h&\ge(1-u/6+14u^2)(1-u/12-(16/5)u^2)\\
 &=1-u/4+\frac{3893}{360}u^2-\frac{19}{30}u^3-\frac{224}{5}u^4
 \ge1-u/4+10u^2>1,
\end{align*}
since $u\ge1/19$. This proves \eqref{n8a:e38:eq:Blower}.

\subsubsection{The squared coordinate \texorpdfstring{$q=1$}{q=1}}
Put
\[
 P(s,\tau)=s^3c_0(\tau)+s^2c_1(\tau)+sc_2(\tau)+c_3(\tau).
\]
For $0\le\tau\le1/12$, the negative terms of $c_4$ give
\[
 c_4\ge200-367\tau-6940\tau^2-20870\tau^4-1275\tau^5
 \ge\frac{9970751}{82944}>120.
\]
Using \eqref{n8a:e38:eq:Blower}, the initial part of the full numerator
is bounded below by
\[
 P(s,1/s)+\frac{c_4(1/s)}s+\frac{s^2\Bc(1/s,1/s)}4
   \ge L_1(s):=P(s,1/s)+\frac9{40}s^2+\frac{120}s.
\]
This function is increasing for $s\ge12$. Indeed,
\begin{align*}
 4608000s^5L_1'(s)={}&518400s^7+s^5(26916960s-290830260)\\
 &+s^2(27329109913s-249065631610)\\
 &+624083869575s-463817700000>0.
\end{align*}
Each bracket is positive for $s\ge12$.
The explicit formula for $L_1$ below gives $L_1(12)>51/4$.
By \eqref{n8a:e38:eq:oldbounds} and \eqref{n8a:e38:eq:tail},
\[
 s^3\left(\Ac+\frac{\Bc}{4s}\right)>
 \frac{51}{4}-\Lambda_1(s),
\]
\[
 \Lambda_1(s)=\frac{1400}{s^2}+\frac{2000}{s^3}
 +\frac{26000}{s^4}+\frac{34000}{s^5}+\frac{40000}{s^6}
 +\frac{12000000}{s^7(1-1/s)}.
\]
This function decreases, and
\[
 \Lambda_1(12)=\frac{405725}{32076}<\frac{253}{20}.
\]
Thus $s^3(\Ac+\Bc/(4s))>1/10$. Dividing by
$\Bc<17$ gives
\begin{equation}\label{n8a:e38:eq:q1}
 \Rcal(s,1)>\frac1{170s^3}>\frac1{200s^3}.
\end{equation}

\subsubsection{The squared coordinate \texorpdfstring{$q=4$}{q=4}}
Now $\tau=4/s\in[4/19,1/3]\subset[1/5,1/3]$.
On this interval we improve \eqref{n8a:e38:eq:oldbounds} to
\begin{equation}\label{n8a:e38:eq:q4coeff}
 c_5>-112,\qquad c_6>0,\qquad c_7>0.
\end{equation}
For the first bound, put $x=\tau-1/4\in[-1/20,1/12]$.
Expand $c_5'$ in powers of $x$. Its coefficients, from degree four
to degree zero, satisfy
\[
 \begin{gathered}
 0<a_4<8600,\qquad -309000<a_3<0,\qquad 0<a_2<236000,\\
 -24000<a_1<-23900,\qquad a_0<-2900.
 \end{gathered}
\]
Then $x^4\le x^2/144$, $a_3x^3\le(309000/20)x^2$ and
$(a_1+24000)x\le100/12$, so
\[
 c_5'(\tau)\le252000x^2-24000x-2800<0.
\]
The last quadratic is convex and negative at both endpoints.
Hence $c_5(\tau)\ge c_5(1/3)=-18498039257/165888000>-112$.

For $c_6$, use the estimate from Appendix~\ref{n8a:app:large}:
\[
 3538944000[c_6(\tau)-c_6(0)]
 \ge10^{12}\tau(198\tau^2-259\tau+85).
\]
The cubic on the right, without the factor $10^{12}$, is concave
on $[1/5,1/3]$ and its endpoint values are at least $62/9$.
Since $c_6(0)=-6723112848829/3538944000$, we have $c_6>0$.
For $c_7$, group its fourth- and fifth-degree terms, whose sum
is nonnegative, and use $\tau^3\le\tau^2/3$. Thus
\[
 14155776000c_7\ge10^{12}(-134\tau^2+103\tau-15)>0.
\]
This quadratic is concave and positive at both endpoints.

By \eqref{n8a:e38:eq:c4old} and \eqref{n8a:e38:eq:Blower}, the initial part is at least
\[
 L_4(s):=P(s,4/s)+\frac{s^2}{16}
                 -\frac{9000(4/s-1/4)^2+50}{s}>3
                     \qquad(12\le s\le19).
\]
To prove the last inequality uniformly, put
$N_4=13824000s^4L_4$, displayed below.
If $r=s-16\ge0$, its expansion gives
\[
 N_4/10^{12}\ge3.94r^2-1.74r+7.23
       \ge7.23-\frac{1.74^2}{4\cdot3.94}>7.
\]
If $x=16-s\in[0,4]$, the constant term exceeds $7\cdot10^{12}$,
the linear term is positive and the remainder is $10^{12}x^2Q(x)$, where
\[
 Q(x)\ge0.092x^2-1.15x+3.94>0.
\]
Use $x^3\le4x^2$, $x^5\le4x^4$ in the two higher negative
terms of $Q$; the final quadratic has negative discriminant.
The expansions are displayed below. Thus $N_4>7\cdot10^{12}$ and
\[
 L_4>\frac{7\cdot10^{12}}{13824000\cdot19^4}>3.
\]
All the preceding decimals are exact rational numbers.

Using \eqref{n8a:e38:eq:q4coeff} and Cauchy with $\alpha=1/3$,
\[
 s^3\left(\Ac+\frac{\Bc}{16s}\right)>3-\Lambda_4(s),\qquad
 \Lambda_4(s)=\frac{112}{s^2}+\frac{34000}{s^5}
 +\frac{40000}{s^6}+\frac{36000000}{s^7(1-1/s)}.
\]
This function decreases, and
$\Lambda_4(12)=259667/128304<5/2$.
Therefore $\Rcal(s,4)>1/(34s^3)>1/(200s^3)$.

\subsubsection{A uniform estimate for all larger coordinates}
If $12\le s\le19$ and $q\ge9$, then $\tau=q/s\ge9/19>2/5$.
The following argument holds on the entire rectangle
$12\le s\le19$, $2/5\le\tau\le1$.
Replace the negative fifth-degree term of $c_4$ by the same
coefficient times $\tau^4$ and define
\begin{align*}
 221184000Q_4(\tau):={}&-4897717653600\tau^4+5398075640325\tau^3\\
 &-1534784052762\tau^2-81024075575\tau+44413522196,
\end{align*}
so that $c_4\ge Q_4$. Also $Q_4<0$ for $\tau\ge2/5$:
with $x=\tau-1/2$, we have
\[
 -Q_4=x^2T(x)+d x+e,
\]
\[
 \begin{gathered}
 T=\frac{680238563}{30720}x^2+\frac{11726292445}{589824}x
                       +\frac{174277127261}{49152000},\\
 d=\frac{64440899573}{884736000},\qquad
 e=\frac{89139416731}{1769472000}.
 \end{gathered}
\]
For $x\ge-1/10$, $T$ is increasing and $T(-1/10)>1700$;
also $|d|<73$, $e>50$. Hence
$-Q_4\ge1700x^2-73|x|+50>0$.

Expand radially about $s=16$:
\[
 P(s,\tau)=P(16,\tau)+(s-16)P_s(16,\tau)
      +(c_1(\tau)+9/5)(s-16)^2+\frac3{80}(s-16)^3.
\]
Since $|s-16|\le4$, \eqref{n8a:e38:eq:c1} bounds the last two
terms below by $-10$. Also $Q_4/s\ge Q_4/12$. Consequently,
\begin{equation}\label{n8a:e38:eq:Z}
 P(s,\tau)+\frac{c_4(\tau)}s\ge\min\{Z_-(\tau),Z_+(\tau)\},
 \qquad Z_\pm=P(16,\tau)\pm4P_s(16,\tau)-10+Q_4/12.
\end{equation}
Both quartics have simple lower bounds. Write
$x=\tau-4/5\in[-2/5,1/5]$ and
\[
 Z_\pm=x^2T_\pm(x)+d_\pm x+e_\pm,
 \qquad T_\pm(x)=a x^2+b_\pm x+c_\pm,
\]
The explicit expansions below give
\[
 23000<a<24000,\quad b_->28000,\quad b_+>26000,
\]
\[
 T_-(-2/5)>600,\quad |d_-|<97,\quad e_->115,
 \qquad
 T_+(-2/5)>1200,\quad |d_+|<512,\quad e_+>146.
\]
On this interval $T_\pm'>0$. Completing squares gives
\[
 Z_->115-\frac{97^2}{2400}>90,\qquad
 Z_+>146-\frac{512^2}{4800}>90.
\]
By \eqref{n8a:e38:eq:Z}, $P+c_4/s>90$.
The global bounds \eqref{n8a:e38:eq:oldbounds} and Cauchy with $\alpha=1$ give
\[
 s^3\Ac>90-\Lambda_9(s),
\]
\[
 \Lambda_9(s)=\frac{1400}{s^2}+\frac{2000}{s^3}
 +\frac{26000}{s^4}+\frac{34000}{s^5}+\frac{40000}{s^6}
 +\frac{250000000}{s^7(1-1/s)}.
\]
This function decreases, and
$\Lambda_9(12)=1701725/85536<20$.
Thus $\Ac>70/s^3$ and $\Rcal>70/(17s^3)>1/(200s^3)$.
This proves $\Rcal(s,q)>1/(200s^3)$ on all three angular regions.

\subsubsection{Auxiliary polynomial identities}
The following identities follow by multiplication or substitution
in the preceding formulas and make the algebra explicit.

For $L_1,L_4$, write $L_j=N_j/(13824000s^4)$, where
\begin{align*}
 N_1={}&518400s^7+40375440s^6-872490780s^5+8518673035s^4\\
 &-81987329739s^3+373598447415s^2-624083869575s+347863275000,\\
 N_4={}&518400s^7+38129040s^6-1614904380s^5+20463633955s^4\\
 &-339802707756s^3+6071551358640s^2-41932023652800s
                                      +89052998400000.
\end{align*}
In particular,
$L_1(12)=3658132863828/(13824000\cdot12^4)>51/4$.
For the bound on $N_4$, the exact expansion is
\begin{align*}
 N_4(16+r)={}&518400r^7+96189840r^6+4832401860r^5
                   +112004621155r^4\\
 &+1148330793364r^3+3944268984432r^2
                   -1738209349568r+7236696853504.
\end{align*}
Substituting $r=-x$ gives the auxiliary polynomial
\begin{align*}
 Q(x)={}&3.944268984432-1.148330793364x+0.112004621155x^2\\
       &-0.004832401860x^3+0.000096189840x^4-0.000000518400x^5.
\end{align*}
These decimals are exact; the two absorptions above prove
$Q\ge0.092x^2-1.15x+3.94$ on $[0,4]$.

The remainders in the products of quintics are
\begin{align*}
 T_1={}&-\frac{206753737}{1843200}+\frac{17519675741}{14745600}u
 -\frac{2917361759}{2457600}u^2-\frac{377854877}{9830400}u^3\\
 &+\frac{120400463}{409600}u^4-\frac{70887171}{26214400}u^5,\\
 T_4={}&\frac{1557109939}{368640}+\frac{39674783893}{2949120}u
 -\frac{80569866503}{491520}u^2+\frac{108648970715}{393216}u^3\\
 &+\frac{2820591425}{16384}u^4-\frac{57729261975}{1048576}u^5.
\end{align*}
Retaining their negative terms and the positive constant term
of $T_4$ proves the bounds used above.

The two quartics in \eqref{n8a:e38:eq:Z}, before translation, are
\begin{align*}
 Z_-={}&\frac{8596115437}{368640}\tau^4
 -\frac{1624788076241}{35389440}\tau^3
 +\frac{4276670657851}{147456000}\tau^2\\
 &-\frac{15698391869303}{2654208000}\tau
 +\frac{159956715989}{663552000},\\
 Z_+={}&\frac{8596115437}{368640}\tau^4
 -\frac{1690859884241}{35389440}\tau^3
 +\frac{4915405813051}{147456000}\tau^2\\
 &-\frac{23477732300663}{2654208000}\tau
 +\frac{530744389589}{663552000}.
\end{align*}
The substitution $\tau=x+4/5$ gives the constants bounded
in the proof. Only monotonicity of $T_\pm$ and completion of
squares are needed.

\subsection{The low radii on two full intervals}\label{n8a:app:low}
\begin{lemma}\label{n8a:e39:thm:low}
$L_2(s)>0$ for $8\le s\le11$, and $L_3(s)>0$ for $9\le s\le11$.
\end{lemma}
The functions are those defined in the exterior reduction.

\subsubsection{Polynomial normalization at low radii}
Write $u=1/s$, $d=cu$, $a=2+d$ and define the quintics
\begin{align*}
 N_c(u)&=\frac d{15360}\bigl(720a^4-4134a^3u+3885d^2a^2
             -3440a^2u^2+7688d^2au-4500d^4\bigr),\\
 H_c(u)&=N_c(u)+\frac{160a^3}{15360}(3a^2-20au+24d^2).
\end{align*}
Put $n=N_1$, $h=H_1$ and, for $q=m^2$,
\[
 J_q=H_{1+2m}H_{1-2m},\qquad
 G_q=2(N_{1+2m}H_{1-2m}+N_{1-2m}H_{1+2m}).
\]
Then $g_s(0)=4n/h$ and $g_s(q)=G_q/J_q$.
With $B_2=hJ_1J_4$, $B_3=hJ_1J_9$, this gives
\begin{align*}
 u^3A_2&=(14u-1)4nJ_1J_4+(1-8u)G_1hJ_4+2uG_4hJ_1-9u^2B_2,\\
 u^3A_3&=(16u-1)4nJ_1J_9+(1-9u)G_1hJ_9+uG_9hJ_1-9u^2B_3.
\end{align*}
The terms of orders $0,1,2$ cancel: $A_2,A_3$ are polynomials
of degree $24$, and
\begin{equation}\label{n8a:e39:eq:lowAB}
 L_j(s)=u^2\frac{A_j(u)}{B_j(u)}\qquad(j=2,3).
\end{equation}

On these intervals $s+t\ge13$, so the inner denominator of
$\Psi$ is positive. The rational denominators are also positive.
For $0<u\le1/8$,
\[
 h\ge1-u/12-4u^2>0,\qquad
 H_{-1}\ge1-7u-14u^3-u^5>0,
\]
\[
 H_{-3}\ge(1-7u)^2-1/500>0.
\]
The quintics $H_3,H_5,H_7$ have positive coefficients.
For $H_{-5}$, needed only on $1/11\le u\le1/9$, take
$y=10u-1$, $|y|\le1/9$. Its expression is
\[
 H_{-5}=\frac{164407}{819200}-\frac{523927}{2457600}y
 +\frac{87531}{409600}y^2-\frac{73941}{409600}y^3
 +\frac{150179}{819200}y^4-\frac{21083}{2457600}y^5.
\]
The triangle inequality gives
$H_{-5}\ge12631924973/72559411200>1/6$.

\paragraph{Taylor expansion at the common endpoint.}
Take $r=11u-1$. Then $0\le r\le3/8$ on the first interval
and $0\le r\le2/9$ on the second. Let $\T_m F$ denote the
Taylor polynomial of $F(r)$ at zero through degree $m$.
Multiplying the preceding quintics gives
\begin{align}
 \T_{10}A_2((1+r)/11)
 &\ge\frac3{40}+\frac r2+r^2-2r^4-2r^5-r^9-r^{10}
 \ge\frac3{40},\label{n8a:e39:eq:T2}\\
 \T_{10}A_3((1+r)/11)
 &\ge\frac1{10}-9r^8-5r^9-r^{10}>\frac9{100}.
 \label{n8a:e39:eq:T3}
\end{align}
In \eqref{n8a:e39:eq:T2}, absorb the negative terms into $r^2$:
$2r^2+2r^3+r^7+r^8<1$ for $0\le r\le3/8$.
For \eqref{n8a:e39:eq:T3}, use monotonicity of the subtracted polynomial.
Both estimates hold on the full intervals.

\paragraph{The full remainder.}
Use the coefficient norm $\nu_U$ of Section~\ref{subsec:generalcauchy}.
On $|r|=4$, we have
$|u|\le U=5/11$ and $|u|\ge\ell=3/11$.
With $n_c=\nu_U(N_c)$, $h_c=\nu_U(H_c)$, put
\[
 j_q=h_{1+2\sqrt q}h_{1-2\sqrt q},\qquad
 \gamma_q=2(n_{1+2\sqrt q}h_{1-2\sqrt q}
                    +n_{1-2\sqrt q}h_{1+2\sqrt q}).
\]
The triangle inequality in the formulas for $u^3A_j$ gives
\begin{align*}
 M_2&=\ell^{-3}\bigl[(14U+1)4n_1j_1j_4
 +(1+8U)\gamma_1h_1j_4+2U\gamma_4h_1j_1+9U^2h_1j_1j_4\bigr]
 <5\cdot10^9,\\
 M_3&=\ell^{-3}\bigl[(16U+1)4n_1j_1j_9
 +(1+9U)\gamma_1h_1j_9+U\gamma_9h_1j_1+9U^2h_1j_1j_9\bigr]
 <8\cdot10^{10}.
\end{align*}
Applying \eqref{eq:generalcauchytail} bounds all omitted orders:
\[
 |A_j-\T_{10}A_j|\le
 M_j\frac{(r/4)^{11}}{1-r/4}.
\]
On the respective intervals, these bounds are less than $1/30$
and $1/500$. Together with \eqref{n8a:e39:eq:T2}--\eqref{n8a:e39:eq:T3}, this proves
\[
 A_2((1+r)/11)>1/25,\qquad A_3((1+r)/11)>1/25.
\]
Positivity of the denominators and \eqref{n8a:e39:eq:lowAB} complete
the proof of Lemma~\ref{n8a:e39:thm:low}.

For \eqref{n8a:e39:eq:T2}, writing $A_2((1+r)/11)=\sum a_jr^j$, we use
$a_0>3/40$, $a_1>1/2$, $a_2>1$, $a_3,a_6,a_7,a_8>0$,
$a_4,a_5>-2$ and $a_9,a_{10}>-1$.
For \eqref{n8a:e39:eq:T3}, we use $a_0>1/10$, $a_1,\ldots,a_7>0$,
$a_8>-9$, $a_9>-5$, $a_{10}>-1$.
These follow by multiplying the quintics $N_c,H_c$; the Cauchy
remainder estimate controls all higher orders.

\subsection{Denominators and norms for the constructed profile}\label{n8a:app:denominators}
We give the norms used in Lemma~\ref{n8a:lem:tails}.
Define $N_B=D_BB$, $N_E=D_EE$ using the exact parameters
of Section~\ref{n8a:sec:profile}.

Choose the interior denominators
\begin{align*}
 D_B(s)={}&20000000s^5(s-1)^5(s+1)^4(s+3)^4
                         (2s+3)(2s+5)(2s+7),\\
 D_E(s)={}&80000000s^5(s-3)^5(s-2)(s-1)^5(s+1)^4(2s+3)
                               K_5(s)K_3(s),\\
 K_5(s)={}&3072s^5+39680s^4+325328s^3+1806700s^2
                                      +5198000s+2723875,\\
 K_3(s)={}&15360s^5+98560s^4+423024s^3+1469748s^2
                                      +2910744s+1011069.
\end{align*}
Their factors are positive on the respective intervals.
For the Cauchy norms, split $B$ into its four summands in
\eqref{eq:generalchainresiduals}.
For $E$, also split the diagonal into $(16-9/s)z(s)$,
$-\beta_8(s,s+5)z(s)$ and $-(s-4)\beta_8(s,s+3)z(s)$,
giving six summands in total.

For each summand $b(s)$, $D_X(s)b(s)$ is a polynomial by the choice
of $D_X$. Apply the factored circle bound \eqref{eq:generalfactorbound}
with the values of $\sigma_X,n_X$ fixed in
Appendix~\ref{n8a:app:interior}. In the order of the four terms of
$B$, the respective bounds on $|t|=2$ are smaller than
\[
 24201,\qquad2605,\qquad7798,\qquad805.
\]
Their sum is $35409<36000$.
For $E$, take the three diagonal terms in the order displayed above,
followed by the three off-diagonal terms in
\eqref{eq:generalchainresiduals}. On $|t|=3$ their bounds are smaller than
\[
 160963394,\quad2792315,\quad27681138,\quad
 36267974,\quad68835090,\quad807760.
\]
Their sum is $297347671<300000000$.
These are bounds for factored polynomials on entire circles.

Here is a direct calculation of the first bound. Put $\sigma=13/2$
and $P(s)=s^4p(s)$. Its normalized polynomial is
\[
 \widehat P(t)=a+\frac b\sigma(1+t)
 -\frac{1.4824723}{\sigma^2}(1+t)^2
 +\frac{9.9151639}{\sigma^3}(1+t)^3
 -\frac{17.4876379}{\sigma^4}(1+t)^4.
\]
Expanding these four powers and using \eqref{n8a:eq:parameters} gives
$\nu_2(\widehat P)<37091/25000$. For the first summand
$b(s)=(16-9/s)p(s)$,
\[
 D_B(s)b(s)=20000000(16s-9)P(s)(s-1)^5(s+1)^4(s+3)^4
                  \prod_{j\in\{3,5,7\}}(2s+j).
\]
Its degree is $21$, so the normalization with $n_B=22$ supplies one
factor $1+t$, bounded by $3$. After cancellation of the factors at
$s=\sigma$, \eqref{eq:generalfactorbound} is less than
\[
 3\left(16+\frac9\sigma\right)\frac{37091}{25000}
 \frac{(\sigma+1)(\sigma+9)^4}{(\sigma-1)^5}
 \prod_{j\in\{3,5,7\}}\frac{2\sigma+3j}{2\sigma+j}
 =3\frac{226}{13}\frac{37091}{25000}
       \frac{15\cdot31^4}{11^5}\frac{1309}{360}<24201.
\]
The remaining bounds follow from the same factor rule applied to the
specified summands; their factors have degree at most five.
In the second chain the degree-$32$ terms cancel, so the sum $F_E$
is polynomial even though a normalized summand may have a pole at
$t=-1$. Cauchy's estimate is applied only to this full polynomial,
as explained in Section~\ref{subsec:generalcauchy}.

\section{The exterior completion in dimension nine}\label{app:gaugecore9}

The quadratic calibration is given in Appendix~\ref{app:calibration9}.
We complete it on the finite set needed in Proposition~\ref{prop:K9}.

The exterior estimate in Proposition~\ref{prop:K9} already covers
$s\ge12$. For $2\le s\le11$, the only types that need modification
are among $(1^j)$, $2\le j\le6$. We verify this through three scalar
remainder estimates. Write, using \eqref{eq:field9},
\[
 g_j(s)=2-\beta_9(s,s+1+2j)-\beta_9(s,s+1-2j),\qquad 0\le j\le3,
\]
\[
 B(s)=9g_0(s)+s\bigl(g_1(s)-g_0(s)\bigr)-\frac{49}{4s},\qquad
 E_j(s)=g_j(s)-g_0(s)-j^2\bigl(g_1(s)-g_0(s)\bigr).
\]
If $m_j$ is the number of coordinates of absolute value $j$, the
full-lattice remainder is exactly $B+m_2E_2+m_3E_3$.
Deleting first-sphere neighbours only increases it. Every non-Boolean
type in this range has either one or two coordinates equal to two,
or one coordinate equal to three. Its remainder is therefore one of
\begin{align*}
 R_1&=B+E_2 &&(4\le s\le11),\\
 R_2&=B+2E_2 &&(8\le s\le11),\\
 R_3&=B+E_3 &&(9\le s\le11).
\end{align*}
Here the coordinates not specified are zeros and ones.

These three one-variable expressions are positive on their required
integer ranges. A short algebraic verification uses
\[
 \begin{gathered}
 Q_a(s)=80s^2+150as+123a^2>0,\qquad D_B=4sQ_1Q_3Q_{-1},\\
 D_2=D_BQ_5Q_{-3},\qquad D_3=D_BQ_7Q_{-5}.
 \end{gathered}
\]
All coefficients of each of the following polynomials in $u$ are
strictly positive:
\begin{equation}\label{eq:ninepositivepolys}
 (D_BB)(7+u),\qquad (D_2R_1)(6+u),\qquad
 (D_2R_2)(8+u),\qquad (D_3R_3)(9+u).
\end{equation}
Their degrees are $5,9,9,9$, respectively; the displayed quadratic
products specify every coefficient by direct expansion.
The two remaining values satisfy $R_1(4)>3/4$ and $R_1(5)>1/5$.
Thus every non-Boolean type has positive remainder, as do the
Boolean types with $s\ge7$.

Let $\mathcal O_j^{(9)}$ be the signed-permutation orbit of $(1^j)$.
Complete only $\Omega_9=\bigcup_{j=2}^{6}\mathcal O_j^{(9)}$.
Its outgoing-cost diagonal is
\[
 c_j=18-\frac{49}{4j}-j\beta_9(j,j+3)
       -\mathbf1_{\{j=6\}}6\beta_9(6,7),\qquad 2\le j\le6,
\]
and direct substitution gives
$(c_2,c_3,c_4,c_5,c_6)>(11,25/2,13,13,8)$ componentwise.
Define
\begin{equation}\label{eq:orbitmatrix9}
 M_9^0=\begin{pmatrix}
 11&-14&0&0&0\\
 -3&25/2&-12&0&0\\
 0&-4&13&-10&0\\
 0&0&-5&13&-8\\
 0&0&0&-6&8
 \end{pmatrix},\qquad \mathbf h^{(9)}=(M_9^0)^{-1}\mathbf1.
\end{equation}
The orbit sizes $2^j\binom9j$ symmetrize this matrix. Its leading
principal minors are
\[
 11,\qquad\frac{191}{2},\qquad\frac{1427}{2},\qquad
 \frac{9001}{2},\qquad1756.
\]
The symmetrized matrix is positive definite with nonpositive
off-diagonal entries. Its inverse has strictly positive entries:
write it as $aI-P$, where $P$ is entrywise nonnegative, irreducible,
and $a$ exceeds its spectral radius, then use the Neumann series.
Hence $\mathbf h^{(9)}>0$.
Assign the ratio $h_{j+1}^{(9)}/h_j^{(9)}$ to an oriented edge from
$\mathcal O_j^{(9)}$ to $\mathcal O_{j+1}^{(9)}$ inside $\Omega_9$,
and retain the base field on all other edges of $Y_9$.
Each such vertex has $j$ neighbours on the preceding Boolean layer
and $2(9-j)$ on the next. Thus its remainder is at least
$(M_9^0\mathbf h^{(9)})_j/h_j^{(9)}=1/h_j^{(9)}>0$.
At $j=2$ the first-sphere neighbours have been deleted while the
full diagonal is retained. All exterior rows keep every incident
field factor and hence their positive remainder. The edge identity,
with these deleted neighbours treated as Dirichlet edges, proves
$K_9\ge M_{W_9}$ with $W_9>0$ everywhere.

\section*{Statements and Declarations}

\noindent\textbf{Competing Interests.}
The author has no relevant financial or non-financial interests to disclose.


\begin{thebibliography}{99}
\raggedright

\bibitem{AlpayThree}
Alpay, N.: The sharp discrete Hardy inequality on $\Z^3$.
Preprint (2026).
\href{https://arxiv.org/abs/2608.25262v1}{arXiv:2608.25262v1}.

\bibitem{AlpayLow}
Alpay, N.: Sharp discrete Hardy constants in low dimensions.
Preprint (2026).
\href{https://arxiv.org/abs/2610.05421v1}{arXiv:2610.05421v1}.

\bibitem{BarasGoldstein}
Baras, P., Goldstein, J.~A.:
The heat equation with a singular potential.
\emph{Trans. Amer. Math. Soc.} \textbf{284}(1), 121--139 (1984).
\url{https://doi.org/10.1090/S0002-9947-1984-0742415-3}

\bibitem{Oberwolfach2025}
Berchio, E., Keller, M., Pinchover, Y., Roncal, L.:
Mini-Workshop: Hardy Inequalities in Discrete and Continuum Settings.
\emph{Oberwolfach Rep.} \textbf{22}(1), 579--614 (2025).
\url{https://doi.org/10.4171/OWR/2025/14}

\bibitem{BogdanHardy}
Bogdan, K., Grzywny, T., Jakubowski, T., Pilarczyk, D.:
Fractional Laplacian with Hardy potential.
\emph{Comm. Partial Differential Equations} \textbf{44}(1), 20--50 (2019).
\url{https://doi.org/10.1080/03605302.2018.1539102}

\bibitem{ChenHuaXu}
Chen, H., Hua, B., Xu, W.: Poisson problem with Hardy operator on lattice graph $\mathbb Z^d$ and the application to semilinear equations.
Preprint (2026).
\href{https://arxiv.org/abs/2609.13933v1}{arXiv:2609.13933v1}.

\bibitem{CiaurriRoncal}
Ciaurri, \'O., Roncal, L.:
Hardy's inequality for the fractional powers of a discrete Laplacian.
\emph{J. Anal.} \textbf{26}(2), 211--225 (2018).
\url{https://doi.org/10.1007/s41478-018-0141-2}

\bibitem{DasFuente}
Das, U., de la Fuente-Fern\'andez, R.:
An optimal fractional Hardy inequality on the discrete half-line.
\emph{Calc. Var. Partial Differential Equations} \textbf{65}, Paper No.~46 (2026).
\url{https://doi.org/10.1007/s00526-025-03217-w}

\bibitem{Delmotte}
Delmotte, T.:
Parabolic Harnack inequality and estimates of Markov chains on graphs.
\emph{Rev. Mat. Iberoam.} \textbf{15}(1), 181--232 (1999).
\url{https://doi.org/10.4171/RMI/254}

\bibitem{Doob}
Doob, J.~L.:
Conditional Brownian motion and the boundary limits of harmonic functions.
\emph{Bull. Soc. Math. France} \textbf{85}, 431--458 (1957).
\url{https://doi.org/10.24033/bsmf.1494}

\bibitem{Dyda}
Dyda, B.:
Fractional and non-fractional Hardy inequality on a lattice $\Z^d$.
\emph{J. Differential Equations} \textbf{485}, article 114749 (2026).
\url{https://doi.org/10.1016/j.jde.2026.114749}

\bibitem{FischerHardy}
Fischer, F.:
On the optimality and decay of $p$-Hardy weights on graphs.
\emph{Calc. Var. Partial Differential Equations} \textbf{63}, Paper No.~162 (2024).
\url{https://doi.org/10.1007/s00526-024-02754-0}

\bibitem{FLS}
Frank, R.~L., Lieb, E.~H., Seiringer, R.:
Hardy--Lieb--Thirring inequalities for fractional Schr\"odinger operators.
\emph{J. Amer. Math. Soc.} \textbf{21}, 925--950 (2008).
\url{https://doi.org/10.1090/S0894-0347-07-00582-6}

\bibitem{GalaktionovKamotski}
Galaktionov, V.~A., Kamotski, I.~V.:
On nonexistence of Baras--Goldstein type for higher-order parabolic
equations with singular potentials.
\emph{Trans. Amer. Math. Soc.} \textbf{362}(8), 4117--4136 (2010).
\url{https://doi.org/10.1090/S0002-9947-10-04855-5}

\bibitem{GKSRellich}
Gerhat, B., Krej\v{c}i\v{r}\'ik, D., \v{S}tampach, F.:
An improved discrete Rellich inequality on the half-line.
\emph{Israel J. Math.} \textbf{268}, 45--72 (2025).
\url{https://doi.org/10.1007/s11856-025-2790-6}

\bibitem{GradshteynRyzhik}
Gradshteyn, I.~S., Ryzhik, I.~M.:
\emph{Table of Integrals, Series, and Products}, 7th ed.
Jeffrey, A., Zwillinger, D. (eds.). Academic Press, Amsterdam (2007).

\bibitem{GuptaLattice}
Gupta, S.:
Hardy and Rellich inequality on lattices.
\emph{Calc. Var. Partial Differential Equations} \textbf{62}, Paper No.~81 (2023).
\url{https://doi.org/10.1007/s00526-022-02407-0}

\bibitem{Gupta1D}
Gupta, S.:
One-dimensional discrete Hardy and Rellich inequalities on integers.
\emph{J. Fourier Anal. Appl.} \textbf{30}, Paper No.~15 (2024).
\url{https://doi.org/10.1007/s00041-024-10070-6}

\bibitem{GuptaAsymptotic}
Gupta, S.: Sharp asymptotics for higher-order Hardy constants on lattices.
Preprint (2026).
\href{https://arxiv.org/abs/2607.15181v1}{arXiv:2607.15181v1}.

\bibitem{HKPFractional}
Hake, P., Keller, M., Pogorzelski, F.: Optimal Hardy inequality for fractional Laplacians on the lattice.
Preprint (2025).
\href{https://arxiv.org/abs/2601.00902v1}{arXiv:2601.00902v1}.

\bibitem{HKPSupersolution}
Hake, P., Keller, M., Pogorzelski, F.: Supersolution construction and optimal Hardy inequality for fractional Laplacians.
Preprint (2026).
\href{https://arxiv.org/abs/2607.24169v1}{arXiv:2607.24169v1}.

\bibitem{Hardy}
Hardy, G.~H.:
Note on a theorem of Hilbert.
\emph{Math. Z.} \textbf{6}, 314--317 (1920).
\url{https://doi.org/10.1007/BF01199965}

\bibitem{HoangHundertmarkRichterVugalter}
Hoang, V., Hundertmark, D., Richter, J., Vugalter, S.:
Quantitative bounds versus existence of weakly coupled bound states for
Schr\"odinger type operators.
\emph{Ann. Henri Poincar\'e} \textbf{24}, 783--842 (2023).
\url{https://doi.org/10.1007/s00023-022-01228-3}

\bibitem{HuangYe}
Huang, X., Ye, D.:
One-dimensional sharp discrete Hardy--Rellich inequalities.
\emph{J. Lond. Math. Soc.} (2) \textbf{109}, e12838 (2024).
\url{https://doi.org/10.1112/jlms.12838}

\bibitem{HuangYeAsymptotic}
Huang, X., Ye, D.: Asymptotically sharp Hardy--Rellich inequalities on lattices.
Preprint (2026).
\href{https://arxiv.org/abs/2607.23057v1}{arXiv:2607.23057v1}.

\bibitem{IgnatZuazua}
Ignat, L.~I., Zuazua, E.:
Sharp numerical approximation of the Hardy constant.
\emph{Discrete Contin. Dyn. Syst.} \textbf{54}, 15--32 (2026).
\url{https://doi.org/10.3934/dcds.2025165}

\bibitem{IKO}
Ishige, K., Kabeya, Y., Ouhabaz, E.~M.:
The heat kernel of a Schr\"odinger operator with inverse square potential.
\emph{Proc. Lond. Math. Soc.} (3) \textbf{115}, 381--410 (2017).
\url{https://doi.org/10.1112/plms.12041}

\bibitem{ItoJensen}
Ito, K., Jensen, A.:
Branching form of the resolvent at thresholds for multi-dimensional
discrete Laplacians.
\emph{J. Funct. Anal.} \textbf{277}, 965--993 (2019).
\url{https://doi.org/10.1016/j.jfa.2019.05.018}

\bibitem{JensenNenciu}
Jensen, A., Nenciu, G.:
A unified approach to resolvent expansions at thresholds.
\emph{Rev. Math. Phys.} \textbf{13}, 717--754 (2001).
\url{https://doi.org/10.1142/S0129055X01000843}

\bibitem{JensenNenciuAddendum}
Jensen, A., Nenciu, G.:
Erratum to the paper ``A unified approach to resolvent expansions at thresholds''.
\emph{Rev. Math. Phys.} \textbf{16}, 675--677 (2004).
\url{https://doi.org/10.1142/S0129055X04002102}

\bibitem{JexStampach}
Jex, M., \v{S}tampach, F.:
On the ground state of lattice Schr\"odinger operators.
\emph{J. Spectr. Theory} \textbf{15}, 647--678 (2025).
\url{https://doi.org/10.4171/JST/558}

\bibitem{JiaZhao}
Jia, X., Zhao, Y.:
Coupling constant limits of Schr\"odinger operators with critical potentials.
\emph{Boundary Value Problems} \textbf{2013}, article 62, 15 pp. (2013).
\url{https://doi.org/10.1186/1687-2770-2013-62}

\bibitem{KapitanskiLaptev}
Kapitanski, L., Laptev, A.:
On continuous and discrete Hardy inequalities.
\emph{J. Spectr. Theory} \textbf{6}(4), 837--858 (2016).
\url{https://doi.org/10.4171/JST/144}

\bibitem{KellerLemm}
Keller, M., Lemm, M.:
Optimal Hardy weights on the Euclidean lattice.
\emph{Trans. Amer. Math. Soc.} \textbf{376}(9), 6033--6062 (2023).
\url{https://doi.org/10.1090/tran/8939}

\bibitem{KLW}
Keller, M., Lenz, D., Wojciechowski, R.~K.:
\emph{Graphs and Discrete Dirichlet Spaces}.
Grundlehren der mathematischen Wissenschaften, vol.~358.
Springer, Cham (2021).
\url{https://doi.org/10.1007/978-3-030-81459-5}

\bibitem{KPPMonthly}
Keller, M., Pinchover, Y., Pogorzelski, F.:
An improved discrete Hardy inequality.
\emph{Amer. Math. Monthly} \textbf{125}, 347--350 (2018).
\url{https://doi.org/10.1080/00029890.2018.1420995}

\bibitem{KPPhardy}
Keller, M., Pinchover, Y., Pogorzelski, F.:
Optimal Hardy inequalities for Schr\"odinger operators on graphs.
\emph{Comm. Math. Phys.} \textbf{358}, 767--790 (2018).
\url{https://doi.org/10.1007/s00220-018-3107-y}

\bibitem{KPP}
Keller, M., Pinchover, Y., Pogorzelski, F.:
Criticality theory for Schr\"odinger operators on graphs.
\emph{J. Spectr. Theory} \textbf{10}, 73--114 (2020).
\url{https://doi.org/10.4171/JST/286}

\bibitem{KholmatovLakaevAlmuratov}
Kholmatov, Sh.~Yu., Lakaev, S.~N., Almuratov, F.~M.:
Bound states of Schr\"odinger-type operators on one and two dimensional
lattices.
\emph{J. Math. Anal. Appl.} \textbf{503}(1), article 125280 (2021).
\url{https://doi.org/10.1016/j.jmaa.2021.125280}

\bibitem{KondejLotoreichik}
Kondej, S., Lotoreichik, V.:
Weakly coupled bound state of 2-D Schr\"odinger operator with potential-measure.
\emph{J. Math. Anal. Appl.} \textbf{420}, 1416--1438 (2014).
\url{https://doi.org/10.1016/j.jmaa.2014.06.053}

\bibitem{KrejcirikStampach}
Krej\v{c}i\v{r}\'ik, D., \v{S}tampach, F.:
A sharp form of the discrete Hardy inequality and the
Keller--Pinchover--Pogorzelski inequality.
\emph{Amer. Math. Monthly} \textbf{129}, 281--283 (2022).
\url{https://doi.org/10.1080/00029890.2022.2011569}

\bibitem{LizamaSharp}
Lizama, C.: Sharp discrete Hardy constants in dimensions three and four and strict upper bounds from dimension nine.
Preprint (2026).
\href{https://arxiv.org/abs/2608.30831v1}{arXiv:2608.30831v1}.

\bibitem{LizamaAsymptotic}
Lizama, C.: Three-term asymptotics for discrete Hardy--Rellich constants in high dimension.
Preprint (2026).
\href{https://arxiv.org/abs/2608.30888v1}{arXiv:2608.30888v1}.

\bibitem{DLMF}
National Institute of Standards and Technology:
\emph{NIST Digital Library of Mathematical Functions}.
Version 1.2.8 (2026).
\url{https://dlmf.nist.gov/}. 

\bibitem{PinchoverTintarev}
Pinchover, Y., Tintarev, K.:
A ground state alternative for singular Schr\"odinger operators.
\emph{J. Funct. Anal.} \textbf{230}, 65--77 (2006).
\url{https://doi.org/10.1016/j.jfa.2005.05.015}

\bibitem{Polya}
P\'olya, G.:
\"Uber eine Aufgabe der Wahrscheinlichkeitsrechnung betreffend die
Irrfahrt im Stra\ss ennetz.
\emph{Math. Ann.} \textbf{84}, 149--160 (1921).
\url{https://doi.org/10.1007/BF01458701}

\bibitem{Simon}
Simon, B.:
The bound state of weakly coupled Schr\"odinger operators in one and
two dimensions.
\emph{Ann. Physics} \textbf{97}, 279--288 (1976).
\url{https://doi.org/10.1016/0003-4916(76)90038-5}

\bibitem{WangConical}
Wang, X.~P.:
Asymptotic expansion in time of the Schr\"odinger group on conical manifolds.
\emph{Ann. Inst. Fourier} \textbf{56}, 1903--1945 (2006).
\url{https://doi.org/10.5802/aif.2230}

\bibitem{WangCorrigendum}
Wang, X.~P.:
Corrigendum to: Asymptotic expansion in time of the Schr\"odinger group
on conical manifolds.
\emph{Ann. Inst. Fourier} \textbf{57}, 2081--2082 (2007).
\url{https://doi.org/10.5802/aif.2325}

\end{thebibliography}
\end{document}